\documentclass[aop,MSNbibl,dvips]{arximspdf}
\usepackage{graphicx}
\doi{10.1214/13-AOP868} 
\volume{42}
\issue{6}
\pubyear{2014}
\firstpage{2314}
\lastpage{2382}

\makeatletter
\newcommand{\rrvert}{\vert}
\newcommand{\llvert}{\vert}
\newtheorem{theorem}{Theorem}[section]
\newtheorem{lemma}[theorem]{Lemma}
\newtheorem{proposition}[theorem]{Proposition}
\newtheorem{corollary}[theorem]{Corollary}
\newproclaim{remark}[theorem]{Remark}
\newproclaim{definition}[theorem]{Definition}
\makeatother

\begin{document}
\begin{frontmatter}

\title{From duality to determinants for \lowercase{\textit{q}}-TASEP and ASEP}
\runtitle{From duality to determinants for \lowercase{\textit{q}}-TASEP and ASEP}

\begin{aug}
\author[A]{\fnms{Alexei} \snm{Borodin}\ead[label=e1]{borodin@math.mit.edu}\thanksref{T1}},
\author[B]{\fnms{Ivan} \snm{Corwin}\corref{}\ead[label=e2]{ivan.corwin@gmail.com}\thanksref{T2}}
\and
\author[C]{\fnms{Tomohiro} \snm{Sasamoto}\ead[label=e3]{sasamoto@math.s.chiba-u.ac.jp}\thanksref{T3}}
\runauthor{A. Borodin, I. Corwin and T. Sasamoto}
\affiliation{Massachusetts Institute of Technology and Institute for Information Transmission Problems,
Columbia University, Clay Mathematics\\ Institute and Massachusetts Institute of Technology,
and Chiba University}
\address[A]{A. Borodin\\
Department of Mathematics\\
Massachusetts Institute of Technology\\
77 Massachusetts Avenue\\
Cambridge, Massachusetts 02139-4307\\
USA\\
and\\
Institute for Information\\
\quad Transmission Problems\\
Bolshoy Karetny per. 19\\
Moscow 127994\\
Russia\\
\printead{e1}} 
\address[B]{I. Corwin\\
Department of Mathematics\\
Columbia University\\
2990 Broadway\\
New York, New York 10027\\
USA\\
and\\
Clay Mathematics Institute\\
10 Memorial Blvd. Suite 902\\
Providence, Rhode Island 02903\\
USA\\
and\\
Department of Mathematics\\
Massachusetts Institute of Technology\\
77 Massachusetts Avenue\\
Cambridge, Massachusetts 02139-4307\\
USA\\
\printead{e2}}
\address[C]{T. Sasamoto\\
Department of Mathematics\\
Chiba University\\
1-33 Yayoi-cho, Inage, Chiba, 263-8522\\
Japan\\
\printead{e3}}
\end{aug}
\thankstext{T1}{Supported in part by NSF Grant DMS-10-56390.}
\thankstext{T2}{Supported in part by NSF PIRE Grant OISE-07-30136 and DMS-12-08998 as well as by
Microsoft Research through the Schramm Memorial Fellowship, and by the Clay Mathematics Institute.}
\thankstext{T3}{Supported by KAKENHI (22740054).}

\received{\smonth{8} \syear{2012}}
\revised{\smonth{3} \syear{2013}}

%
\begin{abstract}
We prove duality relations for two interacting particle systems: the
\mbox{$q$-}deformed totally asymmetric simple exclusion process (\mbox{$q$-}TASEP)
and the asymmetric simple exclusion process (ASEP). Expectations of the
duality functionals correspond to certain joint moments of particle
locations or integrated currents, respectively. Duality implies that
they solve systems of ODEs. These systems are integrable and for
particular step and half-stationary initial data we use a nested
contour integral ansatz to provide explicit formulas for the systems'
solutions, and hence also the moments.

We form Laplace transform-like generating functions of these moments
and via residue calculus we compute two different types of Fredholm
determinant formulas for such generating functions. For ASEP, the first
type of formula is new and readily lends itself to asymptotic analysis
(as necessary to reprove GUE Tracy--Widom distribution fluctuations for
ASEP), while the second type of formula is recognizable as closely
related to Tracy and Widom's ASEP formula
[\textit{Comm. Math. Phys.} \textbf{279} (2008) 815--844,
\textit{J. Stat. Phys.} \textbf{132} (2008) 291--300,
\textit{Comm. Math. Phys.} \textbf{290} (2009) 129--154,
\textit{J. Stat. Phys.} \textbf{140} (2010) 619--634].
For \mbox{$q$-}TASEP, both formulas coincide with those computed via Borodin and
Corwin's Macdonald processes [\textit{Probab. Theory Related Fields} (2014) \textbf{158}
225--400].

Both \mbox{$q$-}TASEP and ASEP have limit transitions to the free energy of
the continuum directed polymer, the logarithm of the solution of the
stochastic heat equation or the Hopf--Cole solution to the
Kardar--Parisi--Zhang equation. Thus, \mbox{$q$-}TASEP and ASEP are integrable
discretizations of these continuum objects; the systems of ODEs
associated to their dualities are deformed discrete quantum delta Bose
gases; and the procedure through which we pass from expectations of
their duality functionals to characterizing generating functions is a
rigorous version of the replica trick in physics.
\end{abstract}

%

\end{frontmatter}

\mbox{}

\tableofcontents[level=1]

\setcounter{footnote}{3}
\section{Introduction}\label{sec1}
One-dimensional driven diffusive systems play an important role in both
physics and mathematics (see, e.g., \cite{Lig,Spohn,CorwinReview}). As
physical models they are used to study mass transport, traffic flow,
queueing behavior, driven lattice gases, and turbulence. Their
integrated current defines height functions which model one-dimensional
interface growth. In certain cases, they can be mapped into models for
directed polymers in random media and propagation of mass in a
disordered environment. The particle systems provide efficient means to
implement simulations of these various types of systems and, in some
rare cases, yield themselves to exact and rigorous mathematical analysis.

This article is concerned with two interacting particle systems---\mbox{$q$-}TASEP with general particle jump rate parameters, and ASEP with
general bond jump rate parameters---which contain rich mathematical
structure. Presently, we seek to shed light on structure which exists
in parallel for both of these systems. We demonstrate duality relations
(see Definition~\ref{dualdef}) for both of these systems directly from
their Markovian dynamics: \mbox{$q$-}TASEP is dual to a totally asymmetric
zero range process TAZRP (Theorem~\ref{thmqtasepduality}) whereas
ASEP is self-dual (Theorems~\ref{thmASEPdualitytilde} and~\ref{thmASEPduality}). A consequence of duality is that expectations of a
large class of natural observables of these systems evolve according to
systems of ODEs.

For \mbox{$q$-}TASEP, the duality result is, to our knowledge, new. When all
particle jump rate parameters are equal, dynamics of \mbox{$q$-}TASEP can be
encoded via a quantum integrable system in terms of $q$-Bosons \cite
{BogIzerKit,SasWad}. For ASEP with all bond jump rate parameters equal,
the ASEP self-duality was observed by Sch\"{u}tz \cite{Schutz} (see
Remark~\ref{4.4}) via a spin chain representation of ASEP (which is
related to the XXZ~model---a well-studied quantum integrable system).
Our results apply for general rates and proceed directly via the Markov
dynamics.

The most surprising observation of this article is that, for certain
initial data called \textit{step} and \textit{half stationary}, we are able
to explicitly solve the systems of ODEs for \mbox{$q$-}TASEP and ASEP in terms
of simple nested-contour integrals. For \mbox{$q$-}TASEP, this works for the
full generality of particle jump rate parameters, whereas for ASEP we
must assume all bond jump rate parameters to be equal at this stage and
henceforth. For \mbox{$q$-}TASEP, the integral representations of the solution
can also be obtained via the formalism of Macdonald processes \cite
{BorCor,BorCorGorShak}, while for ASEP we were guided by analogy and
results of \cite{IS}.

Let us state the simplest versions of these formulas, focusing just on
step initial data in which initially half of the lattice is entirely
empty and the other half entirely full (see Definitions~\ref
{qTASEPstepdef} and~\ref{somedefs}). We also informally introduce the
dynamics of~\mbox{$q$-}TASEP and ASEP.

The \mbox{$q$-}TASEP is a continuous time Markov process $\vec{x}(t)$.
Particles occupy sites of $\mathbb{Z}$ and the location of particle
$i$ at
time $t$ is written as $x_i(t)$ and particles are ordered so that
$x_i(t)>x_j(t)$ for $i<j$. The rate at which the value of $x_i(t)$
increase by one (i.e., the particle jumps right by one) is $a
_i(1-q^{x_{i-1}(t)-x_i(t)-1})$; all jumps occur independently of each
other according to exponential clocks. Here, $q\in[0,1)$ represents
the strength of the repulsion particle $x_i$ feels from particle
$x_{i-1}$. For the purpose of this introduction, we restrict to
$a_i\equiv1$ and consider only step initial data where
particles start at every negative integer location and nowhere else
[i.e., for $i\geq1$, $x_i(0) = -i$]. The following result appears as
Corollary~\ref{cor:qformulas}.
%
\begin{theorem}\label{thm11}
Consider \mbox{$q$-}TASEP with step initial data and particle jump rate
parameters $a_i\equiv1$. Then for any $k\geq1$ and $n_1\geq
n_2\geq\cdots\geq n_k> 0$,
\begin{eqnarray*}
\hspace*{-4.5pt} && \mathbb{E} \Biggl[\prod_{j=1}^{k}
q^{x_{n_j}(t)+n_j} \Biggr]
\\
\hspace*{-4.5pt} &&\!\qquad = \frac{(-1)^k
q^{k(k-1)/2}}{(2\pi\iota)^k} \int\cdots\int\prod
_{1\leq
A<B\leq k} \frac{z_A-z_B}{z_A-qz_B} \prod_{j=1}^{k}
(1-z_j)^{-n_j} e^{(q-1)tz_j} \frac{dz_j}{z_j},
\end{eqnarray*}
where the integration contour for $z_A$ contains $\{qz_B\}_{B>A}$ and 1
but not 0.
\end{theorem}

The ASEP (occupation process) is a continuous time Markov process $\eta
(t)=\{\eta_x(t)\}_{x\in\mathbb{Z}}$. The $\eta_x(t)$ are called occupation
variables and are 1 or 0 based on whether there is a particle or hole
at $x$ at time $t$. The dynamics of this process is specified by
nonnegative real numbers $p\leq q$ (normalized by $p+q=1$) and
uniformly bounded (from infinity and zero) rate parameters $\{a_{x}\}
_{x\in\mathbb{Z}}$. For each pair of neighboring sites $(y,y+1)$, the
following exchanges happen in continuous time:
\begin{eqnarray*}
\eta&\mapsto&\eta^{y,y+1}\qquad\mbox{at rate } a_{y}p\qquad\mbox{if }(\eta_y,\eta_{y+1})=(1,0),
\\
\eta&\mapsto&\eta^{y,y+1}\qquad\mbox{at rate } a_{y}q\qquad\mbox{if }(\eta_y,\eta_{y+1})=(0,1),
\end{eqnarray*}
where $\eta^{y,y+1}$ denotes the state in which the value of the
occupation variables at site $y$ and $y+1$ are switched, and all other
variables remain unchanged. All exchanges occur independently of each
other according to exponential clocks. For the purpose of this
introduction, we restrict to $a_x\equiv1$ and consider only step
initial data\footnote{Observe that the step initial data for \mbox{$q$-}TASEP
involves particles to the left of the origin, whereas for ASEP it
involves particles to the right of the origin. We decided to keep these
conventions to be consistent with previous works on the subject.} where
$\eta_x(0) = \mathbf{1}_{x\geq1}$. Assume $0<p<q$ and let $\tau
=p/q<1$. Finally, let $N_x(t) = \sum_{y\leq x} \eta_y(t)$ record the
number of particles to the left of position $x+1$ at time $t$.

The following result on ASEP appears as Theorem~\ref{qmomInt}.
%
\begin{theorem}\label{thm12}
Consider ASEP with step initial data and all bond rate parameters
$a_{x}\equiv1$. Then for all $n\geq1$ and $x\in\mathbb{Z}$,
\begin{eqnarray*}
\mathbb{E} \bigl[ \tau^{n N_x(t)} \bigr] &=& \tau^{n(n-1)/2}
\frac{1}{(2\pi\iota)^n}
\\
&&{}\times  \int\cdots\int \prod_{1\leq A<B\leq n}
\frac{z_A-z_B}{z_A-\tau z_B}
\\
&&\hspace*{50pt}{}\times \prod_{i=1}^n \exp
\biggl[ -\frac{z_i(p-q)^2}{(z_i+1)(p+qz_i)}t \biggr] \biggl(\frac{1+z_i}{1+z_i/\tau}
\biggr)^{x} \frac{dz_i}{z_i},
\end{eqnarray*}
where the integration contour for $z_A$ includes $0,-\tau$ but does not
include $-1$, or $\{\tau z_B\}_{B>A}$ (see Figure~\ref{ASEPnestedcontours} for an illustration of such contours).
\end{theorem}

These expectations contain sufficient information to uniquely
characterize the distribution of the location of a given collection of
particles (after the system has evolved for some time) in each of these
systems. Focusing on a single $x_n(t)$ or $N_x(t)$ distribution, we can
concisely characterize this via generating functions of suitable
expectations. There are two types of generating functions we consider---both related to \mbox{$q$-}deformed (or for ASEP $\tau$-deformed) Laplace
transform introduced by Hahn \cite{Hahn} in 1949.

These generating functions are naturally suggested from the nested
structure of the contour integral formulas for these expectations.
There are two ways to deform the nested contour integrals so all
contours coincide. Accounting for the residues encountered during these
deformations, we are led to two types of formulas for expectations:
those involving partition-indexed sums of contour integrals and those
involving sums of contour integrals indexed by natural numbers.

Using the partition-indexed formulas, we prove that the first
generating function is equal to a Fredholm determinant which we call
\textit{Mellin--Barnes} type. The following result is contained in Theorem
\ref{ASEPMellinBarnesThm}.

%
\begin{theorem}\label{thm13}
Consider ASEP with step initial data and all bond rate parameters
$a_{x}\equiv1$. Then for all $x\in\mathbb{Z}$ and $\zeta\in\mathbb
{C}\setminus
\mathbb{R}_{+}$,
\[
\mathbb{E} \biggl[\frac{1}{(\zeta\tau^{N_x(t)};\tau)_{\infty
}} \biggr] = \det\bigl(I+K^{\mathrm{ASEP}}_{\zeta}
\bigr),
\]
where $(a;\tau)_{\infty} = (1-a)(1-\tau a)\cdots,$ and where the
$L^2$ space on which $K^{\mathrm{ASEP}}_{\zeta}$ acts can be found in the
statement of Theorem~\ref{ASEPMellinBarnesThm}. The operator $K_{\zeta
}$ is defined in terms of its integral kernel
\[
K^{\mathrm{ASEP}}_{\zeta}\bigl(w,w'\bigr) =
\frac{1}{2\pi\iota}\int_{D_{R,d}} \Gamma (-s)\Gamma(1+s) (-
\zeta)^s \frac{f_{{w}}({x},{t})}{f_{{\tau^s
w}}({x},{t})} \frac{1}{w'-\tau^s w}\,ds.
\]
(The contour $D_{R,d}$ is specified in the statement of Theorem~\ref
{ASEPMellinBarnesThm}.) The function $f_{{z}}({x},{t})$ is given by
\[
f_{{z}}({x},{t})= \exp \biggl[(q-p)t \frac{\tau}{z+\tau} \biggr]
\biggl(\frac{\tau}{z+\tau} \biggr)^x.
\]
\end{theorem}

This type of formula lends itself to rigorous asymptotic analysis. For
ASEP, this formula is new and in Appendix~\ref{GUEasym} we sketch how
it can be used to recover Tracy and Widom's celebrated fluctuation
result \cite{TW3} which states that
%
%
\begin{equation}
\label{ASEPlimitthm} \lim_{t\to\infty} \mathbb{P} \biggl(\frac{N_0(t/\gamma) - (t/4)}{t^{1/3}}
\geq-r \biggr) = F_{\mathrm{GUE}}\bigl(2^{4/3} r\bigr).
\end{equation}
Here, $\gamma=q-p$ is assumed to be strictly positive and $F_{\mathrm{GUE}}$ is the GUE Tracy--Widom distribution. The case when $\gamma=1$
($q=1$ and $p=0$) was proved earlier by Johansson \cite{KJ}. Theorem
\ref{thm13} also allows to access (under a certain weakly asymmetric
scaling) the narrow wedge KPZ equation one point formula \cite{ACQ,SaSp1}.

For \mbox{$q$-}TASEP, such a Mellin--Barnes type formula was obtained from the
theory of Macdonald processes \cite{BorCor}, Theorem 4.1.40. It should
be possible to use this Fredholm determinant to prove cube-root GUE
Tracy--Widom fluctuations for the current past the origin in \mbox{$q$-}TASEP.
This has not yet been done, though in an stationary version of the
TAZRP associated to \mbox{$q$-}TASEP gaps, the cube-root fluctuation scale is
shown in \cite{BKS} (via a different approach). In \cite
{BorCor,BorCorFer}, the \mbox{$q$-}TASEP Mellin--Barnes-type Fredholm
determinant formula is used (via a limit transition) to write the
Laplace transform of the O'Connell--Yor semidiscrete polymer partition
function \cite{OY}. Then \cite{BorCor,BorCorFer} perform rigorous
asymptotic analysis to show cube-root GUE Tracy--Widom free energy
fluctuations as well as to provide a second rigorous derivation of the
narrow wedge KPZ equation formula (first rigorously derived in \cite
{ACQ}). From the perspective of asymptotics, this second approach is a
little less involved than that of \cite{ACQ}.

On the other hand, using the formulas of the second type (i.e.,
deforming contours in Theorems~\ref{thm11} and~\ref{thm12}
differently), we prove that the second generating function is equal to
a Fredholm determinant which we call \textit{Cauchy} type. The ASEP
Fredholm determinant Tracy and Widom derived in \cite{TW1,TW2} is also
of this type (and in fact, after inverting the $e_{\tau}$-Laplace
transform we recover the same formula as in \cite{TW1,TW2}).
Asymptotic analysis of this type of determinant is not as
straightforward as the Mellin--Barnes type. In \cite{TW3}, Tracy and
Widom employ a significant amount of post-processing to turn this type
of formula into one for which they could perform asymptotic analysis.
The final formula still involves a complicated term related to the
Ramanujan summation formula (as observed in \cite{SaSp1}). One should
note that while we do recover (among other formulas) the Tracy--Widom
ASEP Fredholm determinant formula, our approach via duality is entirely
different, our contour integral ansatz is not a version of the
coordinate Bethe ansatz and, along the way, we gain access to other
information about ASEP, like joint moment formulas. The Cauchy-type
Fredholm determinant formula for \mbox{$q$-}TASEP was also first derived in
\cite{BorCor} via Macdonald processes.

In short, by utilizing duality for \mbox{$q$-}TASEP and ASEP, we are able to
provide a short and direct route from Markov dynamics to Fredholm
determinant formulas characterizing single particle location or single
integrated current distributions.

Both \mbox{$q$-}TASEP and ASEP are integrable discretizations of the KPZ
equation. As stochastic processes, they converge to the Hopf--Cole
solution to the KPZ equation \cite{BG,ACQ,QRMF}. The systems of ODEs
associated with their duality appear (though no exact results to this
effect have yet been proved) to have limit transitions to the
attractive quantum delta Bose gas which describes the evolution of
joint moments of the stochastic heat equation (whose logarithm is the
KPZ equation and which describes the partition function for the
continuum random polymer).

An advanced version of the popular physics polymer replica trick
attempts to recover the Laplace transform of the one point distribution
of the solution to the stochastic heat equation in terms of its moments
(see Section~\ref{secreptrick}). However, the moments grow far too
quickly to characterize this distribution, and hence drawing
conclusions from them is mathematically unjustifiable and in any case,
risky. Nevertheless, Dotsenko \cite{Dot} and Calabrese, Le Doussal and
Rosso \cite{CDR} were eventually able to use this trick to recover the
exact formulas of \cite{ACQ,SaSp1}.

It was then natural to consider a discrete analog of this replica
approach. The fact that duality gives a useful tool for computing the
moments for ASEP was first noted in \cite{IS}.
By combining this observation with some of the calculational techniques
developed in \cite{BorCor}, in the present paper we provide a unified
and complete scheme to study both \mbox{$q$-}TASEP and ASEP. Given the results
of our work, the nonrigorous replica trick manipulations can be seen as
shadows of the rigorous duality to determinant approach developed
presently. That is to say that by going to a suitable discrete
approximation we are able to rigorously recover analogs of Laplace
transforms from moments and then in the limit transition these converge
to formulas for the stochastic heat equation's Laplace transform. The
replica trick has proved computationally useful (see, e.g.,
\cite{ImSaKPZ}), thus providing additional motivation for the present work.

The limit transition of \mbox{$q$-}TASEP to the O'Connell--Yor semidiscrete
directed polymer \cite{OY} (and associated semidiscrete stochastic
heat equation) is explored in Appendix~\ref{seclimtran}. Under that
limit transition, duality becomes the replica approach and the duality
system of ODEs become a semidiscrete version of the delta Bose gas. The
nested contour integral ansatz provides means to succinctly compute the
solution to the Bose gas. The Fredholm determinants for \mbox{$q$-}TASEP limit
to Fredholm determinants for the Laplace transform of the polymer
partition function.


While a variety of probabilistic systems arise as degenerations of
Macdonald processes, ASEP is not known to be one of them. For ASEP, it
is not known what, if anything, replaces this additional integrable
structure endowed to \mbox{$q$-}TASEP from its connection to symmetric
functions. However, it is compelling that both \mbox{$q$-}TASEP and ASEP have
duality relations and that the associated systems of ODEs can both be
solved via a nested contour integral ansatz. This leads one to ask
whether \mbox{$q$-}TASEP and ASEP can be unified via a theory even higher than
Macdonald processes. Spohn \cite{SpohnStochasticIntegrability} has
coined the term \textit{stochastic integrability} to describe stochastic
processes which display a great deal of integrable structure. Perhaps,
so as to avoid confusion with stochastic integrals, a more appropriate
name for the present area of study is \textit{integrable sotchastich
particle systems}. Both \mbox{$q$-}TASEP and ASEP are clear examples of such
systems and the contributions of this work provide an additional layer
to that integrability. An upcoming work \cite{BCdiscrete} introduces
two discrete time variants of \mbox{$q$-}TASEP and shows how the methods and
ideas of the present paper extend to the study of these systems as well.

\subsection{Outline}
The paper is organized as follows. In Section~\ref{sec2}, we prove
duality for \mbox{$q$-}TASEP and explicitly solve the associated systems of
ODEs via a nested contour integral ansatz. In Section~\ref{sec3}, we
provide a general scheme to go from such nested contour integral
formulas to two types of Fredholm determinants and in Section~\ref
{secqtasepapp} we apply this to \mbox{$q$-}TASEP in order to prove Theorem
\ref{thm11}. In Section~\ref{sec4}, we prove duality and nested
contour integral formulas for ASEP. In Section~\ref{sec5}, we explain
the passage from Theorems~\ref{thm12}--\ref{thm13}.
Appendix~\ref{seclimtran} deals with a degeneration of \mbox{$q$-}TASEP to a
semidiscrete directed polymer. Appendix~\ref{sec7} collects necessary
combinatorial facts. Appendix~\ref{appenduniq} proves a uniqueness
result for the system of ODEs associated with ASEP duality.
Appendix~\ref{GUEasym} provides critical point analysis of the Fredholm
determinant in Theorem~\ref{thm13}, as necessary to obtain (\ref
{ASEPlimitthm}).

\subsection{Notations}
We fix a few notations used throughout this paper. The imaginary unit
$\iota= \sqrt{-1}$. The indicator function of an event $E$ is denoted
by either $\delta_{E}$ or $\mathbf{1}_{E}$. We write $a_x\equiv1$ if
$a_x=1$ for all $x$. All contours we consider are simple, smooth,
closed and counterclockwise oriented (unless otherwise specified). For
a contour $C$, we write $\alpha C$ as the dilation of $C$ by a factor
of $\alpha>0$. When we write that the integration contour for $z_A$
contains $\{qz_B\}_{B>A}$, we mean that the contour contains the image
of the $z_B$ contour dilated by $q$. Containment is strict so that if
$C$ contains a point $\alpha$, then $C$ separates $\alpha$ from
infinity and the distance from $C$ to~$\alpha$ is strictly positive.

\section{Duality and the nested contour integral ansatz for $q$-TASEP}\label{sec2}
The \mbox{$q$-}deformed totally asymmetric simple exclusion process
(\mbox{$q$-}TASEP) is a continuous time, discrete space interacting particle
system $\vec{x}(t)$. Particles occupy sites of $\mathbb{Z}$ and the location
of particle $i$ at time $t$ is written as $x_i(t)$ and particles are
ordered so that $x_i(t)>x_j(t)$ for $i<j$. The rate at which the value
of $x_i(t)$ increase by one (i.e., the particle jumps right by one) is
$a_i(1-q^{x_{i-1}(t)-x_i(t)-1})$; all jumps occur independently
of each other according to exponential clocks. Here, $q\in[0,1)$,
$a_i>0$ is particle $i$'s jump rate parameter,
$x_{i-1}(t)-x_i(t)-1$ is the number of empty sites to its right (before
particle $x_{i-1}$) and all jumps occur independently of each other
(see left-hand side of Figure~\ref{qtasepfig}). We
will use $\mathbb{E}^x$ and $\mathbb{P}^x$ to denote expectation and
probability
(resp.) of the Markov dynamics with initial data $x$. When the
initial data is itself random, we write $\mathbb{E}$~and~$\mathbb{P}$
to denote
expectation and probability (resp.) of the Markov dynamics as
well as the initial data. We also use $\mathbb{E}$ and $\mathbb{P}$
when the
initial data is otherwise specified.

%
\begin{figure}

\includegraphics{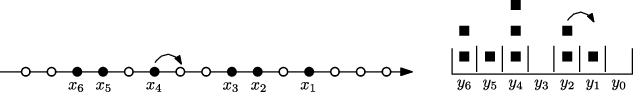}

\caption{Left: \mbox{$q$-}TASEP with six particles. The indicated jump of
$x_4$ occurs at rate $a_4 (1-q^2)$ since the gap $x_{3}-x_{4}-1=2$.
Right: The dual TAZRP with sites $\{0,1,\ldots, 6\}$. The indicated
jump occurs at rate $a_2 (1-q^{2})$ since $y_2=2$.}\label{qtasepfig}
\end{figure}

We presently focus on \mbox{$q$-}TASEP with $N$ particles $x_1>x_2>\cdots
>x_N$. However, to ease the statement of results we include a virtual
particle $x_0(t)\equiv\infty$ and define our state space as
\[
X^N= \bigl\{\vec{x} = (x_0,x_1,
\ldots,x_N)\in\{\infty\}\times \mathbb{Z}^N\dvtx  \infty=
x_0>x_1>\cdots>x_N \bigr\}.
\]
In this case, the dynamics are easily seen to be well defined. Observe
that the evolution of the right-most $M\leq N$ particles performs
\mbox{$q$-}TASEP with $M$ particles (i.e., particles are unaffected by those
to their left). On account of this, it is easy to extend the dynamics
to an infinite number of particles labeled $x_1>x_2>\cdots$ (i.e.,
there is a right-most particle). When studying these infinite systems,
it is generally enough to study related finite systems.

For \mbox{$q$-}TASEP with $N$ particles, the generator of $\vec{x}(t)$ acts
on functions $f\dvtx X^N\to\mathbb{R}$ and is given by
%
%
\begin{equation}
\label{eqnqTASEPgen} \bigl(L^{q\mbox{-}\mathrm{TASEP}}f\bigr) (\vec{x}) = \sum
_{i=1}^{N} a _i\bigl(1-q^{x_{i-1}-x_{i}-1}
\bigr) \bigl(f\bigl(\vec{x}_{i}^{+}\bigr)-f(\vec{x}) \bigr),
\end{equation}
where $\vec{x}_{i}^{+}$ indicates to increase the value of $x_i$ by
one. Note that one may also write down a generator in terms of
occupation variables and (as in \cite{BorCor}) show that for any
initial data \mbox{$q$-}TASEP is, in fact, well defined.

The totally asymmetric zero range process (TAZRP) on an interval\break $\{
0,1,\ldots,N\}$ with site-dependent rate functions $g_i\dvtx \mathbb
{Z}_{\geq0}
\to[0,\infty)$ [with $g_i(0)\equiv0$ fixed] is a Markov process
$\vec{y}(t)$ with state space
\[
Y^N=(\mathbb{Z}_{\geq0})^{\{0,1,\ldots,N\}}.
\]
The dynamics of TAZRP are given as follows: for each $i\in\{1,\ldots,N\}$, $y_i(t)$ decreases by one and $y_{i-1}(t)$ increase by one
(simultaneously) in continuous time at rate given by $g_i(y_i(t))$; for
different $i$'s these changes occur independently (see\vadjust{\goodbreak} right-hand side of Figure~\ref{qtasepfig}). Note\vspace*{1pt} that no
particles leave site 0. The rate functions we consider are given by
$g_{i}(k)=a_{i}(1-q^k)$. When all $a_i\equiv1$, this model was
first introduced in \cite{SasWad} and further studied in \cite{Pov}.



The generator of $\vec{y}(t)$ acts on functions $h\dvtx Y^N\to\mathbb{R}$
and is
given by
%
%
\begin{equation}
\label{eqnqTAZRPgen} \bigl(L^{q\mbox{-}\mathrm{TAZRP}}h\bigr) (\vec{y}) = \sum
_{i=1}^{N} a_{i}\bigl(1-q^{y_{i}}
\bigr) \bigl(h\bigl(\vec{y}^{i,i-1}\bigr) - h(\vec{y}) \bigr),
\end{equation}
where $\vec{y}^{i,i-1}$ indicates to decrease $y_i$ by one and
increase $y_{i-1}$ by one.

Observe that the gaps $\tilde{y}_i(t) = x_{i}(t)-x_{i+1}(t)-1$ of
\mbox{$q$-}TASEP evolve according to a TAZRP, but with boundary conditions
that $\tilde{y}_0(t)\equiv\tilde{y}_{N}(t)\equiv\infty$ for all
$t\in\mathbb{R}_{+}$. Our work will not draw on this obvious coupling.
Rather, our statement of duality will provide a different relationship
between $\vec{x}(t)$ and an independent $\vec{y}(t)$.


\subsection{Duality}
Recall the general definition of duality given in Definition 3.1 of~\cite{Lig}.

%
\begin{definition}\label{dualdef}
Suppose $x(t)$ and $y(t)$ are independent Markov processes with state
spaces $X$ and $Y$, respectively, and let $H(x,y)$ be a bounded
measurable function on $X\times Y$. The processes $x(t)$ and $y(t)$ are
said to be \textit{dual} to one another with respect to $H$ if
%
%
\begin{equation}
\label{eqndualdef} \mathbb{E}^{x} \bigl[H\bigl(x(t),y\bigr) \bigr] =
\mathbb{E}^{y} \bigl[H\bigl(x,y(t)\bigr) \bigr]
\end{equation}
for all $x\in X$ and $y\in Y$. Here $\mathbb{E}^{x}$ refers to the process
$x(t)$ started with $x(0)=x$ (likewise for $y$).
\end{definition}

%
\begin{theorem}\label{thmqtasepduality}
The \mbox{$q$-}TASEP $\vec{x}(t)$ with state space $X^N$ and particle jump
rate parameters $a_i>0$, and the TAZRP $\vec{y}(t)$ with state
space $Y^N$ and rate functions $g_i(k)=a_i(1-q^k)$ are dual
with respect to
\[
H(\vec{x},\vec{y}) = \prod_{i=0}^{N}
q^{(x_i+i)y_i}.
\]
\end{theorem}
%
%
\begin{remark}
The definition of $H(\vec{x},\vec{y})$ means that $H=0$ if $y_0>0$
and $H(\vec{x},\vec{y})=\prod_{i=1}^{N} q^{(x_i+i)y_i}$ if $y_0=0$.
\end{remark}

Before proving Theorem~\ref{thmqtasepduality}, we define the
following system of ODEs.

%
\begin{definition}\label{qtaseptrueevol}
We say that $h(t;\vec{y})\dvtx \mathbb{R}_{+}\times Y^N\to\mathbb{R}$
solves the \textit{true evolution equation} with initial data $h_0(\vec{y})$ if:
\begin{longlist}[(2)]
\item[(1)] For all $\vec{y}\in Y^N$ and $t\in\mathbb{R}_{+}$,
\[
\frac{d}{dt} h(t;\vec{y}) = L^{q\mbox{-}\mathrm{TAZRP}} h(t;\vec{y});
\]
\item[(2)] For all $\vec{y}\in Y^N$ such that $y_0>0$, $h(t;\vec{y})\equiv
0$ for all $t\in\mathbb{R}_{+}$;
\item[(3)] For all $\vec{y}\in Y^N$, $h(0;\vec{y}) = h_0(\vec{y})$.
\end{longlist}
\end{definition}
%
%
\begin{remark}\label{existuniqtrue}
The existence and uniqueness of global solutions to the true evolution
equation in Definition~\ref{qtaseptrueevol} is assured since it
reduces to a finite system of linear ODEs, from which the result
follows from standard methods \cite{Cod}.
\end{remark}

\begin{pf*}{Proof of Theorem~\ref{thmqtasepduality}}
We claim first that for $\vec{x}$ and $\vec{y}$ fixed,
%
%
\begin{equation}
\label{eqnqTASEPLq} L^{q\mbox{-}\mathrm{TASEP}} H(\vec{x},\vec{y}) = L^{q\mbox{-}\mathrm{TAZRP}} H(\vec{x},
\vec{y}),
\end{equation}
where in the above expression, the generator on the left acts in the
$x$ variables and the generator on the right in the $y$ variables.

To prove the claim is easy. Observe that
\begin{eqnarray*}
L^{q\mbox{-}\mathrm{TASEP}} H(\vec{x},\vec{y}) &=& \sum_{i=1}^N
a _i \bigl(1-q^{x_{i-1}-x_{i}-1} \bigr) \Biggl(\bigl(q^{y_i}-1
\bigr) \prod_{j=0}^N q^{(x_j+j)y_j}
\Biggr)
\\
&=& \sum_{i=1}^N a_i
\bigl(1-q^{y_{i}}\bigr) \bigl(H\bigl(\vec{x},\vec {y}^{i,i-1}\bigr)
- H(\vec{x},\vec{y}) \bigr)
\\
&=& L^{q\mbox{-}\mathrm{TAZRP}} H(\vec{x},\vec{y}).
\end{eqnarray*}

Given the claim we may now check that $\mathbb{E}^{\vec{x}}
[H(\vec
{x}(t),\vec{y}) ]$ and $\mathbb{E}^{\vec{y}} [H(\vec
{x},\vec
{y}(t)) ]$ both satisfy the true evolution equation given in
Definition~\ref{qtaseptrueevol}. By
the uniqueness of Remark~\ref{existuniqtrue}, this implies the desired
equality to complete our proof. That $\mathbb{E}^{\vec{y}}
[H(\vec
{x},\vec{y}(t)) ]$ satisfies this evolution equation follows
from the definition of the generator of $\vec{y}(t)$.

On the other hand,
\begin{eqnarray*}
\frac{d}{dt} \mathbb{E}^{\vec{x}} \bigl[H\bigl(\vec{x}(t),\vec {y}
\bigr) \bigr] &=& L^{q\mbox{-}\mathrm{TASEP}} \mathbb{E}^{\vec{x}} \bigl[H\bigl(
\vec{x}(t),\vec {y}\bigr) \bigr]
\\
&=&\mathbb{E}^{\vec{x}} \bigl[ L^{q\mbox{-}\mathrm{TASEP}} H\bigl(\vec {x}(t),\vec {y}
\bigr) \bigr]
\\
&=&\mathbb{E}^{\vec{x}} \bigl[ L^{q\mbox{-}\mathrm{TAZRP}} H\bigl(\vec {x}(t),\vec {y}
\bigr) \bigr]
\\
&=&L^{q\mbox{-}\mathrm{TAZRP}} \mathbb{E}^{\vec{x}} \bigl[ H\bigl(\vec {x}(t),\vec {y}
\bigr) \bigr].
\end{eqnarray*}
The equality of the first line is from the definition of the generator
of $\vec{x}(t)$; the equality between the first and second lines is
from the commutativity of the generator with the Markov semigroup; the
equality between the second and third lines is from applying equality
(\ref{eqnqTASEPLq}) to the expression inside the expectation; the
final equality is from the fact that the generator $L^{q\mbox{-}\mathrm{TAZRP}}$ now acts on the $\vec{y}$ coordinate and the expectation
acts on the $\vec{x}$ coordinate. This shows that $\mathbb{E}^{\vec
{x}} [H(\vec{x}(t),\vec{y}) ]$ solves the system of ODEs
in the true evolution equation (checking the boundary condition and
initial data is easy).
\end{pf*}

\subsection{Systems of ODEs}
As a result of duality, we provide three different systems of ODEs to
characterize $\mathbb{E}^{\vec{x}}  [H(\vec{x}(t),\vec
{y}) ]$.
It is convenient to introduce an alternative way to write a TAZRP state
$\vec{y}\in Y^N$. For a state with $k$ particles, we may instead list
the ordered particle locations $\vec{n}$ as below.
%
\begin{definition}
For $k\geq1$, define
\[
W^{k}_{>0}= \bigl\{\vec{n} =(n_1,n_2,
\ldots,n_k)\in(\mathbb{Z}_{>0})^k\dvtx  N\geq
n_1\geq n_2\geq\cdots\geq n_k\geq0 \bigr\}.
\]
For $\vec{y}\in Y^N$ with $\sum_{i=0}^{N} y_i =k$, we may associate a
vector $\vec{n} = \vec{n}(y)\in W^{k}_{>0}$ which records the
ordered locations of particles in $\vec{y}$. That is to say, for $i\in
\{0,\ldots,N\}$, the vector $\vec{n}(y)$ is specified by $|\{
n_j\dvtx n_j=i\}| = y_i$. Likewise, to a vector $\vec{n}\in W^{k}_{>0}$
we may associate $\vec{y}=\vec{y}(\vec{n})\in Y^N$ by the same
relationship $y_i=|\{n_j\dvtx n_j=i\}|$. For instance, if $N=3$, $y_1=2$,
$y_2=0$ and $y_3=1$ then $k=3$, $n_1=3$ and $n_2=n_3=1$. A vector $\vec
{n}$ naturally splits into clusters, which are maximal groupings of
consecutive equal valued elements. For instance, if $\vec{n}=
(4,4,2,1)$, we would say there are three clusters with the cluster of
$4$ containing two elements, and the clusters of $2$ and $1$ containing
only one elements each.

Also, define the difference operator $\nabla f(n) = f(n-1)-f(n)$. For a
function $f(\vec{n})$, $\nabla_i$ acts as $\nabla$ on the $n_i$
variable. Finally, let $\vec{n}_i^{-} = (n_1,\ldots, n_{i} -1,\ldots, n_k)$.
\end{definition}

%
\begin{proposition}\label{propsystemsODEqTASEP} Let $\vec{x}\in X^N$
and $\vec{x}(t)$ be the \mbox{$q$-}TASEP started from \mbox{$\vec{x}(0)=\vec{x}$}.
\begin{longlist}[(A)]
\item[(A)] \emph{True evolution equation}: If $h(t;\vec{y})\dvtx \mathbb{R}_{+}
\times Y^N\to\mathbb{R}$ solves the system of ODEs given\vspace*{1pt} in
Definition~\ref
{qtaseptrueevol} with initial data $h_0(\vec{y}) = H(\vec{x},\vec
{y})$, then for all \mbox{$\vec{y}\in Y^N$}, $\mathbb{E}^{\vec{x}}
[H(\vec
{x}(t),\vec{y}) ] = h(t;\vec{y})$.

\item[(B)] \emph{Free evolution equation with $k-1$ boundary
conditions}: If $u\dvtx \mathbb{R}_{+}\times(\mathbb{Z}_{\geq0})^k \to
\mathbb{R}$ solves:
\begin{enumerate}[(2)]
\item[(1)] For all $\vec{n}\in(\mathbb{Z}_{\geq0})^k$ and $t\in\mathbb{R}_{+}$,
\[
\frac{d}{dt} u(t;\vec{n}) = (1-q) \sum_{i=1}^{k}
a_{n_i} \nabla_i u (t;\vec{n});
\]
\item[(2)] For all $\vec{n}\in(\mathbb{Z}_{\geq0})^k$ such that for some
$i\in\{
1,\ldots, k-1\}$, $n_i=n_{i+1}$,
\[
\nabla_i u(t;\vec{n}) = q \nabla_{i+1} u(t;\vec{n});
\]
\item[(3)] For all $\vec{n}\in(\mathbb{Z}_{\geq0})^k$ such that $n_k=0$,
$u(t;\vec{n}) \equiv0$ for all $t\in\mathbb{R}_{+}$;
\item[(4)] For all $\vec{n}\in W^{k}_{>0}$, $u(0;\vec{n}) = H(\vec
{x},\vec{y}(\vec{n}))$.
\end{enumerate}
Then for all $\vec{y}\in Y^N$ such that $k=\sum_{i=1}^{N} y_i$,
$\mathbb{E}
^{\vec{x}}  [H(\vec{x}(t),\vec{y}) ] = u(t;\vec{n}(\vec{y}))$.

\item[(C)] \emph{Schr\"{o}dinger equation with Bosonic Hamiltonian}:
If $v\dvtx \mathbb{R}_{+}\times(\mathbb{Z}_{\geq0})^k\to\mathbb{R}$ solves:
\begin{enumerate}[(2)]
\item[(1)] For all $\vec{n}\in(\mathbb{Z}_{\geq0})^k$ and $t\in\mathbb{R}_{+}$,
\begin{eqnarray*}
\frac{d}{dt} v(t;\vec{n}) &=& \mathcal{H}v(t;\vec{n}),
\\
\mathcal{H} &=& (1-q) \Biggl[\sum_{i=1}^{k}
a_{n_i}\nabla_i + \bigl(1-q^{-1}\bigr) \sum
_{i<j}^k \delta_{n_i=n_j}
q^{j-i} a_{n_i}\nabla _i \Biggr];
\end{eqnarray*}
\item[(2)] For all permutations of indices $\sigma\in S_k$, $v(t;\sigma
\vec{n}) = v(t;\vec{n})$;
\item[(3)] For all $\vec{n}\in(\mathbb{Z}_{\geq0})^k$ such that $n_k=0$,
$v(t;\vec{n}) \equiv0$ for all $t\in\mathbb{R}_{+}$;\vspace*{2pt}
\item[(4)] For all $\vec{n}\in W^{k}_{>0}$, $v(0;\vec{n})= H(\vec
{x},\vec{y}(\vec{n}))$.
\end{enumerate}
Then for all $\vec{y}\in Y^N$ such that $k=\sum_{i=1}^{N} y_i$,
$\mathbb{E}
^{\vec{x}}  [H(\vec{x}(t),\vec{y}) ] = v(t;\vec{n}(\vec{y}))$.
\end{longlist}
\end{proposition}

%
\begin{remark}\label{rem:unique}
The existence and uniqueness of global solutions to (A) is explained in
Remark~\ref{existuniqtrue}. This then implies the existence of
solutions in (C). It is not clear, a priori, that there exist solutions
to (B). As we see in the proof of (B), the combination of the four
conditions in (B) implies that restricted to $\vec{n}\in W^{k}_{>0}$,
$u(t;\vec{n}) = h(t;\vec{y}(\vec{n}))$ for all $t\in\mathbb{R}_{+}$.
However, it is not clear that there exists a suitable extension of $u$
outside the physical region $W^{k}_{>0}$ which satisfies the four
conditions. Note that (B) should be considered as an advanced version
of the method of images. Finally, though the above results are written
for deterministic $\vec{x}$ (i.e., deterministic initial data) by
linearity one can average over random $\vec{x}$ and achieve the same
stated results with $\mathbb{E}^{\vec{x}}  [H(\vec{x}(t),\vec
{y}) ]$ replaced by its average over $\vec{x}$, written as
$\mathbb{E}
[H(\vec{x}(t),\vec{y}) ]$, and the initial data for the
ODEs likewise replaced by $\mathbb{E} [H(\vec{x},\vec{y}) ]$.
\end{remark}

\begin{pf*}{Proof of Proposition~\ref{propsystemsODEqTASEP}}
Call the three conditions contained in Definition~\ref{qtaseptrueevol} (A.1), (A.2) and~(A.3). Part~(A) follows from Theorem~\ref
{thmqtasepduality} since it implies that
\[
\frac{d}{dt} \mathbb{E}^{\vec{x}} \bigl[H\bigl(\vec{x}(t),\vec {y}
\bigr) \bigr] = L^{q\mbox{-}\mathrm{TAZRP}}\mathbb{E}^{\vec{x}} \bigl[H\bigl(\vec{x}(t),
\vec {y}\bigr) \bigr],
\]
which matches (A.1). Along with this, the value of $\mathbb{E}^{\vec{x}}
[H(\vec{x}(t),\vec{y}) ]$ is uniquely characterized by
the initial data and the fact that (due to the definition of $H$)
$\mathbb{E}
^{\vec{x}}  [H(\vec{x}(t),\vec{y}) ]=0$ for all $\vec
{y}\in Y^N$ with $y_0>0$. Conditions (A.3) and (A.2), respectively,
match these properties, and hence (A) follows.

Part (B) follows by showing that if the four conditions for $u$ given
in (B) hold, then it implies that $u(t;\vec{n}(\vec{y}))$ satisfies
part (A), and hence that $u(t;\vec{n}(\vec{y}))=h(t;\vec{y})$. Thus,
we must show that (B) implies (A). Going between $\vec{y}$ and $\vec
{n}$ notation, the initial data (A.3) and (B.4) match, as do the
conditions (A.2) and (B.3). To check the system of ODEs (A.1), recall
that the size of the cluster of elements of $\vec{n}$ equal to $i$
equals $y_i$. Consider the cluster of elements equal to $N \dvtx
n_1=n_2=\cdots=n_{y_N}$ (every other cluster works similarly). In
order to prove (A.1), it suffices to show that
%
%
\begin{equation}
\label{eqnidPartB} (1-q) \sum_{i=1}^{y_N}
a_N \nabla_i u(t;\vec{n}) = a_N
\bigl(1-q^{y_N}\bigr) \nabla_{y_N} u(t;\vec{n}).
\end{equation}
This\vspace*{1pt} is because $\nabla_{y_N} u(t;\vec{n}) = u(t;\vec{n}(y^{N,N-1}))
- u(t;\vec{n}(y))$. Summing these terms over all clusters yields
$L^{q\mbox{-}\mathrm{TAZRP}}u(t;\vec{n}(\vec{y}))$, and hence (A.1)
follows. But (B.2) implies $\nabla_i u(t;\vec{n}) = q^{y_N -i} \nabla
_{y_N} u(t;\vec{n})$ for $i=1,\ldots,y_N$, which implies (\ref{eqnidPartB}).


Part (C) also follows by showing that the combination of the four
conditions for~$v$ imply that $v(t;\vec{n}(\vec{y}))$ satisfies (A),
and hence $v(t;\vec{n}(\vec{y}))=h(t;\vec{y})$. As in (B), the
initial data (A.3) and (C.4) match, as do the conditions (A.2) and
(C.3). Also as in (B), it suffices to consider the cluster of $N$. The
portion of the Hamiltonian~$H$ corresponding to this cluster is
\[
(1-q) \Biggl[\sum_{i=1}^{y_N}
a_{N}\nabla_i + \bigl(1-q^{-1}\bigr) \sum
_{i<j}^{y_N} q^{j-i}
a_{N}\nabla_i \Biggr]= (1-q) a_N \sum
_{i=1}^{y_N} q^{y_{N}-i}
\nabla_i,
\]
where the equality follows from summing the factors involving each
$\nabla_i$. 
Due to the symmetry (C.2), $\nabla_i v(t;\vec{n}) = \nabla_{y_N}
v(t;\vec{n})$ for all $i\in\{1,\ldots, y_N\}$. Hence, the sum in~$i$
can be performed, yielding $a_N(1-q^{y_N})\nabla_{y_N}$. This is the
same as in (\ref{eqnidPartB}), and hence (C) follows as well.
\end{pf*}

\subsection{Nested contour ansatz solution}
It is not a priori clear how one might explicitly solve the systems of
ODEs in Proposition~\ref{propsystemsODEqTASEP}. Presently, we show
how this can be done for two distinguished types of initial data.

%
\begin{definition}\label{qTASEPstepdef}
For \mbox{$q$-}TASEP, \textit{step} initial data corresponds with $x_i(0) = -i$
for $i\geq1$.

For $\alpha\in[0,1)$, we say a random variable $X$ is $q$-Geometric
distributed with parameter $\alpha$ [written $X\sim q\operatorname{Geo}(\alpha)$] if
\[
\mathbb{P}(X= k) = (\alpha;q)_{\infty} \frac{\alpha^k}{(q;q)_{k}},
\]
where $(a;q)_{n}=(1-a)(1-aq)(1-aq^2)\cdots(1-aq^{n-1})$ and
$(a;q)_{\infty}=(1-a)(1-aq)(1-aq^2)\cdots.$ \textit{Half stationary}
initial data for \mbox{$q$-}TASEP corresponds with random initial locations
for particles $x_i$ for $i\geq1$ given as follows: let $X_i \sim
q\operatorname{Geo}(\alpha/a_i)$ for $i\geq1$ be independent; then
set $x_1(0)= -1-X_1$ and, for $i>1$, $x_i = x_{i-1} -1-X_i$. The result
is that the gaps between consecutive particles $i$ and $i+1$ are
distributed as $q$-Geometric with parameter $\alpha/a_i$ and
are independent. When $\alpha=0$, the step initial data is recovered
(regardless of the $a_i$).
\end{definition}
%
%
\begin{remark}
When $a_i\equiv1$, the translation invariant measure on
particle configurations in $\mathbb{Z}$ with independent $q\operatorname{Geo}(\alpha)$ distributed distances between neighbors is an invariant
or stationary measure\footnote{Within the probability literature, the
term equilibrium is sometimes also used to describe such a measure,
though to avoid confusion with the physical means of equilibrium
statistical mechanics, we avoid this term.} for \mbox{$q$-}TASEP, cf. \cite
{BorCor}. This explains the usage of the term half stationary (and
likewise for ASEP).
\end{remark}
%
%
\begin{theorem}\label{thmqformulas}
Fix $q\in(0,1)$, $a_i>0$ for $i\geq1$ and let $\vec
{n}=(n_1,\ldots, n_k)$. The system of ODEs given in Proposition~\textup{\ref{propsystemsODEqTASEP}(B)} is solved by the following formulas:
\begin{longlist}[(2)]
\item[(1)] For step initial data,
%
%
\begin{eqnarray}\label{FdefqTASEP}
\qquad u(t;\vec{n}) &=& \frac{(-1)^k q^{k(k-1)/2}}{(2\pi\iota)^k}
\nonumber\\[-8pt]\\[-8pt]
&&{}\times  \int\cdots\int\prod
_{1\leq A<B\leq k} \frac{z_A-z_B}{z_A-qz_B} \prod_{j=1}^{k}
\Biggl(\prod_{m=1}^{n_j} \frac{a_m}{a_m-z_j}
\Biggr) e^{(q-1)tz_j} \frac{dz_j}{z_j},\nonumber
\end{eqnarray}
where the integration contour for $z_A$ contains $\{qz_B\}_{B>A}$ and
all $a_m$'s but not 0.
\item[(2)] For half stationary initial data with parameter $\alpha>0$ [such
that $\alpha q^{-k}<a_m$ for all $1\leq m \leq\max_i(n_i)$],
%
%
\begin{eqnarray}\label{FdefqTASEPhalf}
u(t;\vec{n}) &=& \frac{(-1)^k q^{k(k-1)/2}}{(2\pi\iota)^k}\nonumber
\\
&&{}\times \int\cdots\int\prod
_{1\leq A<B\leq k} \frac{z_A-z_B}{z_A-qz_B}
\\
&&\hspace*{49pt}{}\times  \prod_{j=1}^{k}
\Biggl(\prod_{m=1}^{n_j} \frac{a_m}{a_m-z_j}
\Biggr) e^{(q-1)tz_j} \frac{dz_j}{z_j-\alpha/q},\nonumber
\end{eqnarray}
where the integration contour for $z_A$ contains $\{qz_B\}_{B>A}$ and
all $a_m$'s but not $\alpha/q$.
\end{longlist}
\end{theorem}

On account of Proposition~\ref{propsystemsODEqTASEP} and the
uniqueness of solutions restricted to $\vec{n}\in W^{k}_{>0}$ (see
Remark~\ref{rem:unique}), the above formulas when restricted to $\vec
{n}\in W^{k}_{>0}$ immediately yield the following (we will only
state it for step initial data, though a similar statement holds for
half stationary).
%
\begin{corollary}\label{cor:qformulas}
For \mbox{$q$-}TASEP with step initial data and $\vec{n}\in W^{k}_{>0}$,
%
%
\begin{eqnarray}\label{cor:FdefqTASEP}
\qquad && \mathbb{E} \Biggl[\prod_{j=1}^{k}
q^{x_{n_j}+n_j} \Biggr]\nonumber
\\
&&\qquad = \frac{(-1)^k
q^{k(k-1)/2}}{(2\pi\iota)^k}
\\
&&\quad\qquad {}\times  \int\cdots\int\prod
_{1\leq A<B\leq k} \frac{z_A-z_B}{z_A-qz_B} \prod_{j=1}^{k}
\Biggl(\prod_{m=1}^{n_j} \frac{a_m}{a_m-z_j}
\Biggr) e^{(q-1)tz_j} \frac{dz_j}{z_j},\nonumber
\end{eqnarray}
where the integration contour for $z_A$ contains $\{qz_B\}_{B>A}$ and
all $a_m$'s but not 0.
\end{corollary}

\begin{pf*}{Proof of Theorem~\ref{thmqformulas}}
We need to prove that $u(t;\vec{n})$ as defined in (\ref{FdefqTASEP})
satisfies the four conditions of Proposition~\textup{\ref{propsystemsODEqTASEP}(B)}.

Condition (B.1) is satisfied by linearity and the fact that
\[
\biggl[\frac{d}{dt} - (1-q)a_{n_i}\nabla_i \biggr]
\Biggl( \Biggl(\prod_{m=1}^{n_i}
\frac{a_m}{a_m-z} \Biggr) e^{(q-1)tz} \Biggr) = 0.
\]


Condition (B.2) relies on the Vandermonde-like factors as well as the
nested choice of contours. Without loss of generality, assume that
$n_1=n_2$. We wish to show that
%
%
\begin{equation}
\label{eqnqrelprove} [\nabla_1-q\nabla_2 ] u(t;\vec{n}) = 0.
\end{equation}
Applying $\nabla_1-q\nabla_2$ to the integrand in (\ref{FdefqTASEP})
brings down a factor of $-a_{n_1}^{-1} (z_1-qz_2)$. We must show that
the integral of this new integrand is zero. This new factor cancels the
denominator $(z_1-q z_2)$ in
%
%
\begin{equation}
\label{eqnqskewprod} \prod_{1\leq A<B\leq k} \frac{z_A-z_B}{z_A-qz_B}.
\end{equation}
On account of this, we may deform the contours for $z_1$ and $z_2$ to
be the same without encountering any poles. The term $z_1-z_2$ in the
numerator of (\ref{eqnqskewprod}) remains, and hence we can write
\[
u(t;\vec{n})= \int\!\!\int(z_1-z_2) G(z_1)G(z_2)
\,dz_1 \,dz_2,
\]
where $G(z)$ involves the integrals in $z_3,\ldots, z_k$. Since the
two contours are identical, this integral is clearly zero, proving (B.2).

Condition (B.3) follows from simple residue calculus. When $n_k=0$,
there are no poles in the $z_k$ integral at $\{a_m\}_{m=1}^{n_k}$.
Therefore, by Cauchy's theorem the integral is zero.

Condition (B.4) likewise follows from residue calculus. Let us first
consider the step initial data case of Theorem~\ref{thmqformulas}.
This corresponds to initial data in (B.4) given by
\[
u(0;\vec{n}) = H\bigl(x;\vec{y}(\vec{n})\bigr) = 1
\]
for all $\vec{n}\in W^{k}_{>0}$. Now consider (\ref{FdefqTASEP})
with $t=0$. The $z_1$ contour can be expanded to infinity. The only
pole encountered is at $z_1=0$ [$z_1=\infty$ is not a pole because of
the decay coming from $a_m/(a_m-z_j)$]. Because we pass it from the
outside, the contribution of the residue is $-q^{-(k-1)}$ times the
same integral but with every term involving $z_1$ removed. Repeating
this procedure for $z_2$ leads to $-q^{-(k-2)}$ and so on. Therefore,
the integral can be evaluated and canceling terms we are left with it
equal to $1$ exactly as desired.

Now consider the half stationary initial data case of Theorem~\ref
{thmqformulas}. (B.1)--(B.3) follow in the same way as for the step
initial data. We claim that this corresponds to initial data in (B.4)
given by
%
%
\begin{equation}
\label{eqnhalfeqid} u(0;\vec{n}) = \mathbb{E} \bigl[H\bigl(\vec{x};\vec{y}(\vec{n})
\bigr) \bigr] = \prod_{i=1}^{k}\,
\prod_{m=n_{i+1}+1}^{n_i}\,
\prod_{j=1}^{i} \frac
{a_m}{a_m-\alpha/q^{j}}.
\end{equation}
Showing this requires a calculation.\vadjust{\goodbreak}
%
\begin{lemma}
Fix $r\geq1$. If $X$ is a $q$-Geometric random variable with parameter
$\alpha\in[0,1)$ then
\[
\mathbb{E} \bigl[q^{-rX} \bigr]= \prod_{i=1}^{r}
\frac{1}{1-\alpha/q^{i}},
\]
so long as $\alpha q^{-r}<1$; and otherwise the expectation is infinite.
\end{lemma}
\begin{pf}
Using the $q$-Binomial theorem (see Section~\ref{secqdef}), we may calculate
\[
\mathbb{E} \bigl[q^{-rX} \bigr] = (\alpha;q)_{\infty} \sum
_{k=0}^{\infty} \frac{(\alpha/q^r)^k}{(q;q)_k} =
\frac{(\alpha;q)_{\infty
}}{(\alpha/q^r;q)_{\infty}},
\]
which after canceling terms is exactly as desired.
\end{pf}
Recall that under half stationary initial data, the locations $\{
x_i(0)\}$ are defined in terms of $q$-Geometric random variables $\{
X_j\}$. Using this, we have
\[
\prod_{i=1}^{k} q^{x_{n_i}(0)+n_i} = \prod
_{i=1}^{k} q^{-\sum
_{m=1}^{n_i} X_{m}} = \prod
_{i=1}^{k}\, \prod
_{m=n_{i+1}+1}^{n_i} q^{-i X_{m}}.
\]
Since the $X$'s are independent, we can evaluate individually the
expectation of~$q^{-iX_m}$ using the above lemma, and we immediately
find the right-hand side of~(\ref{eqnhalfeqid}).

Now consider (\ref{FdefqTASEPhalf}) with $t=0$. As in the step initial
data case, we successively peel off the contours and evaluate the
effect via residue calculus. When we expand~$z_1$ to infinity, we now
only encounter a pole at $z_1=\alpha/q$ (which becomes zero when
$\alpha=0$ recovering the step initial data). Evaluating this residue,
we find
\begin{eqnarray*}
u(0;\vec{n}) &=& \prod_{m=1}^{n_1}\frac{a_m}{a_m-\alpha/q}
\frac{(-1)^{k-1} q^{((k-1)(k-2))/2}}{(2\pi\iota)^{k-1}}
\\
&&{}\times  \int\cdots \int\prod
_{2\leq A<B\leq k} \frac{z_A-z_B}{z_A-qz_B} \prod_{j=2}^{k}
\Biggl(\prod_{m=1}^{n_j} \frac{a_m}{a_m-z_j}
\Biggr)\frac
{dz_j}{z_j-\alpha/q^2}.
\end{eqnarray*}
Expanding $z_2$ the pole is now at $z_2=\alpha/q^2$ and a similar
formula results from evaluating the residue. Repeating this procedure
shows that $u(0;\vec{n})$ is given by the right-hand side of (\ref
{eqnhalfeqid}) as desired.
\end{pf*}

\section{A general scheme from nested contour integrals to Fredholm determinants}\label{sec3}

The output of Theorem~\ref{thmqformulas} and Corollary~\ref
{cor:qformulas} is that for step and half stationary initial data we
have relatively simple formulas for a large class of expectations. In
principle, these expectations should characterize the joint
distribution of the locations of any fixed collection of particles
$x_{n_1}(t),\ldots, x_{n_\ell}(t)$ in \mbox{$q$-}TASEP. One may hope to
achieve this via certain generating functions. However, the challenge
is to find expressions for these generating functions which have clear
asymptotic limits (in time and particle labels). For this, we focus
here only on the distribution of a single particle $x_n(t)$. Applying
Corollary~\ref{cor:qformulas} with $n_i\equiv n$ yields a nested
contour integral formula for $\mathbb{E} [q^{kx_n(t)} ]$.

There are two ways to deform this type of nested contour integrals so
that all contours coincide. After accounting for the residues
encountered during these deformations, we are led to two types of
formulas for expectations: those involving partition-indexed sums of
contour integrals and those involving single row-indexed sums of
contour integrals. By taking suitable generating functions of these
indexed sums of contour integrals, we are led to two types of Fredholm
determinants. All of these manipulations are quite general and can be
done purely formally.
Given some analytic estimates, these manipulations turn into numerical
equalities as is the case for \mbox{$q$-}TASEP.

We record these manipulations (and conditions for them to hold as
numerical equalities) as well as their consequences without proofs,
since they can be found in Section~3.2 of \cite{BorCor}. We do this
for completeness and also because when we turn to consider ASEP, the
same manipulations will be used. When we apply this to \mbox{$q$-}TASEP, we
will only consider step initial data. If we consider half stationary
for any $\alpha>0$ fixed, then when $k$ gets so large that $\alpha
>q^k$, the expectation $\mathbb{E} [q^{k x_n(t)} ]$ is infinite.
Thus, when forming a generating function from these $q$-moments, we are
forced to take $\alpha=0$, which corresponds to the step initial data.

Before going into these manipulations, the reader may want to quickly
browse Section~\ref{secqdef} where we record some useful
$q$-deformations as well as briefly review Fredholm determinants.

\subsection{Mellin--Barnes type determinants}\label{secmellin}

The below proposition describes the result of deforming the contours of
a general nested contour integral formula in such a way that all of the
poles corresponding to $z_A=qz_B$ for $A<B$ are encountered. The
residues associated with these poles group into clusters, and hence the
resulting formula is naturally indexed by partitions $\lambda=(\lambda
_1\geq\lambda_2\geq\cdots\geq0)$. Notationally, we write $\lambda
\vdash k$ if $\sum_i \lambda_i = k$, $\lambda=1^{m_1} 2^{m_2}\cdots
$ if $i$ appears $m_i$ times in $\lambda$ (for all $i\geq1$), and
$\ell(\lambda) = \sum_{i}m_i$ for the number of nonzero elements of
$\lambda$.

%
\begin{definition}\label{def:mukdef}
For a meromorphic function $f(z)$ and $k\geq1$ set $\mathbb{A}$ to be a
fixed set of poles of $f$ (not including 0) and assume that $q^m\mathbb{A}
$ is disjoint from $\mathbb{A}$ for all $m\geq1$. Define
%
%
\begin{equation}
\label{mukdef} \mu_k=\frac{(-1)^k q^{k(k-1)/2}}{(2\pi\iota)^k} \int \cdots\int\prod
_{1\leq A<B\leq k} \frac{z_A-z_B}{z_A-qz_B} \prod_{i=1}^{k}
f(z_i) \frac{dz_i}{z_i},
\end{equation}
where the integration contour for $z_A$ contains $\{q z_B\}_{B> A}$,
the fixed set of poles $\mathbb{A}$ of $f(z)$ but not 0 or any other poles.
\end{definition}

%
\begin{proposition}\label{mukprop}
We have that for $\mu_k$ as in Definition~\ref{def:mukdef},
%
%
\begin{eqnarray}\label{muk}
\mu_k &=& k_q! \mathop{\sum_{\lambda\vdash k}}_{\lambda
=1^{m_1}2^{m_{2}}\cdots} \frac{1}{m_1!m_2!\cdots} \frac
{(1-q)^{k}}{(2\pi\iota)^{\ell(\lambda)}}\nonumber
\\
&&{}\times
\int\cdots\int\det \biggl[\frac{1}{w_i q^{\lambda_i}-w_j} \biggr]_{i,j=1}^{\ell(\lambda
)}
\\
&&\hspace*{45pt}{}\times \prod_{j=1}^{\ell(\lambda)} f(w_j)f(qw_j)
\cdots f\bigl(q^{\lambda
_j-1}w_j\bigr) \,dw_j,\nonumber
\end{eqnarray}
where the integration contour for $w_j$ contains the same fixed set of
poles $\mathbb{A}$ of $f$ and no other poles.
\end{proposition}
\begin{pf}
This is proved in \cite{BorCor} as Proposition 3.2.1 via residue calculus.
\end{pf}

%
\begin{figure}

\includegraphics{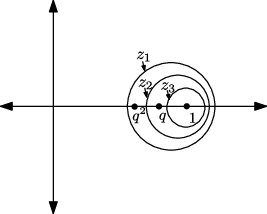}

\caption{Possible contours when $k=3$ for the $z_j$ contour integrals
in Proposition \protect\ref{mukprop}.}\label{circontours}
\end{figure}

As a quick example, consider $f(z)$ which has a pole at $z=1$. Then the
\mbox{$z_k$-}contour is a small circle around 1, the $z_{k-1}$-contour goes
around $1$ and $q$, and so on until the $z_1$-contour encircles $\{
1,q,\ldots, q^{k-1}\}$ (this is illustrated for $k=3$ in Figure~\ref
{circontours}). All the $w$ contours are small circles around 1 and can
be chosen to be the same.

We form a generating function of the $\mu_k$ and identify the result
as a Fredholm determinant.

%
\begin{proposition}\label{gendetprop}
Consider $\mu_k$ as in equation (\ref{muk}) defined with respect to
the same set of poles $\mathbb{A}$ of $f(w)$ for $k=1,2,\ldots$ and set
$C_{\mathbb{A}}$ to be a closed contour which contains $\mathbb{A}$
and no
other poles of $f(w)/w$. Then the following formal equality holds:
\[
\sum_{k\geq0}\mu_k \frac{\zeta^k}{k_q!} =
\det\bigl(I+K_\zeta^1\bigr),
\]
where $\det(I+K_\zeta^1)$ is the formal Fredholm determinant
expansion of
$K_\zeta^1\dvtx\break  L^2(\mathbb{Z}_{>0}\times C_{\mathbb{A}}) \to L^2(\mathbb
{Z}_{>0}\times
C_{\mathbb{A}})$
defined in terms of its integral kernel
\[
K_\zeta^1(n_1,w_1;n_2;w_2)
= \frac{(1-q)^{n_1}\zeta^{n_1}
f(w_1)f(qw_1)\cdots f(q^{n_1-1}w_1)}{q^{n_1} w_1 - w_2}.
\]
The above identity is formal, but also holds numerically if the
following is true: for all $w,w'\in C_{\mathbb{A}}$ and $n\geq1$, $|q^n
w-w'|^{-1}$ is uniformly bounded from zero; and there exists a positive
constant $M$ such that for all $w\in C_{\mathbb{A}}$ and all $n\geq0$,
$|f(q^n w)|\leq M$ and $|(1-q)\zeta| <M^{-1}$.
\end{proposition}
\begin{pf}
This is proved in \cite{BorCor}, Proposition 3.2.8. The proof amounts
to reordering the sums defining $\mu_k$ and recognizing a Fredholm determinant.
\end{pf}

We may replace the space $L^2(\mathbb{Z}_{>0}\times C_{\mathbb{A}})$ by
$L^2(C_{\mathbb{A}})$ via the following Mellin--Barnes representation.

%
\begin{figure}

\includegraphics{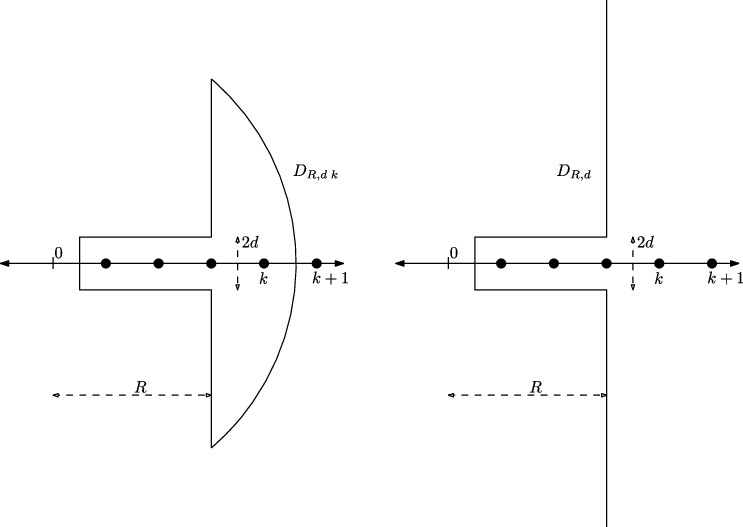}

\caption{Left: The contour $D_{R,d;k}$; Right: The contour $D_{R,d}$.}\label{DRdcontour}
\end{figure}
%
%
\begin{lemma}\label{gammasumlemma}
For all functions $f$ which satisfy the conditions below, we have the
identity that for $\zeta\in\{\zeta\dvtx |\zeta|<1, \zeta\notin\mathbb
{R}_{+}\}$:
%
%
\begin{equation}
\label{gammares} \sum_{n=1}^{\infty} f
\bigl(q^n\bigr) \zeta^n = \frac{1}{2\pi\iota} \int
_{C_{1,2,\ldots}} \Gamma(-s)\Gamma(1+s) (-\zeta)^s f
\bigl(q^s\bigr) \,ds,
\end{equation}
where the infinite contour $C_{1,2,\ldots}$ is a negatively oriented
contour which encloses $1,2,\ldots$ and no poles of $f(q^s)$ (e.g.,
$C_{1,2,\ldots} = \frac{1}{2} + \iota\mathbb{R}$ oriented from
$\frac
{1}{2}-\iota\infty$ to $\frac{1}{2}+\iota\infty$), and $z^s$ is
defined with respect to a branch cut along $z\in\mathbb{R}^{-}$. For the
above equality to be valid, the left-hand side must converge, and the
right-hand side integral must be able to be approximated by integrals
over a sequence of finite contours $C_{k}$ which enclose the poles at
$1,2,\ldots, k$ and which partly coincide with $C_{1,2,\ldots}$ in
such a way that the integral along the symmetric difference of the
contours $C_{1,2,\ldots}$ and $C_{k}$ goes to zero as $k$ goes to infinity.
\end{lemma}

\begin{pf}
The identity follows from $\operatorname{Res}_{s=k}\Gamma(-s)\Gamma
(1+s) = (-1)^{k+1}$.
\end{pf}

%
\begin{definition}\label{DRd}
The infinite contour $D_{R,d}$ is defined as follows. $D_{R,d}$ goes by
straight lines from $R-\iota\infty$, to $R-\iota d$, to $1/2 - \iota
d$, to $1/2+\iota d$, to $R+\iota d$, to $R+\iota\infty$. See
Figure~\ref{DRdcontour}\vadjust{\goodbreak} for an illustration. The finite contour
$D_{R,d;k}$ is defined as follows. Let $p, \bar{p}$ be the points [let
$\operatorname{Im}(p)>0$] at which the circle of radius $k+1/2$,
centered at 0,
intersects $D_{R,d}$. Then $D_{R,d;k}$ is the union of the portion of
$D_{R,d}$ inside the circle with reversed orientation, with the arc
from $\bar{p}$ to $p$ (oriented counterclockwise).
\end{definition}

%
\begin{proposition}\label{propmellindet}
Assume $f(w)=g(w)/g(qw)$ for some function $g$. Then the following
formal equality holds:
\[
\det\bigl(I+K_\zeta^1\bigr) = \det\bigl(I+K_\zeta^2
\bigr),
\]
where $\det(I+K_\zeta^1)$ is given in Proposition~\ref{gendetprop}
and where $\det(I+K_\zeta^2)$ is the formal Fredholm determinant
expansion of
$
K_\zeta^2\dvtx L^2(C_{\mathbb{A}})\to L^2(C_{\mathbb{A}})$.
The operator $K_\zeta^2$ is defined in terms of its integral kernel
\[
K_\zeta^2\bigl(w,w'\bigr) =
\frac{1}{2\pi\iota}\int_{C_{1,2,\ldots}} \Gamma (-s)\Gamma(1+s)
\bigl(-(1-q)\zeta \bigr)^s \frac{g(w)}{g(q^s w)} \frac{1}{q^s w-w'}\,ds.
\]

The above identity holds numerically if $\det(I+K_\zeta^1)$ is a
convergent Fredholm expansion and if $C_{1,2,\ldots}$ is chosen as
$D_{R,d}$ with $d>0$ and $R>0$ such that
\[
\mathop{\inf_{w,w'\in C_{\mathbb{A}}}}_{k\in\mathbb{Z}_{>0}, s\in
D_{R,d;k}} \bigl|q^s
w-w'\bigr|>0\quad\mbox{and}\quad \mathop{\sup_{w,w'\in C_{\mathbb{A}}}}_{k\in\mathbb{Z}_{>0}, s\in
D_{R,d;k}}
\biggl\llvert \frac{g(w)}{g(q^sw)}\biggr\rrvert <\infty.
\]
In that case the function $\zeta\mapsto\det(I+K_\zeta^2)$ is
analytic for all $\zeta\notin\mathbb{R}_{+}$.
\end{proposition}
\begin{pf}
This result can readily be extracted from the proof of \cite{BorCor}
Theorem~3.2.11. In fact, the strong analytic bounds which we require
can be significantly relaxed, however, as they will be sufficient for
our purposes, we do not explore this.
\end{pf}

We say that Fredholm determinants similar to $\det(I+K_\zeta^2)$ are
of Mellin--Barnes-type [because (\ref{gammares}) is a basic tool for
classical Mellin--Barnes integrals].

\subsection{Cauchy-type determinants}\label{seccauchy}
Instead of deforming contours so as to encounter the $z_A=qz_B$ poles,
we may deform our contours to only encounter the pole at 0. The residue
calculus becomes easier and the resulting sum of contour integrals is
indexed by partitions with just a single row (equivalently by
nonnegative integers).

%
\begin{definition}\label{def:muktildedef}
For a meromorphic function $f(z)$ and $k\geq1$ set $\mathbb{A}$ to be a
fixed set of poles of $f$ and assume that $q^m\mathbb{A}$ is disjoint from
$\mathbb{A}$ for all $m\geq1$. Define
%
%
\begin{equation}
\label{muktildedef} \tilde\mu_k=\frac{(-1)^k q^{k(k-1)/2}}{(2\pi\iota)^k} \int\cdots\int\prod
_{1\leq A<B\leq k} \frac{z_A-z_B}{z_A-qz_B} \prod
_{i=1}^{k} f(z_i) \frac{dz_i}{z_i},
\end{equation}
where the integration contour for $z_A$ contains $\{q z_B\}_{B> A}$,
the fixed set of poles $\mathbb{A}$ of $f(z)$ \textit{and} 0, but no
other poles.
\end{definition}

Notice that $\mu_k$ and $\tilde\mu_k$ differ only by the inclusion
of 0 in the contour for $\tilde\mu_k$. They can be related via the following.

%
\begin{proposition}\label{muandmutildeprop}
Assume $f(0)=1$. Then
\[
\tilde{\mu}_k = (-1)^k q^{k(k-1)/2} \sum
_{j=0}^{k} \pmatrix{k\cr j}_{q^{-1}}(-1)^{j}
q^{-j(j-1)/2} \mu_j.
\]
\end{proposition}
\begin{pf}
This is proved in \cite{BorCor}, Proposition 3.2.5.
\end{pf}

If, for instance, we assume now that $\mathbb{A}$ contains all poles of
$f$, then we can deform the contours in (\ref{muktildedef}) to all lie
on a single, large circle. The following symmetrization proposition
then applies.

%
\begin{proposition}\label{mukproplarge}
If the contours of integration in (\ref{muktildedef}) can be deformed
(without passing any poles) to all coincide with a contour $\widetilde
C_{\mathbb{A}}$, then
%
%
\begin{equation}
\label{mukproplargeeqn} \qquad\tilde\mu_k = \frac{k_q!}{k!} \frac{(1-q^{-1})^{k}}{(2\pi\iota
)^{k}}
\int_{\widetilde C_{\mathbb{A}}} \cdots\int_{\widetilde C_{\mathbb{A}}} \det \biggl[
\frac{1}{w_i q^{-1}-w_j} \biggr]_{i,j=1}^{k} \prod
_{j=1}^{k} f(w_j) \,dw_j.
\end{equation}
\end{proposition}
\begin{pf}
This is proved in \cite{BorCor}, Proposition 3.2.2.
\end{pf}

%
\begin{proposition}\label{tildemuDetProp}
If\vspace*{1pt} the contours of integration in (\ref{muktildedef}) can be deformed
(without passing any poles) to all coincide with a contour $\widetilde
C_{\mathbb{A}}$, then the following formal identity holds:
\[
\sum_{k\geq0}\tilde{\mu}_k
\frac{\zeta^k}{k_q!} = \det\bigl(I+\zeta \widetilde{K}^1\bigr),
\]
where $\det(I+\widetilde{K}^1)$ is the formal Fredholm determinant
expansion of $
\widetilde K^1\dvtx\break  L^2(\widetilde{C}_{\mathbb{A}}) \to L^2(\widetilde{C}_{\mathbb{A}})
$
defined in terms of its integral kernel
\[
\widetilde{K}^1\bigl(w,w'\bigr) = (1-q)
\frac{f(w)}{qw'-w}.
\]
The above identity is formal, but also holds numerically for $\zeta$
such that the left-hand side converges absolutely and the right-hand
side operator $\widetilde{K}^1$ is trace-class.
\end{proposition}
\begin{pf}
This is proved in \cite{BorCor}, Proposition 3.2.9.
\end{pf}

%
\begin{remark}\label{k1k2}
By considering the Fredholm series expansion [whose terms are given by
(\ref{mukproplargeeqn})], it is clear that since $f$ arises
multiplicatively, it can be paired either with $w_i$ or $w_j$ in the
Cauchy determinant. As a consequence, it follows that
\[
\det\bigl(I+\zeta\widetilde{K}^1\bigr) = \det\bigl(I+\zeta
\widetilde{K}^2\bigr),
\]
where
$
\widetilde K^1\dvtx L^2(\widetilde{C}_{\mathbb{A}}) \to L^2(\widetilde{C}_{\mathbb{A}})
$
is defined in terms of its integral kernel
\[
\widetilde{K}^2\bigl(w,w'\bigr) = (1-q)\frac{f(w)}{qw-w'}.
\]
\end{remark}

We call Fredholm determinants of this form \textit{Cauchy} type.

\subsection{Application to \mbox{$q$-}TASEP}\label{secqtasepapp}
The following theorems about \mbox{$q$-}TASEP are applications of the
manipulations of the previous section. The required estimates necessary
to make these numerical equalities are provided in \cite{BorCor}.

\subsubsection{Mellin--Barnes-type Fredholm determinant for \mbox{$q$-}TASEP}
%
\begin{theorem}\label{PlancherelfredThm}
Fix $0<q<1$ and $n\geq1$. Fix $0<\delta<1$ and $a_1,\ldots, a_n$
such that for all $i$, $a_i>0$ and $|a_i -1|\leq d$ for some constant
$d <\frac{1-q^{\delta}}{1+q^{\delta}}$. Consider \mbox{$q$-}TASEP with step
initial data and jump parameters $a_i$. Then for all $t\in\mathbb
{R}_{+}$ and
$\zeta\in\mathbb{C}\setminus\mathbb{R}_{+}$, the following
characterizes the
distribution of $x_n(t)$:
%
%
\begin{equation}
\label{thmlaplaceeqn} \mathbb{E} \biggl[ \frac{1}{ (\zeta q^{x_n(t)};q )_{\infty
}} \biggr] = \det
\bigl(I+K^{q\mbox{-}\mathrm{TASEP}}_{\zeta}\bigr),
\end{equation}
where $\det(I+K^{q\mbox{-}\mathrm{TASEP}}_{\zeta})$ is the Fredholm
determinant of
$
K_{\zeta}\dvtx L^2(C_{a})\to L^2(C_{a})
$
for $C_a$ a positively oriented circle $|w-1|=d$. The operator
$K_{\zeta}$ is defined in terms of its integral kernel
\[
K^{q\mbox{-}\mathrm{TASEP}}_{\zeta}\bigl(w,w'\bigr) =
\frac{1}{2\pi\iota}\int_{-\iota\infty
+ \delta}^{\iota\infty+\delta} \Gamma(-s)
\Gamma(1+s) \bigl(-q^{-n}\zeta \bigr)^s \frac{g(w)}{g(q^s w)}
\frac{1}{q^s w - w'}\,ds,
\]
where
%
%
\begin{equation}
\label{gwwprimeeqn} g(w)= \prod_{m=1}^{n}
\frac{1}{(w/a_m;q)_{\infty}} e^{-t w}.
\end{equation}
%
\end{theorem}

\begin{pf}
This is proved in \cite{BorCor}, Theorem 3.2.11. A similar approach is
described in its entirety in the proof of Theorem~\ref{ASEPMellinBarnesThm}, for ASEP.
\end{pf}

The above is an $e_q$-Laplace transform and can be inverted via
Proposition~\ref{qlaplaceinverse}. Since $x_n(t)$ is supported on $\{
-n,-n+1,\ldots\}$, in order\vspace*{1pt} to apply Proposition~\ref
{qlaplaceinverse} it is necessary to shift everything by $n$. Let
$\hat{f}^{q}(\zeta) = \det(I+K^{q\mbox{-}\mathrm{TASEP}}_{\zeta})$
and redefine $C_m$
to encircle the poles $\zeta=q^{-M}$ for $-n\leq M \leq m-n$. Under
these modifications, Proposition~\ref{qlaplaceinverse} gives $\mathbb{P}
(x_n(t) = m)$.

\subsubsection{Cauchy-type Fredholm determinant for \mbox{$q$-}TASEP}

%
\begin{theorem}\label{largeconThm}
Fix $0<q<1$, $n\geq1$ and $a_1,\ldots, a_n$ such that for all $i$,
\mbox{$a_i>0$}. Consider \mbox{$q$-}TASEP with step initial data and jump parameters
$a_i>0$ for all $i\geq1$. Let $x_n(t)$ by the location of particle $n$
at time $t$. Then for all $\zeta\in\mathbb{C}\setminus\{q^{-i}\}
_{i\in
\mathbb{Z}_{\geq0}}$
%
%
\begin{equation}
\label{thmlaplaceeqnLARGE} \mathbb{E} \biggl[ \frac{1}{ (\zeta q^{x_n(t)+n};q
)_{\infty
}} \biggr] = \frac{\det(I+\zeta\widetilde K^{q\mbox{-}\mathrm{TASEP}})}{(\zeta;q)_{\infty}},
\end{equation}
where\vspace*{1.5pt} $\det(I+\zeta\widetilde K^{q\mbox{-}\mathrm{TASEP}})$ is an entire
function of $\zeta
$ and is the Fredholm determinant of
$
\widetilde K^{q\mbox{-}\mathrm{TASEP}}\dvtx L^2(\widetilde{C}_{a}) \to L^2(\widetilde{C}_{a})
$
defined in terms of its integral kernel
\[
\widetilde K^{q\mbox{-}\mathrm{TASEP}}\bigl(w,w'\bigr) =\frac{f(w)}{q w' - w}
\]
with
\[
f(w) = \Biggl(\prod_{m=1}^{n}
\frac{a_m}{a_m-w} \Biggr) \exp\bigl\{(q-1)t w\bigr\}
\]
and $\widetilde{C}_{a}$ a star-shaped contour with respect to 0 (i.e., it
strictly contains 0 and every ray from 0 crosses $\widetilde{C}_{a}$
exactly once) contour containing $a_1,\ldots, a_n$.
\end{theorem}
\begin{pf}
This is proved in \cite{BorCor}, Theorem 3.2.16. A similar approach is
described in its entirety in the proof of Theorem~\ref{ASEPcauchy},
for ASEP.
\end{pf}
The above shows that $\det(I+\zeta\widetilde K^{q\mbox{-}\mathrm{TASEP}})/
(\zeta;q)_{\infty}$ equals the $e_q$-Laplace transform of $q^{x_n(t)+n}$.

\section{Duality and the nested contour integral ansatz for ASEP}\label{sec4}

The asymmetric simple exclusion process (ASEP) was introduced by
Spitzer \cite{Spitzer} in 1970 and also arose in biology in the work
of MacDonald, Gibbs and Pipkin \cite{MGP} in 1968. Since then, it has
become a central object of study in interacting particle systems and
nonequilibrium statistical mechanics.

The ASEP is a continuous time Markov process with state $\eta(t)=\break \{
\eta_x(t)\}_{x\in\mathbb{Z}}\in\{0,1\}^{\mathbb{Z}}$ at time
$t\geq0$. The $\eta
_x(t)$ are called \textit{occupation} variables and can be thought of as
the indicator function for the event that a particle is at site $x$ at
time $t$. The dynamics of this process is specified by nonnegative
real numbers $p\leq q$ (normalized by $p+q=1$) and uniformly bounded
(from infinity and zero) rate parameters $\{a_{x}\}_{x\in\mathbb
{Z}}$. For
each pair of neighboring sites $(y,y+1)$, the following exchanges
happen in continuous time:
\begin{eqnarray*}
\eta&\mapsto&\eta^{y,y+1}\qquad\mbox{at rate } a_{y}p\qquad\mbox{if }(\eta_y,\eta_{y+1})=(1,0),
\\
\eta&\mapsto&\eta^{y,y+1}\qquad\mbox{at rate } a_{y}q\qquad\mbox{if }(\eta_y,\eta_{y+1})=(0,1),
\end{eqnarray*}
where $\eta^{y,y+1}$ denotes the state in which the value of the
occupation variables at site $y$ and $y+1$ are switched, and all other
variables remain unchanged. All exchanges occur independently of each
other according to exponential clocks. These dynamics are called the
\textit{ASEP occupation process} and are defined in terms of the generator
$L^{\mathrm{occ}}$ which acts on local functions $f\dvtx \{0,1\}^{\mathbb
{Z}}\to\mathbb{R}$ by
%
%
\begin{eqnarray}\label{ASEPgen}
&& \bigl(L^{\mathrm{occ}}f\bigr) (\eta)
\nonumber\\[-8pt]\\[-8pt]
&&\qquad = \sum
_{y\in\mathbb{Z}} a_{y} \bigl[p \eta_y (1-\eta
_{y+1}) + q(1-\eta_y)\eta_{y+1} \bigr] \bigl(f
\bigl(\eta^{y,y+1}\bigr) - f(\eta ) \bigr).\nonumber
\end{eqnarray}
The existence of a Markov process with this generator is shown, for
example, in~\cite{Lig}.

In terms of particles, the dynamics of ASEP are that each particle
attempts, in continuous time, to jump right at rate $pa_{y}$ and to the
left at rate $q a_{y-1}$ (presently the particle is at position $y\in
\mathbb{Z}$), subject to the exclusion rule that says that jumps are
suppressed if the destination site is occupied. We assume $p\leq q$
(drift to the left) and define $\gamma:=q-p\geq0$ and $\tau:=p/q\leq1$.

The ASEP preserves the number of particles, thus we can consider ASEP
with $k$ particles as a process on the particle locations. Define
\[
\widetilde W^{k}= \bigl\{\vec{x}=(x_1,x_2,
\ldots,x_k)\in\mathbb {Z}^k\dvtx  x_1<
x_2<\cdots<x_k \bigr\}
\]
and $\vec{x}_i^{\pm}=(x_1,\ldots,x_{i-1},x_i\pm1,x_{i+1},\ldots,x_k)$. Then $\vec{x}(t)= (x_1(t)<x_2(t)<\cdots< x_k(t) )\in
\widetilde W^{k}$ denotes the locations of the $k$ particles of ASEP at
time $t$.

%
\begin{figure}[b]

\includegraphics{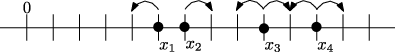}

\caption{ASEP with four particles: $x_1=5$, $x_2=6$, $x_3=9$ and
$x_4=11$. The first two particles form a cluster, and the third and
fourth form two separate clusters. The arrows represent admissible moves.}\label{ASEPfig}
\end{figure}

In order to describe the generator of ASEP in terms of particle
locations, it is convenient to introduce particle cluster notation
(see Figure~\ref{ASEPfig}). A
\textit{cluster} is a collection of particles next to each other:
$x_{i}=x_{i+1}-1=\cdots=x_{i+j}-j$. There is a unique way of
dividing the particles $\vec{x}$ into clusters so that each cluster is
separated by a buffer of at least one site: let $c(\vec{x})$ be the
number of such clusters, $\ell(\vec{x})=(\ell_1,\ldots, \ell_c)$
be the collection of labels of the left-most particles of each cluster,
and $r(\vec{x})=(r_1,\ldots, r_c)$ be the collection of labels for
the right-most particles of each cluster. For instance, if $k=4$ and
$x_1=5$, $x_2=6$, $x_3=9$, $x_4=11$ then $c(\vec{x})=3$, $\ell(\vec
{x})=(1,3,4)$ and $r(\vec{x}) =(2,3,4)$.

For $k\geq1$, the \textit{ASEP particle process} generator acts on
bounded functions $f\dvtx \widetilde W^{k}\to\mathbb{R}$ by
\[
\bigl(L^{\mathrm{part}}f\bigr) (\vec{x}) = \sum_{i\in\ell(\vec{x})}
a_{x_i-1}p \bigl[f\bigl(\vec{x}_{i}^{-}\bigr) - f(
\vec{x}) \bigr] + \sum_{i\in r(\vec{x})} a_{x_i}q
\bigl[f\bigl(\vec{x}_{i}^{+}\bigr) - f(\vec{x}) \bigr].
\]

We will consider initial configurations for ASEP in which there is at
most a finite number of nonzero occupation variables (i.e., particles)
to the left of the origin---we call these \textit{left-finite} initial
data. When ASEP is initialized with left-finite initial data, its state
remains left-finite for all time (simply because it will always have a
left-most particle). We will use $\mathbb{E}^{\eta}$ and $\mathbb
{P}^{\eta}$ to
denote expectation and probability (resp.) of the Markov
dynamics on occupation variables with initial data $\eta$ (and
likewise $\mathbb{E}^{\vec{x}}$ and $\mathbb{P}^{\vec{x}}$ for the Markov
evolution on particle locations with initial data~$\vec{x}$). When the
initial data is itself random, we write $\mathbb{E}$ and $\mathbb{P}$
to denote
expectation and probability (resp.) of the Markov dynamics as
well as the initial data. We also use $\mathbb{E}$ and $\mathbb{P}$
when the
initial data is otherwise specified.

\subsection{Duality}

Recall that $\tau=p/q\leq1$ by assumption and define the following
functions of a state $\eta$:
%
%
\begin{eqnarray}\label{NQdef}
N_x(\eta) &=& \sum_{y=-\infty}^{x}\eta_y,\qquad Q_x(\eta)= \tau ^{N_x(\eta)},
\nonumber\\[-8pt]\\[-8pt]
\widetilde{Q}_x(\eta) &=& \frac{Q_x(\eta)-Q_{x-1}(\eta)}{\tau-1} = \tau^{N_{x-1}(\eta)}
\eta_x.\nonumber
\end{eqnarray}

The following result shows that (with general bond rate parameters) the
ASEP occupation process and the ASEP particle process with the role of
$p$ and $q$ reversed, are dual with respect to a given function $\widetilde
{H}$. This is sometimes called self-duality, despite the fact that the
processes involved are independent and defined with respect to
different state spaces.

%
\begin{theorem}\label{thmASEPdualitytilde}
Fix nonnegative real numbers \mbox{$p\leq q$} (normalized by\break \mbox{$p+q=1$}) and
uniformly bounded (from infinity and zero) bond rate parameters $\{
a_{x}\}_{x\in\mathbb{Z}}$. For any $k\geq1$, the ASEP occupation\vspace*{1.5pt} process
$\eta(t)$ with state space $\{0,1\}^{\mathbb{Z}}$, and the ASEP particle
process $\vec{x}(t)$ with state space $\widetilde W^{k}$ and the role
of~$p$~and~$q$ reversed, are dual with respect to
\[
\widetilde H(\eta,\vec{x}) =\prod_{i=1}^{k}
\widetilde Q_{x_i}(\eta).
\]
\end{theorem}

If we restrict to $a_x\equiv1$ we can prove another ASEP duality.

%
\begin{theorem}\label{thmASEPduality}
Fix nonnegative real numbers $p\leq q$ (normalized by $p+q=1$) and bond
rate parameters $a_{x}\equiv1$. For any $k\geq1$, the ASEP occupation
process $\eta(t)$ with state\vspace*{1.5pt} space $\{0,1\}^{\mathbb{Z}}$, and the ASEP
particle process $\vec{x}(t)$ with state space $\widetilde W^{k}$ and the
role of $p$ and $q$ reversed, are dual with respect to
\[
H(\eta,\vec{x}) =\prod_{i=1}^{k}
Q_{x_i}(\eta).
\]
\end{theorem}

Recall that the concept of duality is given in Definition~\ref
{dualdef}. A few remarks are in order.
%
\begin{remark}
For $p<q$, both forms of duality are trivial for initial data which is
not left-finite, since then $Q_x(\eta)\equiv0$ and likewise $\widetilde
Q_x(\eta) \equiv0$. By working with a height function, rather than
$N_x(\eta)$ it is likely possible to extend consideration to
left-infinite initial data. We do not pursue this here.
\end{remark}

%
\begin{remark}\label{4.4}
For the symmetric simple exclusion process ($p=q=1/2$), the duality
from Theorem~\ref{thmASEPdualitytilde} has been known for some time
(see \cite{Lig}, Chapter~8, Theorem 1.1). For $p<q$, the result of
Theorem~\ref{thmASEPdualitytilde} was discovered by Sch\"{u}tz \cite
{Schutz} in the late 1990s via a spin chain representation of ASEP (the
result was stated for all $a_{x}\equiv1$, though the proof is easily
extended to general $a_{x}$). The approach used therein to show duality
was computationally based on a $U_q(sl_2)$ symmetry for ASEP \cite
{SanSch}. Our proof proceeds directly via the Markov dynamics, without
any use of, or reference to, the $U_q(sl_2)$. Even though in our
applications we quickly set $a_x\equiv1$, it is both useful (in simply
the proof) and informative (in showing that duality is weaker than
integrability) to prove our result for general $a_x$.

The duality of Theorem~\ref{thmASEPduality} appears to be new. It
does not seem possible to extend it to general $a_x$. For instance,
when $k=1$, as a function of the process $\eta(t)$, $H(\eta(t),x)$
only changes value when a particle moves across the bond between $x$
and $x+1$. This only involves the rate $a_x$. On the other hand, as a
process of $x(t)$, $H(\eta,x(t))$ changes value when a particle moves
across either the bond between $x-1$ and $x$, or the bond between $x$
and $x+1$. This involves the rates $a_{x-1}$~and~$a_x$. Hence, the two
sides can only match when $a_{x-1}=a_{x}$.
\end{remark}

%
\begin{remark}
G\"{a}rtner \cite{G} observed that ASEP respected a microscopic (i.e.,
particle-level) version of the Hopf--Cole transform (see, e.g., the
review \cite{CorwinReview}). This observation is equivalent to the
$k=1$, $a_x\equiv1$ case of the duality given in Theorem~\ref
{thmASEPduality}. It says that
\[
dQ_x\bigl(\eta(t)\bigr) = \bigl(p Q_{x-1}\bigl(\eta(t)
\bigr) +q Q_{x+1}\bigl(\eta(t)\bigr) - Q_x\bigl(\eta(t)
\bigr) \bigr)\,dt + Q_x\bigl(\eta(t)\bigr) \,dM(t),
\]
where $dM(t)$ is an explicit martingale. This is a particular
semidiscrete SHE (different than the one coming from \mbox{$q$-}TASEP,
Definition~\ref{def:semidiscSHE}) with a somewhat involved noise (the
martingale is not exactly a discrete space--time white noise).
A~Feynman--Kac representation for this equation shows that $Q_x(\eta
(t))$ can be thought of as a polymer partition function with respect to
an environment defined by the martingale. Therefore, Theorem~\ref
{thmASEPduality} can be thought of as a version of the polymer replica
approach (see Section~\ref{secrepappr}).
\end{remark}

The proof of the two duality theorems boils down to two propositions
which we now state and prove. After this, we prove the theorems.

%
\begin{proposition}\label{propLequivtilde}
Fix nonnegative real numbers $p\leq q$ (normalized by $p+q=1$) and
uniformly bounded (from infinity and zero) bond rate parameters $\{
a_{x}\}_{x\in\mathbb{Z}}$. Then, with $\eta$, $\vec{x}$, and
$\widetilde H(\eta,\vec{x})$ defined in Theorem~\ref{thmASEPdualitytilde},
%
%
\begin{equation}
\label{starstar26} L^{\mathrm{occ}}\widetilde H(\eta,\vec{x}) = L^{\mathrm{part}}\widetilde
H(\eta,\vec{x}),
\end{equation}
where the generator $L^{\mathrm{occ}}$ acts in the $\eta$ variable
and the
generator $L^{\mathrm{part}}$ acts in the $\vec{x}$ variable.
\end{proposition}
\begin{pf}
We will first prove the desired result for a single cluster
configuration $\vec{x}=(x,x+1,\ldots, x+\ell)$ and then easily
deduce it for general $\vec{x}\in\widetilde W^{k}$. For the single cluster
$\vec{x}$, by the definition of $L^{\mathrm{occ}}$,
%
\[
L^{\mathrm{occ}}\widetilde H(\eta,\vec{x}) = \sum_{i=-1}^{\ell}
a_{x+i} A_{x+i}(\eta),
\]
where
\[
A_{y}(\eta) = \bigl(p \eta_y (1-\eta_{y+1}) +
q(1-\eta_y)\eta _{y+1} \bigr) \bigl[\widetilde H\bigl(
\eta^{y,y+1},\vec{x}\bigr) - \widetilde H(\eta,\vec{x}) \bigr].
\]
We now compute the $A_y$'s explicitly. Recall the notations introduced
in (\ref{NQdef}). There are three different types of $A_y$ that must
be considered: (1) $A_{x-1}(\eta)$; (2) $A_{x+i}(\eta)$ for $0\leq
i\leq\ell-1$; (3) $A_{x+\ell}(\eta)$.
\begin{longlist}[(2)]
\item[(1)] Consider $A_{x-1}(\eta)$. We may rewrite
\[
\widetilde H(\eta,\vec{x}) = \tau^{N_{x-2}(\eta)} \tau^{\eta_{x-1}}
\eta_x \prod_{j=1}^{\ell}
\widetilde{Q}_{x+j}(\eta)
\]
and
\[
\widetilde H\bigl(
\eta^{x-1,x},\vec{x}\bigr) = \tau^{N_{x-2}(\eta)} \tau^{\eta_{x}}
\eta_{x-1} \prod_{j=1}^{\ell}
\widetilde {Q}_{x+j}(\eta).
\]
Thus,
%
%
\begin{eqnarray}\label{Axm1}
A_{x-1}(\eta)
&=& \tau^{N_{x-2}(\eta)}
\prod _{j=1}^{\ell} \widetilde {Q}_{x+j}(\eta) \bigl(p
\eta_{x-1}(1-\eta_{x}) + q(1-\eta_{x-1})\eta
_x \bigr)
\nonumber\\[-8pt]\\[-8pt]
&&{}\times  \bigl[\tau^{\eta_x} \eta_{x-1} -
\tau^{\eta_{x-1}}\eta _x \bigr].\nonumber
\end{eqnarray}
\item[(2)] Consider $A_{x+i}(\eta)$ for $0\leq i\leq\ell-1$. We may rewrite
\[
\widetilde H(\eta,\vec{x}) = \tau^{2N_{x+i-1}(\eta) + \eta_{x+i}} \eta _{x+i}
\eta_{x+i+1} \mathop{\prod_{j=0}}_{j\neq i,i+1}^{\ell}
\widetilde{Q}_{x+j}(\eta)
\]
and
\[
\widetilde H\bigl(\eta^{x+i,x+i+1},\vec{x}\bigr) = \tau^{2N_{x+i-1}(\eta)+\eta
_{x+i+1}}
\eta_{x+i+1}\eta_{x+i} \mathop{\prod
_{j=0}}_{j\neq
i,i+1}^{\ell} \widetilde{Q}_{x+j}(
\eta).
\]
Thus,
%
%
\begin{eqnarray}
\label{Axi} A_{x+i}(\eta) &=& \tau^{2N_{x+i-1}(\eta)} \Biggl( \mathop{
\prod_{j=0}}_{j\neq i,i+1}^{\ell}
\widetilde{Q}_{x+j}(\eta) \Biggr)
\\
&&{}\times \bigl(p\eta_{x+i}(1-\eta_{x+i+1}) + q(1-\eta
_{x+i})\eta_{x+i+1} \bigr)
\nonumber\\[-8pt]\\[-8pt]
&&{}\times  \bigl[\tau^{\eta_{x+i+1}} -
\tau^{\eta_{x+i}} \bigr]\eta_{x+i}\eta_{x+i+1}.\nonumber
\end{eqnarray}
\item[(3)] Consider $A_{x+\ell}(\eta)$. We may rewrite
\[
\widetilde H(\eta,\vec{x}) = \tau^{N_{x+\ell-1}(\eta)}\eta_{x+\ell} \prod
_{j=0}^{\ell-1} \widetilde{Q}_{x+j}(\eta)
\]
and
\[
\widetilde H\bigl(\eta^{x+\ell,x+\ell+1},\vec{x}\bigr) =
\tau^{N_{x+\ell
-1}(\eta)}\eta_{x+\ell+1} \prod_{j=0}^{\ell-1}
\widetilde {Q}_{x+j}(\eta).
\]
Thus,
%
%
\begin{eqnarray}\label{Axell}
A_{x+\ell}(\eta) &=& \tau^{N_{x+\ell-1}(\eta)} \Biggl(\prod
_{j=0}^{\ell-1} \widetilde{Q}_{x+j}(\eta)
\Biggr)
\\
&&{}\times \bigl(p\eta_{x+\ell}(1-\eta_{x+\ell+1}) + q(1-
\eta_{x+\ell})\eta_{x+\ell+1} \bigr)
\nonumber\\[-8pt]\\[-8pt]
&&{}\times  [\eta_{x+\ell+1} - \eta_{x+\ell} ].\nonumber
\end{eqnarray}
\end{longlist}

Observe that $A_{x+i}(\eta)=0$ for $0\leq i\leq\ell-1$. To see this,
it suffices to consider the four values that the pair $(\eta
_{x+i},\eta_{x+i+1})$ may take: for $(0,0)$ or $(1,1)$
\[
\bigl(p\eta_{x+i}(1-\eta_{x+i+1}) + q(1-\eta_{x+i})
\eta_{x+i+1} \bigr) \bigl[ \tau^{\eta_{x+i}} - \tau^{\eta_{x+i+1}}
\bigr]= 0
\]
and thus $(\ref{Axi})=0$; for $(0,1)$ or $(1,0)$, the factor $\eta
_{x+i}\eta_{x+i+1}=0$ and thus $(\ref{Axi})=0$ again.

The above observation shows that, in fact,
\[
\bigl(L^{\mathrm{occ}}f\bigr) (\eta) = a_{x-1}A_{x-1}(\eta)
+ a_{x+\ell
}A_{x+\ell}(\eta).
\]
In light of equations (\ref{Axm1}) and (\ref{Axell}), we may rewrite
\[
A_{x-1}(\eta) = M(\eta) A'_{x-1}(\eta) \quad\mbox{and}\quad A_{x+\ell}(\eta) = M(\eta) A'_{x+\ell}(
\eta),
\]
where
\[
M(\eta) = \tau^{N_{x-2}(\eta)+N_{x+\ell-1}(\eta)} \prod_{j=1}^{\ell-1}
\widetilde{Q}_{x+j}(\eta)
\]
and
\begin{eqnarray*}
A'_{x-1}(\eta) &=& \eta_{x+\ell} \bigl(p
\eta_{x-1}(1-\eta_{x}) + q(1-\eta_{x-1})
\eta_x \bigr) \bigl[\tau^{\eta_x} \eta_{x-1} - \tau
^{\eta_{x-1}}\eta_x \bigr],
\\
A'_{x+\ell}(\eta) &=& \tau^{\eta_{x-1}}
\eta_x \bigl(p\eta_{x+\ell
}(1-\eta_{x+\ell+1}) + q(1-
\eta_{x+\ell})\eta_{x+\ell+1} \bigr) [\eta_{x+\ell+1} -
\eta_{x+\ell} ].
\end{eqnarray*}

Now turn to the right-hand side of equation (\ref{starstar26}). We may
also factor $M(\eta)$ out from that expression
\begin{eqnarray*}
\mbox{RHS (\ref{starstar26})} &=& M(\eta) \bigl[ a_{x-1} p \eta
_{x-1}\eta_{x+\ell} + a_{x+\ell} q \tau^{\eta_{x-1}}
\eta_x \tau ^{\eta_{x+\ell}}\eta_{x+\ell+1}
\\
&&\hspace*{84pt}{} - (a_{x-1}q + a_{x+\ell}p)\tau ^{\eta_{x-1}}\eta_x \eta_{x+\ell}\bigr].
\end{eqnarray*}
Therefore, for the single cluster case of the proposition, we are left
to prove
\begin{eqnarray*}
&& a_{x-1}A'_{x-1}(\eta) + a_{x+\ell}A'_{x+\ell}(
\eta)
\\
&&\qquad  = a_{x-1} p \eta_{x-1}\eta_{x+\ell} +
a_{x+\ell} q \tau^{\eta_{x-1}}\eta_x \tau^{\eta_{x+\ell}}
\eta_{x+\ell+1}
\\
&&\quad\qquad{} - (a_{x-1}q + a_{x+\ell}p)
\tau^{\eta_{x-1}}\eta_x \eta_{x+\ell}.
\end{eqnarray*}
The above equation is a function of only four occupation variables
$\eta_{x-1},\eta_x,\break \eta_{x+\ell}$ and $\eta_{x+\ell+1}$ and one
can systematically check that for all sixteen combinations of values of
these variables, the above equation is true. In fact, it is even easier
than this since the coefficients of $a_{x-1}$ and $a_{x+\ell}$
coincide separately. For instance, we must show that $A'_{x-1}(\eta) =
\eta_{x+\ell} (p\eta_{x-1} - q \tau^{\eta_{x-1}} \eta
_x )$. There are only four cases of $(\eta_{x-1},\eta_x)$ that
have to be considered and this can be confirmed in one's head
[similarly for $A'_{x+\ell}(\eta)$]. This completes the proof of
Proposition~\ref{propLequivtilde} for $\vec{x}$ with just a single cluster.

For a general $\vec{x}\in\widetilde W^{k}$ there may be many
clusters, each
pair separated by at least one empty site. The terms in $\widetilde H(\eta,x)$ factor into clusters and the generator $L^{\mathrm{occ}}$ acts on
each of
these clusters according to the above proved single cluster result.
This immediately yields the general statement and completes the proof.
\end{pf}

%
\begin{proposition}\label{propLequiv}
Fix nonnegative real numbers $p\leq q$ (normalized by $p+q=1$) and set
all bond rate parameters $a_{x}\equiv1$. Then, with $\eta$, $\vec
{x}$, and $H(\eta,\vec{x})$ defined in Theorem~\ref{thmASEPduality},
\[
L^{\mathrm{occ}}H(\eta,\vec{x}) = L^{\mathrm{part}}H(\eta,\vec{x}),
\]
where the generator $L^{\mathrm{occ}}$ acts in the $\eta$ variable
and the
generator $L^{\mathrm{part}}$ acts in the $\vec{x}$ variable.
\end{proposition}

\begin{pf}
As in the proof of Proposition~\ref{propLequivtilde}, we will first
prove the desired result for a single cluster configuration $\vec
{x}=(x,x+1,\ldots, x+\ell)$ and then easily deduce it for general
$\vec{x}\in\widetilde W^{k}$. For the single cluster $\vec{x}$, by the
definition of $L^{\mathrm{occ}}$,
\[
L^{\mathrm{occ}}H(\eta,\vec{x}) = \sum_{i=0}^{\ell}
A_{x+i}(\eta),
\]
where
\[
A_{y}(\eta) = \bigl(p \eta_y (1-\eta_{y+1}) +
q(1-\eta_y)\eta _{y+1} \bigr) \bigl[H\bigl(
\eta^{y,y+1},\vec{x}\bigr) - H(\eta,\vec{x}) \bigr].
\]
By grouping terms, this may be rewritten as
\begin{eqnarray*}
A_{x+i}(\eta) &=& Q_{x+i-1}(\eta) \Biggl(\mathop{\prod
_{j=0}}_{j\neq i}^{\ell}
Q_{x+j}(\eta) \Biggr) \bigl(p\eta _{x+i}(1-
\eta_{x+i+1}) + q\eta_{x+i+1}(1-\eta_{x+i}) \bigr)
\\
&&{}\times \bigl[\tau^{\eta_{x+i+1}}-\tau^{\eta_{x+i}} \bigr]
\\
&=& \prod_{j=0}^{\ell} Q_{x+j-1}(
\eta) \mathop{\prod_{j=0}}_{j\neq i}^{\ell}
\tau^{\eta_{x+j}} \cdot \bigl(p+q\tau ^{\eta_{x+i}} \tau^{\eta_{x+i+1}} -
\tau^{\eta_{x+i}} \bigr).
\end{eqnarray*}
In order to get the second line above, we utilized the definition of
$Q_{x+i-1}(\eta)$ and separately the fact (which can readily be
checked) that for the four possible pairs of values that $(\eta
_{x+i},\eta_{x+i+1})$ can take
\begin{eqnarray*}
&& \bigl(p\eta_{x+i}(1-\eta_{x+i+1}) + q\eta_{x+i+1}(1-
\eta_{x+i}) \bigr) \bigl[\tau^{\eta_{x+i+1}}-\tau^{\eta_{x+i}}
\bigr]
\\
&&\qquad = p+q\tau ^{\eta_{x+i}} \tau^{\eta_{x+i+1}} - \tau^{\eta_{x+i}}.
\end{eqnarray*}

Recall that we seek to show
\begin{eqnarray*}
&& \sum_{i=0}^{\ell} A_{x+i}(\eta)
\\
&&\qquad=
p Q_{x-1}(\eta) \prod_{j=1}^{\ell}
Q_{x+j}(\eta) + q Q_{x+\ell+1}(\eta) \prod
_{j=0}^{\ell-1} Q_{x+j}(\eta) - \prod
_{j=0}^{\ell} Q_{x+j}(\eta).
\end{eqnarray*}
Factoring out $ (\prod_{j=0}^{\ell} Q_{x+j-1}(\eta) )$
from both sides we are left to prove
%
%
\begin{eqnarray}\label{eqnqleftprove}
&& \Biggl(\sum_{i=0}^{\ell}
\mathop{\prod_{j=0}}_{j\neq i}^{\ell}
\tau^{\eta_{x+j}} \Biggr) \bigl(p+q\tau^{\eta_{x+i}}\tau^{\eta
_{x+i+1}}-
\tau^{\eta_{x+i}} \bigr)
\nonumber\\[-8pt]\\[-8pt]
&&\qquad = p \prod_{j=1}^{\ell}
\tau ^{\eta_{x+j}} + q \prod_{j=0}^{\ell+1}
\tau^{\eta_{x+j}} - \prod_{j=0}^{\ell}
\tau^{\eta_{x+j}}.\nonumber
\end{eqnarray}
The terms in the left-hand side of the above expression can be grouped as
\[
p \prod_{j=1}^{\ell} \tau^{\eta_{x+j}} +
\sum_{i=1}^{\ell} \mathop{\prod
_{j=0}}_{j\neq i}^{\ell} \tau^{\eta_{x+j}}
\bigl(p+q \tau^{2\eta_{x+i}} - \tau^{\eta_{x+i}} \bigr) + q \prod
_{j=0}^{\ell+1} \tau^{\eta_{x+j}} - \prod
_{j=0}^{\ell} \tau^{\eta_{x+j}}.
\]
We may now utilize the easily checked identity that for $\eta\in\{
0,1\}$,
\[
p+q\tau^{2\eta}-\tau^{\eta}=0,
\]
to see that the above expression reduces to the right-hand side of
(\ref{eqnqleftprove}), thus completing the proof\vspace*{1pt} of Proposition~\ref{propLequiv} for $\vec{x}$ with just a single cluster.

From a general $\vec{x}\in\widetilde W^{k}$, there might be many clusters,
each pair separated by at least one empty site. The terms in $H(\eta,x)$ factor into clusters and the generator $L^{\mathrm{occ}}$ acts on
each of these clusters according to the above proved single cluster
result. This immediately yields the general statement and completes the proof.
\end{pf}

Before giving the proof of Theorem~\ref{thmASEPdualitytilde} we
define the following system of~ODEs.

%
\begin{definition}\label{WASEPtrueevol}
We say that $\tilde h(t;\vec{x})\dvtx \mathbb{R}_{+}\times\widetilde
W^{k} \to\mathbb{R}$
solves the
\textit{true evolution equation} with initial data $\tilde h_0(\vec{x})$ if:
\begin{longlist}[(2)]
\item[(1)] For all $\vec{x}\in\widetilde W^{k}$ and $t\in\mathbb{R}_{+}$,
\[
\frac{d}{dt} \tilde h(t;\vec{x}) = L^{\mathrm{part}} \tilde h(t;\vec{x});
\]
\item[(2)] There exist constants $c,C>0$ and $\delta>0$ such that for all
$\vec{x}\in\widetilde W^{k}$, $t\in[0,\delta]$
\[
\bigl|\tilde h(t;\vec{x})\bigr| \leq C e^{c\|\vec{x}\|_1};
\]
%
%
\item[(3)] As $t\to0$, $\tilde h(t;\vec{x})$ converges pointwise to
$\tilde h_0(\vec{x})$.
\end{longlist}
\end{definition}

%
\begin{proposition}\label{asepuniq}
Assume that there exists constants $c,C>0$ such that for all $\vec
{x}\in\widetilde W^{k}$,
%
%
\begin{equation}
\label{ineq1a} \bigl|\tilde h_0(\vec{x})\bigr| \leq C e^{c\|\vec{x}\|_1}.
\end{equation}
Then there exists a unique solution to the system of ODEs given in
Definition~\ref{WASEPtrueevol} which is given by
%
%
\begin{equation}
\label{propsoln} \tilde h(t;\vec{x}):= \mathbb{E}^{-t;\vec{x}}
\bigl[h_0\bigl(\vec {x}(0)\bigr) \bigr],
\end{equation}
where the expectation is with respect to the ASEP particle process
$\vec{x}(\cdot)$ started at time $-t$ in configuration $\vec{x}$.
\end{proposition}

This existence and uniqueness result is proved in Appendix~\ref{appenduniq}. We use this result presently in the proof of Theorem~\ref
{thmASEPdualitytilde}, and also later in the proof of Theorem~\ref
{QtildeInt}. It is in the second application of this result that we
fully utilize the weakness of conditions 2 and 3 in Definition~\ref
{WASEPtrueevol}.

\begin{pf*}{Proof of Theorem~\ref{thmASEPdualitytilde}}
We follow the same approach as in the proof of Theorem~\ref
{thmqtasepduality}. Our present theorem follows from Proposition~\ref
{propLequivtilde} along with Proposition~\ref{asepuniq}. Observe that
\begin{eqnarray*}
\frac{d}{dt} \mathbb{E}^{\eta} \bigl[\widetilde H\bigl(\eta(t),\vec
{x}\bigr) \bigr] &=& L^{\mathrm{occ}}\mathbb{E}^{\eta} \bigl[\widetilde H
\bigl(\eta(t),\vec {x}\bigr) \bigr]
\\
&=& \mathbb{E}^{\eta} \bigl[L^{\mathrm{occ}}\widetilde H\bigl(\eta(t),\vec
{x}\bigr) \bigr]
\\
&=& \mathbb{E}^{\eta} \bigl[L^{\mathrm{part}}\widetilde H\bigl(\eta(t),\vec
{x}\bigr) \bigr]
\\
&=& L^{\mathrm{part}}\mathbb{E}^{\eta} \bigl[\widetilde H\bigl(\eta(t),\vec
{x}\bigr) \bigr].
\end{eqnarray*}
The equality of the first line is from the definition of the generator
of $\eta(t)$; the equality between the first and second lines is from
the commutativity of the generator with the Markov semigroup; the
equality between the second and third lines is from applying
Proposition~\ref{propLequivtilde} to the expression inside the
expectation; the final equality is from the fact that the generator
$L^{\mathrm{part}}$ now acts on the $\vec{x}$ coordinate and the expectation
acts on the $\eta$ coordinate. This shows that, as a function of $t$~and~$\vec{x}$, $ \mathbb{E}^{\eta} [\widetilde H(\eta(t),\vec
{x})
]$ solves the true evolution equation of Definition~\ref
{WASEPtrueevol} (checking condition 2 is straightforward and condition
3 can be checked as in the proof of Proposition~\ref{asepuniq}).

On the other hand, Proposition~\ref{asepuniq} implies that $\mathbb
{E}^{\eta
} [\widetilde H(\eta,\vec{x}(t)) ]$ also solves the true
evolution equation of Definition~\ref{WASEPtrueevol} and that it is
the unique such solution. This proves the desired equality to show the
claimed duality.
\end{pf*}

\begin{pf*}{Proof of Theorem~\ref{thmASEPduality}}
This follows exactly as in the proof of Theorem~\ref
{thmASEPdualitytilde}, with Proposition~\ref{propLequivtilde}
replaced by Proposition~\ref{propLequiv}.
\end{pf*}

\subsection{Systems of ODEs}
As a result of duality, we provide two different systems of ODEs to
characterize $\mathbb{E}^{\eta} [\widetilde H(\eta(t),\vec
{x}) ]$.
These two systems should be compared to the first two systems of ODEs
associated to \mbox{$q$-}TASEP duality, given in Proposition~\ref
{propsystemsODEqTASEP}. It is not entirely clear how to formulate a
Schr\"{o}dinger\vspace*{1pt} equation with Bosonic Hamiltonian for ASEP due to the
\textit{strict} ordering of $\vec{x}\in\widetilde W^{k}$. This does not,
however, pose any significant impediment as we are more concerned with
solving the free evolution equation with $k-1$ boundary conditions.

We first state the result for the $\widetilde H(\eta,\vec{x})$ duality.

%
\begin{proposition}\label{propasepeqnstilde}
Let $\eta$ be a left-finite occupation configuration in $\{0,1\}
^\mathbb{Z}$
and $\eta(t)$ be ASEP started from $\eta(0)=\eta$.
\begin{longlist}[(A)]
\item[(A)] \emph{True evolution equation}: If $\tilde h(t;\vec
{x})\dvtx \mathbb{R}_{+}\times\widetilde W^{k} \to\mathbb{R}$ solves the
system of ODEs given
in Definition~\ref{WASEPtrueevol} with initial data $\tilde h_0(\vec
{x}) = \widetilde H(\eta,\vec{x})$, then for all \mbox{$\vec{x}\in\widetilde
W^{k}$}, $\mathbb{E}^{\eta}  [\widetilde H(\eta(t),\vec{x}) ]
= \tilde h(t;\vec{x})$.

\item[(B)] \emph{Free evolution equation with $k-1$ boundary
conditions}: If $\tilde u\dvtx \mathbb{R}_{+}\times\mathbb{Z}^k \to
\mathbb{R}$ solves:
\begin{enumerate}[(2)]
\item[(1)] For all $\vec{x}\in\mathbb{Z}^k$ and $t\in\mathbb{R}_{+}$,
%
%
\begin{eqnarray}\label{eqnASEP}
&& \frac{d}{dt} \tilde u(t;\vec{x})
\nonumber\\[-8pt]\\[-8pt]
&&\qquad = \sum
_{i=1}^{k} \bigl[ a_{x_i -1} p \tilde u
\bigl(t;\vec{x}_i^-\bigr) + a_{x_i} q \tilde u\bigl(t;
\vec{x}_i^+\bigr) - (a_{x_i -1}q + pa_{x_i})\tilde u(t;\vec{x}) \bigr];\nonumber
\end{eqnarray}
\item[(2)] For all $\vec{x}\in\mathbb{Z}^k$ such that for some $i\in\{
1,\ldots,
k-1\}$, $x_{i+1}=x_{i}+1$,
%
%
\begin{equation}
\label{bcASEP} p \tilde u\bigl(t;\vec{x}_{i+1}^{-}\bigr)+q
\tilde u\bigl(t;\vec{x}_{i}^{+}\bigr) = \tilde u(t;\vec{x});
\end{equation}
\item[(3)] There exist constants $c,C>0$ and $\delta>0$ such that for all
$\vec{x}\in\widetilde W^{k}$, $t\in[0,\delta]$
\[
\bigl|\tilde u(t;\vec{x})\bigr| \leq C e^{c\|\vec{x}\|_1};
\]
\item[(4)] For all $\vec{x}\in\widetilde W^{k}$, as $t\to0$, $\tilde u(t;\vec
{x}) \to\widetilde H(\eta,\vec{x})$.
\end{enumerate}
Then for all $\vec{x}\in\widetilde W^{k}$, $\mathbb{E}^{\eta}
[\widetilde
H(\eta(t),\vec{x}) ] = \tilde u(t;\vec{x})$.
\end{longlist}
\end{proposition}

\begin{pf}
Part (A) is an immediate consequence of the duality result of Theorem
\ref{thmASEPdualitytilde} along with its proof. Call the three
conditions contained in Definition~\ref{WASEPtrueevol}~(A.1), (A.2)
and (A.3).

Part (B) follows by showing that if the four conditions for $\tilde u$
given in (B) hold, then it implies that $u(t;\vec{x})$ restricted to
$\vec{x}\in\widetilde W^{k}$ actually satisfies conditions~(A.1),
(A.2)~and~(A.3). Conditions (B.3) and (B.4) immediately imply conditions (A.2)~and~(A.3), respectively. It is easy to check that the $k-1$ boundary
conditions (B.2) along with the free evolution equation (B.1) combine
to yield the generator $L^{\mathrm{part}}$ and hence yield (A.1).
Applying part
(A), we see that given the conditions of (B), we may conclude that for
all $\vec{x}\in\widetilde W^{k}$, $\mathbb{E}^{\eta}  [\widetilde
H(\eta
(t),\vec{x}) ] = \tilde u(t;\vec{x})$.
\end{pf}

We have an almost identical result and proof associated with the
$H(\eta,\vec{x})$ duality.

%
\begin{proposition}\label{propasepeqns}
Let $\eta$ be a left-finite occupation configuration in $\{0,1\}
^\mathbb{Z}$
and $\eta(t)$ be ASEP started from $\eta(0)=\eta$.
\begin{longlist}[(A)]
\item[(A)] \emph{True evolution equation}: If $h(t;\vec{x})\dvtx \mathbb{R}_{+}
\times\widetilde W^{k} \to\mathbb{R}$ solves the system of ODEs
given in
Definition~\ref{WASEPtrueevol} with initial data $h_0(\vec{x}) =
H(\eta,\vec{x})$, then for all \mbox{$\vec{x}\in\widetilde W^{k}$},
$\mathbb{E}^{\eta
}  [H(\eta(t),\vec{x}) ] = h(t;\vec{x})$.

\item[(B)] \emph{Free evolution equation with $k-1$ boundary
conditions}: If $\tilde u\dvtx \mathbb{R}_{+}\times\mathbb{Z}^k \to
\mathbb{R}$ solves:
\begin{enumerate}[(4)]
\item[(1)] For all $\vec{x}\in\mathbb{Z}^k$ and $t\in\mathbb{R}_{+}$,
\[
\frac{d}{dt} u(t;\vec{x}) = \sum_{i=1}^{k}
\bigl[ p u\bigl(t;\vec {x}_i^-\bigr) + q u\bigl(t;
\vec{x}_i^+\bigr) - u(t;\vec{x}) \bigr];
\]
\item[(2)] For all $\vec{x}\in\mathbb{Z}^k$ such that for some $i\in\{
1,\ldots,
k-1\}$, $x_{i+1}=x_{i}+1$,
\[
p u\bigl(t;\vec{x}_{i+1}^{-}\bigr)+q u\bigl(t;
\vec{x}_{i}^{+}\bigr) = u(t;\vec{x});
\]
\item[(3)] There exist constants $c,C>0$ and $\delta>0$ such that for all
$\vec{x}\in\widetilde W^{k}$, $t\in[0,\delta]$
\[
\bigl|u(t;\vec{x})\bigr| \leq C e^{c\|\vec{x}\|_1};
\]
\item[(4)] For all $\vec{x}\in\widetilde W^{k}$, as $t\to0$, $u(t;\vec
{x}) \to
H(\eta,\vec{x})$.
\end{enumerate}
Then for all $\vec{x}\in\widetilde W^{k}$, $\mathbb{E}^{\eta}
[ H(\eta
(t),\vec{x}) ] = u(t;\vec{x})$.
\end{longlist}
\end{proposition}

\begin{pf}
Similar to that of Proposition~\ref{propasepeqnstilde}.
\end{pf}

\subsection{Nested contour integral ansatz}

From now on, we assume that all bond rate parameters $a_{x}\equiv1$,
in which case equation (\ref{eqnASEP}) becomes
%
%
\begin{equation}
\label{eqnASEPprime} \frac{d}{dt}\tilde u(t;\vec{x})= \sum
_{i=1}^{k} \bigl[ p \tilde u\bigl(t;
\vec{x}_i^-\bigr) + q \tilde u\bigl(t;\vec{x}_i^+\bigr)
- \tilde u(t;\vec {x}) \bigr].
\end{equation}

It is not a priori clear how one might explicitly solve the systems of
ODEs in Propositions~\ref{propasepeqnstilde} and~\ref
{propasepeqns}. For \mbox{$q$-}TASEP, when confronted with the analogous
problem of solving the system of ODEs in Proposition~\ref
{propsystemsODEqTASEP}, we appealed to a nested contour integral
ansatz which was suggested from the algebraic framework of Macdonald
processes (into which \mbox{$q$-}TASEP fits).

ASEP, on the other hand, is not known to fit into the Macdonald process
framework, nor any similar framework from which solutions to these
systems of ODEs would be suggested. Nevertheless, we demonstrate now
that we may apply a nested contour integral ansatz. We focus on solving
the system of ODEs in Proposition~\ref{propasepeqnstilde} for two
distinguished types of initial data. Notice that in the below theorem,
the contours are not nested, however, they are chosen in a particular
manner to avoid poles coming from the denominator $z_A-\tau z_B$.

%
\begin{definition}\label{somedefs}
For $\rho\in[0,1]$ consider an i.i.d. collection $\{Y_x\}_{x\geq1}$
of Bernoulli random variables taking value 1 with probability $\rho$.
Then the \textit{step Bernoulli} initial data for ASEP is given by setting
$\eta_x(0)=0$ for $x\leq0$ and $\eta_x(0)=Y_x$ for $x\geq1$. When
$\rho=1$, this is called \textit{step} initial data and
(deterministically) $\eta_x(0)=\mathbf{1}_{x\geq1}$. We also define
$\theta= \rho/(1-\rho)$.

Define the function
%
%
\begin{equation}
\label{fASEPdef} f_{{z}}({x},{t};{\rho}) = \exp \biggl[ -
\frac{z(p-q)^2}{(1+z)(p+qz)}t \biggr] \biggl(\frac{1+z}{1+z/\tau} \biggr)^{x-1}
\frac{1}{\tau+z}\frac
{-\tau\theta}{z-\tau\theta}.
\end{equation}
When $\rho=1$ (and hence $\theta=\infty$), the definition of
$f_{{z}}({x},{t};{1})$ corresponds to the expression above, with the final
fraction removed. Also define
%
%
\begin{equation}
\label{fASEPtwodef} F_{{z}}({x},{t};{\rho}) = \exp \biggl[ -
\frac
{z(p-q)^2}{(z+1)(p+qz)}t \biggr] \biggl(\frac{1+z}{1+z/\tau} \biggr)^{x}
\frac{-\tau\theta}{z -
\tau\theta}
\end{equation}
and likewise extend to $\rho=1$.

Finally, define an integration contour $C_{-\tau;-1}$ as a circle around
$-\tau$, chosen with small enough radius so that $-1$ is not included,
nor is the image of $C_{-\tau;-1}$ under multiplication by $\tau$. It
is also important that $\tau\theta$ and $0$ are not contained in
$C_{-\tau;-1}$, but these facts are necessarily true from the definition.
\end{definition}

%
\begin{theorem}\label{QtildeInt}
Fix nonnegative real numbers $0<p<q$ (normalized by $p+q=1$) and set
all bond rate parameters $a_{x}\equiv1$. Consider step Bernoulli
initial data with density $\rho\in(0,1]$. The system of ODEs given in
Proposition~\textup{\ref{propasepeqnstilde}(B)} is solved by the following formula:
%
\begin{equation}
\label{Qtilde_int_step_bern} \qquad\tilde u(t;\vec{x})= \frac{\tau^{k(k-1)/2}}{(2\pi\iota)^k} \int\cdots\int \prod
_{1\leq A<B\leq k} \frac{z_A-z_B}{z_A-\tau z_B} \prod
_{i=1}^{k} f_{{z_i}}({x_i},{t};{
\rho}) \,dz_i,
\end{equation}
where the integration contour is given by $C_{-\tau;-1}$.
\end{theorem}

As an immediate corollary of Theorem~\ref{QtildeInt} and Proposition
\ref{propasepeqnstilde}(B) we find formulas for joint moments of the
$\widetilde Q_x(t)$ defined in (\ref{NQdef}).
%
\begin{corollary}\label{QtildeCor}
Fix $k\geq1$, nonnegative real numbers $0<p<q$ (normalized by $p+q=1$)
and set all bond rate parameters $a_{x}\equiv1$.
For step Bernoulli initial data with density $\rho\in(0,1]$ and any
$\vec{x} \in\widetilde W^{k}$,
%
%
\begin{eqnarray}\label{Qtildefreeevol}
&& \mathbb{E} \bigl[\widetilde{Q}_{x_1}\bigl(\eta(t)\bigr)
\cdots\widetilde {Q}_{x_k}\bigl(\eta (t)\bigr) \bigr]
\nonumber\\[-8pt]\\[-8pt]
&&\qquad = \frac{\tau^{k(k-1)/2}}{(2\pi\iota)^k}
\int\cdots\int \prod_{1\leq A<B\leq k} \frac{z_A-z_B}{z_A-\tau z_B} \prod
_{i=1}^k f_{{z_i}}({x_i},{t};{
\rho}) \,dz_i,\nonumber
\end{eqnarray}
where the integration contour is given by $C_{-\tau;-1}$.
\end{corollary}

%
\begin{remark}
The true evolution equation (A) in Proposition~\ref
{propasepeqnstilde} can alternatively be solved using the Green's
function formula of \cite{TW1} for the ASEP particle process
generator. This results in a rather different expression than we find
in (\ref{Qtildefreeevol}) since the Green's function is expressed as a
sum of $k!$, $k$-fold contour integrals. The equivalence of the
expression in (\ref{Qtildefreeevol}) to the expression one arrives at
using \cite{TW1} is a result of a nontrivial symmetrization. The
single $k$-fold contour integral formula we find is essential as it
enables us to proceed from duality to the two types (Mellin--Barnes and
Cauchy) of Fredholm determinant formulas we find for~ASEP.
\end{remark}
\begin{pf*}{Proof of Theorem~\ref{QtildeInt}}
We give the proof for step Bernoulli initial data with $\rho\in
(0,1)$, and hence $\theta\in(0,\infty)$. The modification for the
$\rho=1$ case is trivial.

We need to prove that $\tilde u(t;\vec{x})$, as defined in (\ref
{Qtilde_int_step_bern}), satisfies the four conditions of Proposition
\ref{propasepeqnstilde}(B).

Condition (B.1) is satisfied by linearity and the fact that
\[
\biggl[\frac{d}{dt} - \Delta^{p,q} \biggr] f_{{z}}({x},{t};{
\rho}) = 0,
\]
where $\Delta^{p,q}g(x) = p g(x-1) + q g(x+1) - g(x)$ acts on the
$x$-variable in $f_{{z}}({x},{t};{\rho})$.


Condition (B.2) relies on the Vandermonde factors as well as the choice
of contours. Without loss of generality, assume that $x_{2}=x_{1} +1$.
We wish to show that
\[
p\tilde u\bigl(t;\vec{x}_{2}^{-}\bigr)+q \tilde u\bigl(t;
\vec{x}_{1}^{+}\bigr)-\tilde u(t;\vec{x})=0.
\]
Thinking of the left-hand side as an operator applied to $\tilde u(t;\vec{x})$, we compute the effect of this operator on the integrand
of (\ref{Qtilde_int_step_bern}) and find that it just brings out an
extra factor in the integrand [when compared to $\tilde u(t;\vec
{x}_{2}^{-})$] which is
%
%
\begin{eqnarray}\label{ASEPbcinteff}
&& p + q \biggl(\frac{1+z_1}{1+z_1/\tau} \biggr) \biggl(\frac
{1+z_{2}}{1+z_{2}/\tau}
\biggr) - \biggl(\frac{1+z_{2}}{1+z_{2}/\tau
} \biggr)
\nonumber\\[-8pt]\\[-8pt]
&&\qquad  = (z_1-\tau z_{2}) \frac{(p-q)/\tau}{(1+z_1/\tau )(1+z_{2}/\tau)}.\nonumber
\end{eqnarray}
We must show that the integral with this new factor times the integrand
in (\ref{Qtilde_int_step_bern}) is zero. The factor $(z_1-\tau z_2)$
cancels the term corresponding to $A=1$ and $B=2$ in the denominator of
\[
\prod_{1\leq A<B\leq k} \frac{z_A-z_B}{z_A-\tau z_B}.
\]
The term $(z_1-z_2)$ in the numerator remains, and the additional terms
coming from~(\ref{ASEPbcinteff}) are symmetric in $z_1$ and $z_2$.
Therefore, we can write
\[
\tilde u(t;\vec{x}) = \int\!\!\int(z_1-z_2)
G(z_1)G(z_2) \,dz_1 \,dz_2,
\]
where $G(z)$ involves the integrals in $z_3,\ldots, z_k$. Since the
contours are identical, this integral is zero, proving (B.2).

Condition (B.3) follows via very soft bounds. Observe that as $z$
varies along the contour $C_{-\tau;-1}$, and as $t$ varies in
$[0,\delta
]$ for any $\delta$, it is easy to bound $|f_{{z}}({x},{t};{\rho
})|\leq Ce^{c\|x\|_1}$ for some constants $c,C>0$. Since the contours
are finite and since the other terms in the integrand defining $\tilde u$ are bounded along $C_{-\tau;-1}$, $\tilde u(t;\vec{x})$ is likewise
bounded, thus implying the desired inequality to show condition~(B.3).

Condition (B.4) follows from residue calculus. In order to check it,
however, we must first determine what initial data corresponds to step
Bernoulli ASEP initial data.

%
\begin{lemma}
For step Bernoulli initial data with density parameter \mbox{$\rho\in(0,1]$}
and $\vec{x}\in\widetilde W^{k}$,
%
%
\begin{eqnarray}\label{stepBerinitialdata}
&& \mathbb{E} \bigl[\widetilde{Q}_{x_1}\bigl(\eta(0)\bigr)
\cdots\widetilde {Q}_{x_k}\bigl(\eta (0)\bigr) \bigr]
\nonumber\\[-8pt]\\[-8pt]
&&\qquad  = \mathbf{1}_{x_1>0} \prod_{j=1}^{k}
\rho\tau ^{k-j}\bigl(\rho\tau^{k-j+1} +1-\rho
\bigr)^{x_j-x_{j-1}-1}\nonumber
\end{eqnarray}
with the convention that $x_0=0$.
\end{lemma}
\begin{pf}
From the definition of $\widetilde{Q}_x(\eta)$, one readily sees that
%
%
\begin{equation}
\label{qtildefactor} \widetilde{Q}_{x_1}\bigl(\eta(0)\bigr)\cdots
\widetilde{Q}_{x_k}\bigl(\eta(0)\bigr) = \prod_{j=1}^k
\eta_{x_j} \tau^{(k-j)\eta_{x_j}} \tau^{(k-j+1)(\sum
_{y=x_{j-1}+1}^{x_{j}-1} \eta_y)}.
\end{equation}
This expression involves two types of terms: $\eta_x \tau^{\ell\eta
_x}$ and $\tau^{\ell\eta_x}$. Observe that
\[
\mathbb{E} \bigl[ \eta_x \tau^{\ell\eta_x} \bigr]= \rho\tau
^{\ell}, \qquad\mathbb{E} \bigl[ \tau^{\ell\eta_x} \bigr] = \rho\tau
^{\ell} + 1-\rho.
\]
Taking expectations of (\ref{qtildefactor}) and using the above
formulas, we get the desired result.
\end{pf}

Thus, in order to show (B.4) we must prove that
%
%
\begin{equation}
\label{u0vecx} \lim_{t\to0} \tilde u(t;\vec{x})=
\mathbf{1}_{x_1>0} \prod_{j=1}^{k}
\rho\tau^{k-j}\bigl(\rho\tau^{k-j+1} +1-\rho
\bigr)^{x_j-x_{j-1}-1}.
\end{equation}
(Note that for $\rho=1$ this simply reduces to $\mathbf{1}_{x_1>0}
\prod_{j=1}^{k} \tau^{x_j -1}$.) The first observation is that we can
take the limit of $t\to0$ inside of the integral defining $\tilde u(t;\vec{x})$. This is because the integral defining $\tilde u$ is
along a finite contour and the integrand is uniformly converging to its
$t=0$ limiting value along this contour.

When $t=0$ the exponential term in the integrand of (\ref
{Qtilde_int_step_bern}) disappears. If $x_1\leq0$, then the integrand
no longer has a pole at $z_1=-\tau$. Since there are no other poles
contained in the $z_1$ contour, Cauchy's theorem implies that the
integral is zero, hence the condition that $\tilde u(0;\vec{x})=0$ is
satisfied.

Alternatively, we must consider the case where $0<x_1<x_2<\cdots<x_k$.
We can write $\tilde u(0;\vec{x})$ as
%
%
\begin{equation}
\label{gl1} \tilde u(0;\vec{x}) = \tau^k \tau^{k(k-1)/2}
g_1(x_1,\ldots,x_k),
\end{equation}
where we define (for $\ell\geq1$),
\begin{eqnarray*}
g_{\ell}(x_1,\ldots,x_k) &=& \frac{(-1)^k}{(2\pi\iota)^k}
\int \cdots\int \prod_{1\leq A<B\leq k} \frac{z_A-z_B}{z_A-\tau z_B}
\\
&&\hspace*{70pt}{}\times  \prod
_{i=1}^k \biggl(\frac{1+z_i}{1+z_i/\tau}
\biggr)^{x_i-1} \frac
{1}{\tau+z_i}\frac{\theta}{z_i-\tau^\ell\theta} \,dz_i.
\end{eqnarray*}
As a convention, when $k=0$ we define $g_{\ell}\equiv1$.
%
\begin{lemma}\label{gllemma}
For $\ell\geq1$ and $0<x_1<x_2<\cdots<x_k$,
\[
g_{\ell}(x_1,\ldots,x_k) = \biggl(
\frac{ 1+\tau^{\ell} \theta}{
1+\tau^{\ell-1} \theta} \biggr)^{x_k-1} \frac{\theta}{\tau+
\tau^{\ell} \theta} g_{\ell+1}(x_1,
\ldots,x_{k-1}).
\]
\end{lemma}
\begin{pf}
The lemma follows from residue calculus. Expand $z_k$ to infinity. Due
to quadratic decay in $z_k$ at infinity there is no pole. Thus, the
integral is equal to~$-1$ times the sum of the residues at $z_k= \tau
^{-1}z_j$ for $j<k$ and at $z_k = \tau^{\ell}\theta$.

First, consider the residue at $z_k= \tau^{-1}z_j$ for some $j<k$.
That residue equals an integral with one fewer variable:
\begin{eqnarray*}
&& \frac{(-1)^{k-1}}{(2\pi\iota)^{k-1}} \int\cdots\int \prod_{1\leq A<B\leq k-1}
\frac{z_A-z_B}{z_A-\tau z_B} \prod_{i=1}^{k-1} \biggl(
\frac{1+z_i}{1+z_i/\tau} \biggr)^{x_i-1} \frac{1}{\tau+z_i}\frac{dz_i}{z_i-\tau^\ell\theta}
\\
&&\qquad {}\times\frac{z_j/\tau-z_j}{\tau} \mathop{\prod_{i=1}}_{i\neq
j}^{k-1}
\frac{z_i-z_j/\tau}{z_i-z_j} \biggl(\frac{1+ z_j/\tau
}{1+z_j/\tau^2} \biggr)^{x_k-1}
\frac{1}{\tau+ z_j/\tau} \frac{1}{
z_j/\tau- \tau^{\ell}\theta}.
\end{eqnarray*}
This integrand has no pole at $z_j=-\tau$. This is because the new
factor contains $(1+z_j/\tau)^{x_k-1}$ in the numerator and, since
$x_k>x_j$, this factor cancels the pole coming from the denominator
$(1+z_j/\tau)^{x_j-1}$. Since the contour for $z_j$ was a small circle
around $-\tau$ the fact that this pole is no longer present implies
that the entire integral is zero. This shows that the residue at
$z_k=z_j/\tau$ for any $j<k$ is zero.

The remaining residue to consider is from $z_k = \tau^{\ell}\theta
$. One readily checks that evaluating this residue leads to the desired
recursion relation between $g_{\ell}(x_1,\ldots,x_k)$ and $g_{\ell
+1}(x_1,\ldots, x_{k-1})$. Finally, note that when $k=1$ the recursion
holds under the convention which we adopted that without any arguments,
$g_{\ell}$~equals~1.
\end{pf}

We may now conclude the proof of condition (B.4). Iteratively applying
Lemma~\ref{gllemma} leads to
\[
g_1(x_1,\ldots,x_k) = \prod
_{j=1}^{k} \biggl(\frac{1+ \tau^j
\theta}{1+\tau^{j-1} \theta}
\biggr)^{x_{k-j+1}-1} \frac
{\theta}{\tau+\tau^j \theta}.
\]
After some algebra one confirms that plugging this into (\ref{gl1})
leads to the desired equation of (\ref{u0vecx}), and hence completes
the proof of condition (B.4).
\end{pf*}

\subsection{ASEP moment formula}

We seek to compute an integral formula for the moments of $Q_x(\eta
(t))$. Even if we were to solve the system of equations in Proposition
\ref{propasepeqns}(B), this would not suffice since $\vec{x}$ is
restricted to lie in $\widetilde W^{k}$ (i.e., all $x_i$ distinct). The
extension of that solution outside $\widetilde W^{k}$ does not have any
necessary meaning as an expectation. Instead, the following lemma shows
that we may recover the moments of $Q_x$ from the formula given in
Corollary~\ref{QtildeCor} for $\mathbb{E} [\widetilde
{Q}_{x_1}(t)\cdots
\widetilde{Q}_{x_k}(t) ]$. Theorem~\ref{qmomInt} below gives the
final formula for $\mathbb{E} [(Q_x(\eta))^n ]$.

%
\begin{lemma}\label{qmomQtilde}
Recalling $Q_x(\eta)$ and $\widetilde{Q}_x(\eta)$ defined in (\ref
{NQdef}), we have
%
%
\begin{equation}
\label{QnQtilde} \bigl(Q_x(\eta)\bigr)^n = \sum
_{k=0}^n \pmatrix{n\cr k}_{\tau} (\tau;
\tau)_{k} (-1)^k \sum_{x_1<\cdots<x_k\leq x}
\widetilde{Q}_{x_1}(\eta)\cdots \widetilde{Q}_{x_k}(\eta),
\end{equation}
where the empty sum (when $k=0$) is defined as equal to $1$.
\end{lemma}

\begin{pf}
This lemma can be found as Proposition 3 in \cite{IS}. The derivation
provided therein utilizes the $U_q(sl_2)$ symmetry of the spin chain
representation of~ASEP. We provide an elementary proof.

Recall that $Q_x(\eta)$ and $\widetilde Q_x(\eta)$ are functions of the
occupation variables $\eta$ and if $\eta$ is not left-finite, then
both sides above are zero.

In order to prove the identity, we develop generating functions for
both sides and show that they are equal. Multiply both sides of the
claimed identity by $u^n/(\tau;\tau)_n$ and sum over $n\geq0$. The
$\tau$-binomial theorem (see Section~\ref{secqdef} with $q$ replaced
by~$\tau$) implies that the generating function for the left-hand side
of (\ref{QnQtilde}) can be summed as
\[
\sum_{n=0}^{\infty}\frac{u^n}{(\tau;\tau)_n}
\bigl(Q_x(\eta)\bigr)^n = \frac{1}{(u Q_x(\eta);\tau)_{\infty}}.
\]
For $|u|$ small enough, this series is convergent and it represents an
analytic function of $u$.

Turning to the generating function for the right-hand side of (\ref
{QnQtilde}), if $|u|$ is sufficiently small, it is justifiable to
rearrange the series in $u$ into
\[
\sum_{k=0}^{\infty}\,  \sum
_{x_1<x_2<\cdots<x_k\leq x} (-1)^k u^k
\widetilde{Q}_{x_1}(\eta)\cdots\widetilde{Q}_{x_k}(\eta) \sum
_{n\geq
k}^{\infty} \frac{u^{n-k}}{(\tau;\tau)_{n-k}}.
\]
The summation over $n\geq k$ can be evaluated as $1/(u;\tau)_{\infty
}$ and factored out. Also, the summation over $k$ and ordered sets
$x_1<x_2<\cdots<x_k\leq x$ can be rewritten yielding the right-hand
side of (\ref{QnQtilde}) equals
\[
\frac{\prod_{y\leq x}(1-u \widetilde{Q}_{y}(\eta))}{(u;\tau)_{\infty}}.
\]
The above manipulations are justified as long as $|u|$ is small enough,
due to the fact that all but finitely many of the $\widetilde{Q}_y(\eta)$
are zero.

The proof now reduces to showing that
\[
\frac{1}{(u Q_x(\eta);\tau)_{\infty}} = \frac{\prod_{y\leq x}(1-u
\widetilde{Q}_{y}(\eta))}{(u;\tau)_{\infty}}.
\]
This, however, is an immediate consequence of the definitions of
$Q_x(\eta)$ and $\widetilde{Q}_y(\eta)$. To see this, assume that $\eta
_y=0$ for all $y\leq x$ except when $y=n_1,\ldots, n_r$. Then
$Q_x(\eta)= \tau^r$ and the left-hand side can be written as
\[
\frac{(1-u)\cdots(1-u\tau^{r-1})}{(u;\tau)_{\infty}}.
\]
On the other hand, note that $\widetilde{Q}_y(\eta)=0$ for all $y\leq x$
except $\widetilde{Q}_{n_i}(\eta)=\tau^{i-1}$. Thus, the right-hand side
can also be rewritten as
\[
\frac{(1-u)\cdots(1-u\tau^{r-1})}{(u;\tau)_{\infty}},
\]
hence completing the proof of the lemma.\vadjust{\goodbreak}
\end{pf}

For step and step Bernoulli initial data, using Corollary~\ref
{QtildeCor} and the symmetrization identities contained in Lemma~\ref
{combiden} we can evaluate part of (\ref{QnQtilde}) via the following result.

%
\begin{lemma}\label{QtildeIntlemma}
For step Bernoulli initial data with $\rho\in(0,1]$ and for all
$k\geq1$,
%
%
\begin{eqnarray}\label{QnQtildeprime}
&& (\tau;\tau)_k (-1)^k \sum
_{x_1<\cdots< x_k\leq x} \mathbb{E} \bigl[ \widetilde{Q}_{x_1}\bigl(
\eta(t)\bigr)\cdots\widetilde{Q}_{x_k}\bigl(\eta(t)\bigr) \bigr]
\nonumber\\[-8pt]\\[-8pt]
&&\qquad  = \frac{\tau^{k(k-1)/2}}{(2\pi\iota)^k} \int\cdots\int\prod_{1\leq A<B\leq k}
\frac{z_A-z_B}{z_A-\tau z_B} \prod_{i=1}^k
F_{{z_i}}({x},{t};{\rho}) \frac{dz_i}{z_i},\nonumber
\end{eqnarray}
where the contours of integration are all $C_{-\tau;-1}$.
\end{lemma}
\begin{pf}
The starting point for this is the formula provided in Corollary~\ref
{QtildeCor} for $ \mathbb{E} [ \widetilde{Q}_{x_1}(\eta)\cdots
\widetilde
{Q}_{x_k}(\eta) ]$. In that formula, set $\xi_i =
(1+z_i)/(1+z_i/\tau)$ and note that the contour $C_{-\tau;-1}$ can be
chosen to be a sufficiently small circle around $-\tau$ so that $|\xi
_i|>1$ as $z_i$ varies in $C_{-\tau;-1}$. The summation over
$x_1<\cdots
<x_k\leq x$ on the left-hand side of (\ref{QtildeIntlemma}) can be
brought into the integrand and is performed by using (here we rely upon
$|\xi_i|>1$ for convergence)
\[
\sum_{x_1<\cdots< x_k\leq x}\, \prod_{i=1}^k
\xi_i^{x_i-1} = (\xi_1 \cdots
\xi_k)^x \prod_{i=1}^k
\frac{1}{\xi_{1}\cdots\xi_{i}-1}.
\]
After performing the summation as above, we observe that since all
contours are the same, we may symmetrize the left-hand side. For the
same reason, we may symmetrize the right-hand side integrand in (\ref
{QnQtildeprime}). The symmetrization is achieved by using the two
combinatorial identities in Lemma~\ref{combiden}---identity (\ref
{combiden1}) is used to symmetrize the left-hand side, while (\ref
{combiden2}) is used to symmetrize the right-hand side. The two
resulting symmetrized formulas are identical, thus yielding the proof.
\end{pf}
%

We may now prove the following moment formula.


%
\begin{theorem}\label{qmomInt}
Fix nonnegative real numbers $0<p<q$ (normalized by $p+q=1$) and set
all bond rate parameters $a_{x}\equiv1$. Consider step Bernoulli
initial data with density $\rho\in(0,1]$. Then for all $n\geq1$,
%
%
\begin{eqnarray}\label{eqnqmomInt}
\qquad && \mathbb{E} \bigl[ \tau^{n N_x(\eta(t))} \bigr]\nonumber
\\
&&\qquad = \mathbb{E} \bigl[\bigl(Q_x\bigl(\eta (t)\bigr)\bigr)^n \bigr]
\\
&&\qquad  = \tau^{n(n-1)/2}\frac{1}{(2\pi\iota)^n} \int\cdots\int \prod
_{1\leq A<B\leq n} \frac{z_A-z_B}{z_A-\tau z_B} \prod_{i=1}^n
F_{{z_i}}({x},{t};{\rho}) \frac{dz_i}{z_i},\nonumber
\end{eqnarray}
%
where the integration contour for $z_A$ is composed of two disconnected
pieces which include $0,-\tau$ but does not
include $-1$, $\tau\theta$ or $\{\tau z_B\}_{B>A}$ (see Figure~\ref{ASEPnestedcontours} for an illustration of such contours).
\end{theorem}

%
\begin{figure}[b]

\includegraphics{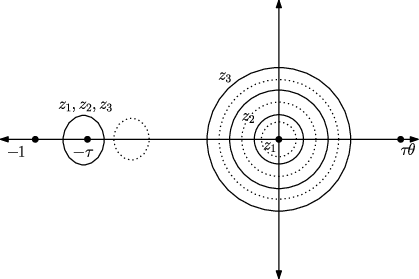}

\caption{The contour for $z_A$ includes $0,-\tau$ but does not
include $-1$, $\tau\theta$ or $\{\tau z_B\}_{B>A}$. The dotted
lines represent the images of the contours under multiplication by
$\tau$. For instance, observe that the $z_1$ contour does not include
the image under multiplication by $\tau$ of $z_2$ or $z_3$.}\label{ASEPnestedcontours}
\end{figure}

\begin{pf}
The left-most equality of (\ref{eqnqmomInt}) is just by definition.
The proof of the second equality relies on the following lemma. For an
illustration of the types of contours involved, see Figure~\ref
{ASEPnestedcontours}.

%
\begin{lemma}\label{lem:ASEPmu}
Fix $n\geq1$. Assume $f(z)$ is a meromorphic function on $\mathbb{C}$ which
has no poles in a ball around 0 and which has $f(0)=1$. Let $C^0$ be a
small circle centered at 0 and $C^{1}$ be another closed contour.
Assume that there exists \mbox{$r>\tau^{-1}$} such that $\tau C^1$ is not
contained inside~$r^n C^{0}$, and such that $f$ has no poles inside~$r^nC^0$. Define $C^0_i=r^i C^0$, $C^1_i = C^1$ and $C_i = C^0_i \cup
C^1_i$ for $1\leq i \leq n$. Let
\[
\nu_n = \frac{1}{(2\pi\iota)^n}\int_{C_1} \cdots\int
_{C_n} \prod_{1\leq A<B\leq n}
\frac{z_A-z_B}{z_A-\tau z_B} \prod_{i=1}^{n}
f(z_i)\frac{dz_i}{z_i}
\]
and
\[
\tilde\nu_k = \frac{1}{(2\pi\iota)^k}\int_{C^1}
\cdots\int_{C^1} \prod_{1\leq A<B\leq k}
\frac{z_A-z_B}{z_A-\tau z_B} \prod_{i=1}^{k}
f(z_i)\frac{dz_i}{z_i}
\]
with the convention that $\tilde\nu_0=1$.
Then
\[
\nu_n = \sum_{k=0}^{n} \pmatrix{n\cr k}_{\tau} \tau^{(k(k-1)/2)- (n(n-1)/2)} \tilde\nu_k.
\]
\end{lemma}
\begin{pf}
In order to evaluate the integrals defining $\nu_n$ we split them into
$2^n$ integrals indexed by $S\subset\{1,\ldots,n\}$ which determines
which integrations are along $C^0_i$ (all $z_i$ with $i\in S$) and
which are along $C_i^1$ (all $z_i$ with $i\notin S$). This shows that
\[
\nu_n= \sum_{k=0}^{n}\,
\mathop{\sum_{S\subset\{1,\ldots,n\}
}}_{|S|=k}
\frac{1}{(2\pi\iota)^n} \int_{C_1^{\varepsilon_1}} \cdots \int_{C_n^{\varepsilon_n}}
\prod_{1\leq A<B\leq n} \frac{z_A-z_B}{z_A-\tau z_B} \prod
_{i=1}^{n} f(z_i)\frac{dz_i}{z_i},
\]
where $\varepsilon_i=\mathbf{1}_{i\notin S}$, $1\leq i\leq n$.
We now claim that for any $S\subset\{1,\ldots,n\}$ with $|S|=k$,
%
%
\begin{equation}
\label{claimred} \qquad \frac{1}{(2\pi\iota)^n} \int_{C_1^{\varepsilon_1}} \cdots\int
_{C_n^{\varepsilon
_n}} \prod_{1\leq A<B\leq n}
\frac{z_A-z_B}{z_A-\tau z_B} \prod_{i=1}^{n}
f(z_i)\frac{dz_i}{z_i} = \tau^{-nk +\|S\|} \tilde\nu_{n-k},
\end{equation}
where we use the notation $\|S\| = \sum_{i\in S} i$.
Note that $C_i^{\varepsilon_i}$ is $C_i^0$ when $i\in S$ and $C_i^1$ (and hence
$C^1$) when $i\notin S$. To prove this claim, label the elements of $S$
as $i_1<i_2<\cdots<i_k$. By the fact that $z_{i_1}$ is contained in
$\tau C^0_{j}$ for all $j>i_1$, we may shrink the $z_{i_1}$ contour to
zero without crossing any poles except at $z_{i_1}=0$. The residue at
that pole is $\tau^{-(n-i_1)}$. Then we may shrink the $z_{i_{2}}$
contour to zero with contribution of $\tau^{-(n-i_{2})}$. Repeating
this up to $z_{i_k}$ yields a factor of
\[
\prod_{j=1}^{k} \tau^{-(n-i_{j})} =
\tau^{-nk +\|S\|}.
\]
The remaining integration variables can be relabeled so as to yield the
expression for $\tilde\nu_{n-k}$.

By using (\ref{claimred}), we find that
\begin{eqnarray*}
\nu_n &=& \sum_{k=0}^{n}\,
\mathop{\sum_{S\subset\{1,\ldots,n\}
}}_{|S|=k}
\tau^{-nk +\|S\|} \tilde\nu_{n-k}
\\
&=& \sum_{k=0}^{n} \tilde \nu_{n-k} \tau^{-nk + (k(k+1)/2)} \mathop{\sum
_{S\subset\{1,\ldots,n\}}}_{|S|=k} \tau^{\|S\| - (k(k+1)/2)}
\\
&=& \sum_{k=0}^{n} \tilde \nu_{n-k} \tau^{-nk + (k(k+1)/2)} \pmatrix{n\cr k}_{\tau}
\\
&=& \sum_{k=0}^{n} \pmatrix{n\cr k}_{\tau} \tau^{(k(k-1)/2) - (n(n-1)/2)} \tilde\nu_{k}
\end{eqnarray*}
as desired. From the first line to second line is by factoring. The
second line to third is by (\ref{qBinexpansion}). The third line to
fourth line is via changing $k$ to $n-k$.
\end{pf}

We return now to the proof of Theorem~\ref{qmomInt}. Consider the
second equality in (\ref{eqnqmomInt}). By virtue of the conditions
imposed on the contours, we may apply Lemma~\ref{lem:ASEPmu} with
$f(z) =F_{{z}}({x},{t};{\rho})$ and $C_i$ chosen to match the
contours defined in Theorem~\ref{qmomInt}. This shows that
%
%
\begin{eqnarray}\label{eqnthis}
\mbox{RHS of (\ref{eqnqmomInt})} &=& \sum
_{k=0}^{n} \pmatrix{n\cr k}_{\tau}
\tau^{k(k-1)/2} \frac{1}{(2\pi\iota)^k}
\nonumber\\[-8pt]\\[-8pt]
&&{}\times \int \cdots\int\prod_{1\leq A<B\leq k} \frac{z_A-z_B}{z_A-\tau z_B} \prod_{i=1}^{k}
F_{{z_i}}({x},{t};{\rho})\frac{dz_i}{z_i},\nonumber
\end{eqnarray}
%
where the integration contours are $C_{-\tau;-1}$ (which coincide with
$C^1$ from Lemma~\ref{lem:ASEPmu}).
By Lemma~\ref{QtildeIntlemma}, we rewrite (\ref{eqnthis}) as
\[
\mbox{RHS of (\ref{eqnqmomInt})} = \sum_{k=0}^{n}
\pmatrix{n\cr k}_{\tau} (\tau;\tau)_k (-1)^k \sum
_{x_1<\cdots<x_k\leq x} \widetilde {Q}_{x_1}\cdots
\widetilde{Q}_{x_k},
\]
where the empty sum (when $k=0$) is defined as equal to $1$. Lemma~\ref
{qmomQtilde} implies that this expression equals $\mathbb{E}
[(Q_x)^n ]$, proving the theorem.
\end{pf}

\section{From nested contour integrals to Fredholm determinants for ASEP}\label{sec5}

Using the nested contour integral formula of Theorem~\ref{qmomInt} for
$\mathbb{E} [\tau^{n N_x(\eta(t))} ]$ under step-Bernoulli
initial data for ASEP, we prove Mellin--Barnes and\vspace*{1pt} Cauchy-type Fredholm
determinant formulas for the $e_{\tau}$-Laplace transform of $\tau
^{N_x(\eta(t))}$. This transform characterizes the distribution of
$N_x(\eta(t))$ and is the starting point for asymptotic analysis. The
Mellin--Barnes-type formula we discover is new. The Cauchy-type formula
is, after inverting the $e_{\tau}$-Laplace transform, equivalent to
Tracy and Widom's ASEP formula for step Bernoulli \cite{TW4} initial
data (see also \cite{TW1,TW2} for step initial data where $\rho=1$).

The route from the nested contour integral of Theorem~\ref{qmomInt} to
the Fredholm determinants is similar to what was outlined in
Section~\ref{secmellin} (for the Mellin--Barnes-type) and
Section~\ref{seccauchy} (for the Cauchy-type). There are, however,
some differences due to the nature of the nested contours. For
\mbox{$q$-}TASEP the integration contour for $z_A$ was on a single connected
contour and the set of such contours (as $A$ varied) was nested so that
the $z_A$ contour contained $\{qz_B\}_{B>A}$. For ASEP, the integration
contour for $z_A$ is the union of two contours and the set of such
contours (as $A$ varies) is chosen such that the $z_A$ contour \textit{does not} contain $\{qz_B\}_{B>A}$. This difference in contours
necessitates an analogous result to Proposition~\ref{mukprop} (given
below as Proposition~\ref{ASEPmukprop}) when developing the
Mellin--Barnes-type formula, and an analogous result to Proposition
\ref{muandmutildeprop} (given via the combination of Lemmas~\ref
{qmomQtilde} and~\ref{QtildeIntlemma} above) when developing the
Cauchy-type formula.

\subsection{Mellin--Barnes-type determinant}

%
\begin{definition}\label{def:ASEPmukdef}
Fix $\alpha\in\mathbb{C}\setminus\{0\}$ and consider a meromorphic function
$f(z)$ which has a pole at $\alpha$ but does not have any other poles
in an open neighborhood of the line segment connecting $\alpha$ to 0.
For such a function and for any $k\geq1$, define
%
%
\begin{equation}
\label{ASEPmukdef} \mu_k= \frac{\tau^{k(k-1)/2}}{(2\pi\iota)^k} \int\cdots \int\prod
_{1\leq A<B\leq k} \frac{z_A-z_B}{z_A-\tau z_B} \prod
_{i=1}^{k} f(z_i) \frac{dz_i}{z_i},
\end{equation}
where the integration contour for $z_A$ contains 0, $\alpha$ but does
not include any other poles of $f$ or $\{\tau z_B\}_{B>A}$. For
instance, when $f$ is as in (\ref{fASEPtwodef}) and $\alpha=-\tau$,
then the contours illustrated in Figure~\ref{ASEPnestedcontours}
suffice (for $k=3$).
\end{definition}

\begin{proposition}\label{ASEPmukprop}
We have that for $\mu_k$ as in Definition~\ref{def:ASEPmukdef},
%
%
\begin{eqnarray}\label{ASEPmuk}
\mu_k &=& k_\tau! \mathop{\sum
_{\lambda\vdash k}}_{\lambda
=1^{m_1}2^{m_{2}}\cdots} \frac{1}{m_1!m_2!\cdots} \frac{(1-\tau
)^{k}}{(2\pi\iota)^{\ell(\lambda)}}\nonumber
\\
&&{}\times \int_C \cdots\int_C \det \biggl[
\frac{-1}{w_i \tau^{\lambda_i}-w_j} \biggr]_{i,j=1}^{\ell
(\lambda)}
\\
&&\hspace*{50pt}{}\times \prod_{j=1}^{\ell(\lambda)} f(w_j)f(\tau
w_j)\cdots f\bigl(\tau^{\lambda_j-1}w_j\bigr)
\,dw_j,\nonumber
\end{eqnarray}
where the integration contour $C$ for $w_j$ contains 0, $\alpha$ and
no other poles of $f$, and it does not intersect its image under
multiplication by any positive power of $\tau$ (see Figure~\ref
{ASEPunnestedcontours}).
\end{proposition}

\begin{pf}
The proof is via residue calculus and follows in the same manner as
Proposition~\ref{mukprop}, whose proof is found in \cite{BorCor} as
Proposition 3.2.1. Rather than repeating that proof, we just illustrate
the $k=2$ case.

%
\begin{figure}

\includegraphics{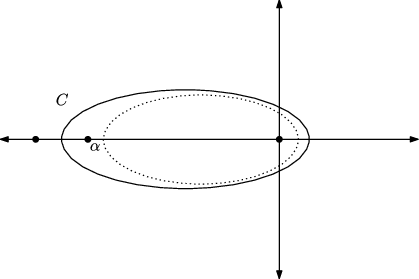}

\caption{The contour $C$ is chosen so as to contain 0, $\alpha$ and
no other poles of $f$ (such as the one indicated with a black dot to
the left of $\alpha$).}\label{ASEPunnestedcontours}
\end{figure}
%

%
\begin{figure}

\includegraphics{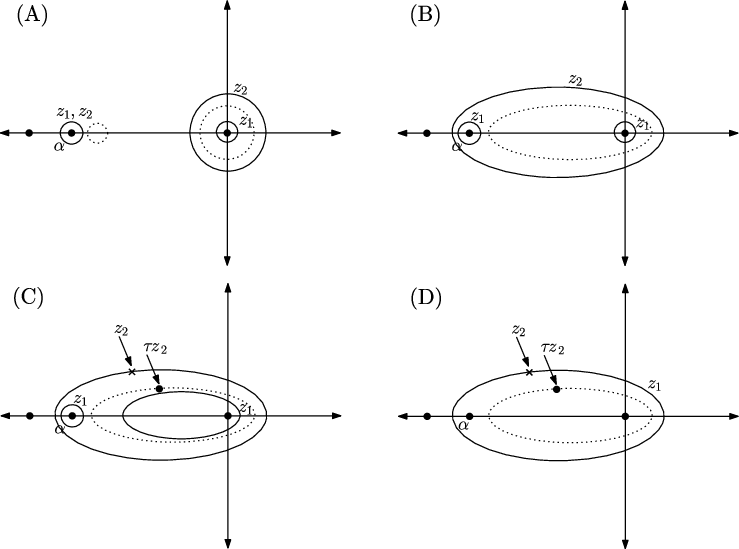}

\caption{\textup{(A)}: Both contours contain $0$ and $\alpha$, but the $z_1$
contour must not contain $\tau z_2$ (and neither contour may contain
any other poles such as the one indicated by the black dot to the right
of $\alpha$). \textup{(B)} The $z_2$ contour may freely (without crossing
poles) be deformed to a single circle containing $0$ and $\alpha$. \textup{(C)}
For $z_2$ fixed along that circle, the $z_1$ contour can be deformed
and only picks a pole when crossing the point $\tau z_2$. \textup{(D)} After
crossing that pole, the $z_1$ contour can be freely deformed to the
same contour on which $z_2$ is integrated.}\label{deformations}
\end{figure}

Consider $\mu_2$ as in Definition~\ref{def:ASEPmukdef} with contours
like in Figure~\ref{ASEPnestedcontours}. Initially, the $z_1$ contour
is chosen so as not to contain $\tau z_2$. Because the contours include
$\alpha$~and~$0$, they must be composed of two disjoint closed parts.
Around $\alpha$, the contours can be the same small circle, but around
0, the $z_2$ contour must have radius which is at least $\tau^{-1}$
times that of the $z_1$ contour. For $k=2$, such a contour is given in
Figure~\ref{deformations}(A). We may freely (without crossing any
poles) deform the $z_2$ contour to a single circle $C$ enclosing $0$
and $\alpha$ (but no poles of $f$). Such a resulting contour is given
in Figure~\ref{deformations}(B). The integration in $z_1$ and $z_2$
may be taken sequentially, so that for each fixed value of $z_2$ along
its contour of integration, we perform the integral in $z_1$. Thinking
of $z_2$ as fixed, we see that the $z_1$ contour can be deformed to the
circle $C$ by crossing a single pole at $z_1=\tau z_2$. This shown in
Figure~\ref{deformations}(C)~and~(D).\vadjust{\goodbreak} On account of crossing a pole,
we find that $\mu_2$ can be expressed as
\begin{eqnarray*}
\mu_2 &=& \frac{\tau}{(2\pi\iota)^2} \int_C \int
_C \frac
{z_1-z_2}{z_1-\tau z_2} f(z_1)f(z_2)
\frac{dz_1}{z_1}\frac{dz_2}{z_2}
\\
&&{} - \frac{1}{2\pi\iota} \int
_C (\tau-1) f(\tau z_2)f(z_2)
\frac{dz_2}{z_2}.
\end{eqnarray*}
Observe that there are also two terms contained in the right-hand side
of (\ref{ASEPmuk})---one term is a single integral and one is a
double integral. The single integral term matches exactly while to
match the double integral we simply symmetrize the integrand (as can be
done since $z_1$ and $z_2$ are on the same contour) and find those
terms match as well. In general, the partition $\lambda$ indexes the
clustering of residues into chains.
\end{pf}

Using the above result as well as Proposition~\ref{gendetprop}, we
find following Fredholm determinant formula for the $e_{\tau}$-Laplace
transform of $Q_x(\eta(t))=\tau^{N_x(\eta(t))}$.

\begin{theorem}\label{ASEPMellinBarnesThm}
Consider ASEP with $0<p<q$ (normalized by $p+q=1$), all bond rate
parameters $a_{x}\equiv1$, and step Bernoulli initial data with
density parameter $\rho\in(0,1]$. Then with notation $\theta=
\frac{\rho}{1-\rho}$ we have that for all $\zeta\in\mathbb
{C}\setminus
\mathbb{R}_{+}$,
\[
\mathbb{E} \biggl[\frac{1}{(\zeta\tau^{N_x(\eta(t))};\tau)_{\infty
}} \biggr] = \det\bigl(I+K^{\mathrm{ASEP}}_{\zeta}
\bigr),
\]
where $\det(I+K^{\mathrm{ASEP}}_{\zeta})$ is the Fredholm
determinant of
$K^{\mathrm{ASEP}}_{\zeta}\dvtx\break  L^2(C_{0,-\tau;-1,\tau\theta})\to
L^2(C_{0,-\tau;-1,\tau\theta})$, where
$C_{0,-\tau;-1,\tau\theta}$ a positively oriented contour containing
0, $-\tau$ on its
interior and with $-1$ and $\tau\theta$ on its exterior. The
operator $K_{\zeta}$ is defined in terms of its integral kernel
\[
K^{\mathrm{ASEP}}_{\zeta}\bigl(w,w'\bigr) =
\frac{1}{2\pi\iota}\int_{D_{R,d}} \Gamma (-s)\Gamma(1+s) (-
\zeta)^s \frac{g_{{w}}({x},{t};{\rho})}{g_{{\tau
^s w}}({x},{t};{\rho})} \frac{-1}{\tau^s w-w'}\,ds.
\]
The contour $D_{R,d}$ is given in Definition~\ref{DRd} with $d>0$
taken to be sufficiently small and $R>0$ sufficiently large so that
\[
\mathop{\inf_{w,w'\in C_{0,-\tau;-1,\tau\theta}}}_{s\in D_{R,d}} \bigl|q^s
w-w'\bigr|>0\quad\mbox{and}\quad \mathop{\sup
_{w,w'\in C_{0,-\tau;-1,\tau\theta}}}_{s\in D_{R,d}} \biggl\llvert \frac
{g(w)}{g(q^sw)}\biggr
\rrvert <\infty.
\]
The function $g_{{z}}({x},{t};{\rho})$, is given by
%
%
\begin{equation}
\label{gASEP} g_{{z}}({x},{t};{\rho}) = \exp \biggl[(q-p)t
\frac{\tau}{z+\tau
} \biggr] \biggl(\frac{\tau}{z+\tau} \biggr)^x
\frac{1}{
(z/(\tau\theta);\tau )_{\infty}}.
\end{equation}
%
\end{theorem}

%
\begin{corollary}
We have that
\[
\mathbb{P}\bigl(N_x\bigl(\eta(t)\bigr) = m\bigr) =
\frac{-\tau^m}{2\pi\iota} \int \bigl(\tau ^{m+1} \zeta; \tau
\bigr)_{\infty} \det\bigl(I+K^{\mathrm{ASEP}}_{\zeta}\bigr) \,d\zeta,
\]
where the contour of integration encloses $\zeta=\tau^{-M}$ for
$0\leq M\leq m$ and only intersects $\mathbb{R}_{+}$ in finitely many points.
\end{corollary}
\begin{pf}
This follows almost immediately from the inversion formula in
Proposition~\ref{qlaplaceinverse}. The one small impediment is that
our formula for the $q$-Laplace transform via the Fredholm determinant
$\det(I+K^{\mathrm{ASEP}}_{\zeta})$ is not defined for $\zeta\in
\mathbb{R}_{+}$. On
the other hand, it is easy to see (and explained in the proof of
Theorem~\ref{ASEPMellinBarnesThm}) that the function $f(\zeta)$
defined by $\zeta\mapsto\mathbb{E} [1/ (\zeta\tau
^{N_x(\eta
(t))};\break \tau )_{\infty} ]$ is analytic away from $\zeta=
\tau^{-M}$, for integers $M\geq0$. Thus $\mathbb{P}(N_x(\eta(t)) =
m)$ can
be computed via a contour integral (as specified in the inversion
formula)\vspace*{1.5pt} involving $f(\zeta)$ in the integrand.

On the other hand, we know that $f(\zeta)=\det(I+K^{\mathrm
{ASEP}}_{\zeta})$
for $\zeta$ not on $\mathbb{R}_{+}$, and
hence $\det(I+K^{\mathrm{ASEP}}_{\zeta})$
extends analytically through $\mathbb{R}_{+}\setminus\{\tau^{-M}\}
_{M\geq
0}$. Thus, as long as the integration contour for $\zeta$ only
intersects $\mathbb{R}_{+}$, in finitely
many points, we can compute the
necessary\vspace*{1.5pt} inversion contour integral with $f(\zeta)$ replaced by $\det
(I+K^{\mathrm{ASEP}}_{\zeta})$.
\end{pf}

\begin{pf*}{Proof of Theorem~\ref{ASEPMellinBarnesThm}}
Theorem~\ref{qmomInt} gives a nested contour integral formula for
$\mathbb{E}
[\tau^{kN_x(\eta(t))} ]$. Comparing it with Definition
\ref{def:ASEPmukdef} we see that $\mu_k = \mathbb{E} [\tau
^{kN_x(\eta
(t))} ]$ if
the contour is chosen as the one in Theorem~\ref{qmomInt} and if
\mbox{$f(z)=F_{{z}}({x},{t};{\rho})$}.
This function can be written as $f(z) = g(z)/g(\tau z)$ where
$g(z)=g_{{z}}({x},{t};{\rho})$, is given in (\ref{gASEP}).

We apply Proposition~\ref{ASEPmukprop}, yielding an expression for
$\mu_k$ as in (\ref{ASEPmuk}). This matches the expression in (\ref
{muk}) up to changing $q$ to $\tau$ and sign inside the determinant.
We may therefore apply Proposition~\ref{gendetprop}, followed by
Proposition~\ref{propmellindet} (with $q$ replaced by $\tau$).
At a formal level, this shows that
%
%
\begin{equation}
\label{lhsk1ASEP} \sum_{k\geq0}\mu_k
\frac{\xi^k}{k_\tau!} = \det\bigl(I+K^1_{\xi
}\bigr)=\det
\bigl(I+K^2_{\xi}\bigr)
\end{equation}
the kernels $K^1_{\xi}$ and $K^2_{\xi}$ defined\vspace*{1pt} with respect to
$F_{{z}}({x},{t};{\rho})$ and $g_{{z}}({x},{t};{\rho})$, as above.
The contour $C_{\mathbb{A}}$ in those propositions should be taken to be
$C_{0,-\tau;-1,\tau\theta}$, as in the hypothesis of Theorem~\ref
{ASEPMellinBarnesThm}. In applying Proposition~\ref{propmellindet},
the contour $C_{1,2,\ldots}$ should be chosen to be $D_{R,d}$ with
$d>0$ sufficiently small, and $R>0$ sufficiently large, and the
contours $C_k$, $k\geq1$ should be chosen to be $D_{R,d;k}$. From the
definition of $C_{0,-\tau;-1,\tau\theta}$ and $g_{{z}}({x},{t};{\rho
})$, it is easy to
check that as long as $|\xi|$ is sufficiently small, the criteria for
these to be numerical equalities is satisfied.

Since by definition $N_x(\eta(t)) \geq0$ and $\tau<1$, it is
immediate that $\tau^{kN_x(\eta(t))}\leq1$. Hence, considering the
left-hand side of (\ref{lhsk1ASEP}), by choosing $|\xi|$ small enough
it is justifiable to interchange the summation in $k$ and the
expectation. By the $\tau$-Binomial theorem (see Section~\ref
{secqdef}), we find
%
%
\begin{equation}
\label{lhsk2ASEP} \mathbb{E} \biggl[\frac{1}{ ((1-\tau)\xi\tau^{N_x(\eta
(t))};\tau
)_{\infty}} \biggr]=\det
\bigl(I+K^2_{\xi}\bigr).
\end{equation}
This equality holds for all $|\xi|$ sufficiently small. However, the
right-hand side is analytic in $\xi\notin\mathbb{R}_{+}$ due to Proposition
\ref{propmellindet}. From the definition of the left-hand side,
\[
\mathbb{E} \biggl[\frac{1}{ ((1-\tau)\xi\tau^{N_x(\eta
(t))};\tau
)_{\infty}} \biggr] = \sum
_{\ell=0}^{\infty} \frac{ \mathbb{P}
(N_x(\eta(t)) = \ell)}{ ((1-\tau)\xi\tau^\ell;\tau
)_{\infty}}.
\]
For any $\xi\notin\{(1-\tau)^{-1} \tau^{-M}\}_{M=0,1,\ldots},$
within a neighborhood of $\xi$ the infinite products are uniformly
convergent and bounded away from zero. As a result, the series is
uniformly convergent in a neighborhood of any such $\xi$ which implies
that its limit is analytic. Therefore, both sides of (\ref{lhsk2ASEP})
are analytic for $\xi\notin\mathbb{R}_{+}$ and hence by uniqueness
of the
analytic continuation they are equal on this set.

The desired result for this theorem is achieved by setting $\xi=
(1-\tau)^{-1} \zeta$ thus completing the proof.
\end{pf*}

\subsection{Cauchy-type determinant}
%
\begin{theorem}\label{ASEPcauchy}
Consider ASEP with $0<p<q$ (normalized by $p+q=1$), all bond rate
parameters $a_{x}\equiv1$, and step Bernoulli initial data with
density parameter $\rho\in(0,1]$. Then with notation $\theta=
\frac{\rho}{1-\rho}$ we have that for all $\zeta\in\mathbb{C}$
%
%
\begin{equation}
\label{ASEPcauchyeqn} \mathbb{E} \biggl[ \frac{1}{ (\zeta\tau^{N_x(\eta(t))};\tau
)_{\infty}} \biggr] =
\frac{\det(I-\zeta\widetilde K^{\mathrm
{ASEP}})}{(\zeta;\tau
)_{\infty}},
\end{equation}
where $\det(I-\zeta\widetilde K^{\mathrm{ASEP}})$ is an entire function
of $\zeta$
and is the Fredholm\vspace*{1pt} determinant of
$\widetilde K^{\mathrm{ASEP}}\dvtx  L^2(C_{-\tau;-1})\to L^2(C_{-\tau;-1})$
defined in terms of its integral kernel
\[
\widetilde K^{\mathrm{ASEP}}\bigl(w,w'\bigr) = \frac{F_{{w}}({x},{t};{\rho})}{\tau
w -w'}
\]
with $F_{{w}}({x},{t};{\rho})$ defined in (\ref{fASEPtwodef}),
and $C_{-\tau;-1}$ is a circle around $-\tau$, chosen with small enough
radius so that $-1$ is not included, and nor is the image of the circle
under multiplication by $\tau$ (see Definition~\ref{somedefs}).
\end{theorem}

%
\begin{corollary}
Consider ASEP with $0<p<q$ (normalized by \mbox{$p+q=1$}), all bond rate
parameters $a_{x}\equiv1$, and step Bernoulli initial data with
density parameter $\rho\in(0,1]$. Then
%
%
\begin{equation}
\label{TWcoreqn} \mathbb{P}\bigl(N_x\bigl(\eta(t)\bigr)=m\bigr) = -
\tau^m \frac{1}{2\pi\iota} \int \frac
{\det(I-\zeta\widetilde K^{\mathrm{TW\mbox{-}ASEP}})}{(\zeta;\tau)_{m+1}} \,d\zeta,
\end{equation}
where\vspace*{2pt} the integral is over a contour enclosing $\zeta=q^{-M}$ for
$0\leq M\leq m-1$. Here, $\det(I-\zeta\widetilde K^{\mathrm{TW\mbox{-}ASEP}})$
is the Fredholm
determinant of
$
\widetilde K^{\mathrm{TW\mbox{-}ASEP}}\dvtx\break   L^2(C_R)\to L^2(C_R)
$
defined in terms of its integral kernel
\[
\widetilde K^{\mathrm{TW\mbox{-}ASEP}}\bigl(\xi,\xi'\bigr) = q
\frac{\xi^x e^{\varepsilon
(\xi)t}}{p+q\xi\xi'-\xi} \frac{\rho(\xi-\tau)}{\xi-1+\rho(1-\tau)}
\]
and $\varepsilon(\xi) = p\xi^{-1}+q\xi-1$ and $C_R$ is a circle around zero
of radius $R$ so large that the denominator $p+q\xi\xi'-\xi$ and
$\xi-1\rho(1-\tau)$ are\vspace*{1pt} nonzero on and outside the contour. As a
function of $\zeta$, $\det(I-\zeta\widetilde K^{\mathrm{TW\mbox{-}ASEP}})$ is entire.
\end{corollary}

\begin{pf}
This follows from Theorem~\ref{ASEPcauchy} (after a change of
variables) and the $e_{\tau}$-Laplace transform inversion formula in
Proposition~\ref{qlaplaceinverse}.
The change of variables is
\[
\xi= \frac{1+w}{1+w/\tau}.
\]
Using the equivalences given in Remark~\ref{xirem}, and using the
definition of $\theta= \rho/(1-\rho)$ we find that
\[
f(w)\mapsto e^{\varepsilon(\xi)} \xi^x \frac{\rho(\xi-\tau)}{\xi
-1+\rho
(1-\tau)}.
\]
Similarly, we find
\[
\frac{1}{\tau w-w'} \mapsto\frac{(\tau-\xi)(\tau-\xi')}{\tau
(1-\tau)} \frac{q}{(p+q\xi\xi'-\xi)}.
\]
The change of variables introduces an additional Jacobian factor into
the new kernel which is given by
\[
\frac{-\tau(1-\tau)}{(\tau-\xi)(\tau-\xi')}.
\]
Finally, under this change of variables, the contour $C_{-\tau;-1}$
becomes $C_R$ as specified in the statement of the corollary, but with
clockwise orientation. Changing this to the standard counterclockwise
orientation introduces a factor of $-1$ into the kernel. Combining
these calculations, we find
\[
\mathbb{E} \biggl[ \frac{1}{ (\zeta\tau^{N_x(\eta(t))};\tau
)_{\infty}} \biggr] = \frac{\det(I-\zeta\widetilde K^{\mathrm
{TW\mbox{-}ASEP}})}{(\zeta;\tau
)_{\infty}},
\]
where $\widetilde K^{\mathrm{TW\mbox{-}ASEP}}$ is as in the statement of the corollary.

From Proposition~\ref{qlaplaceinverse}, it follows that
\begin{eqnarray*}
\mathbb{P}\bigl(N_x\bigl(\eta(t)\bigr)=m\bigr) &=& -\tau^m
\frac{1}{2\pi\iota} \int \bigl(\tau ^{m+1}\zeta;\tau
\bigr)_{\infty} \frac{\det(I-\zeta\widetilde K^{\mathrm
{TW\mbox{-}ASEP}})}{(\zeta;\tau
)_{\infty}} \,d\zeta
\\
&=& -\tau^m
\frac{1}{2\pi\iota} \int\frac{\det
(I-\zeta\widetilde K^{\mathrm{TW\mbox{-}ASEP}})}{(\zeta;\tau)_{m+1}} \,d\zeta,
\end{eqnarray*}
where the integral is taken over a contour enclosing $\zeta=q^{-M}$
for $0\leq M\leq m-1$, thus proving the corollary.
\end{pf}

%
\begin{remark}
For ASEP with step-Bernoulli initial data, Tracy and Widom \cite{TW4}
(for step initial data see \cite{TW1,TW2}) arrive at a very similar
formula which says
%
%
\begin{equation}
\label{TWform} \mathbb{P}\bigl(N_x\bigl(\eta(t)\bigr) \geq m\bigr) =
\frac{1}{2\pi\iota} \int\frac
{\det
(I-\zeta\widetilde K^{\mathrm{TW\mbox{-}ASEP}})}{(\zeta;\tau)_m} \frac{d\zeta
}{\zeta},
\end{equation}
where the integral is taken over a contour enclosing $\zeta=0$ and
$\zeta=q^{-M}$ for $0\leq M\leq m-1$.
Since $\mathbb{P}(N_x(\eta(t)) \geq m)- \mathbb{P}(N_x(\eta(t))
\geq m+1) = \mathbb{P}
(N_x(\eta(t)) = m)$, it is straightforward to go from (\ref{TWform})
to (\ref{TWcoreqn}) since
\[
\frac{1}{(\zeta;\tau)_{m} \zeta} - \frac{1}{(\zeta;\tau)_{m+1}
\zeta} = -\tau^m
\frac{1}{(\zeta;\tau)_{m+1}}.
\]
Going in the reverse direction uses a telescoping sum and would require
an a priori confirmation that the right-hand side of (\ref{TWform})
goes to zero as $m$ goes to infinity.
\end{remark}

\begin{pf*}{Proof of Theorem~\ref{ASEPcauchy}}
Let $\tilde\mu_k$ be given as in (\ref{muktildedef}) with $f(w) =
F_{{w}}({x},{t};{\rho})$ defined by (\ref{fASEPtwodef}) and contour
$C_{-\tau;-1}$ as in Definition~\ref{somedefs}. Then Proposition~\ref
{tildemuDetProp} and Remark~\ref{k1k2} imply that
%
%
\begin{equation}
\label{mseriesabove} \sum_{k\geq0} \tilde\mu_k
\frac{\xi^k}{k_{\tau}!} = \det(I+ \xi \widetilde K),
\end{equation}
where $\det(I+\xi\widetilde K)$ is the Fredholm determinant of
\[
\widetilde{K} \bigl(w,w'\bigr) = (1-\tau) \frac{f(w)}{\tau w-w'}.
\]
We need to check that this is a numerical equality\vspace*{1pt} (not just formal).
Because the kernel is bounded as $w$ varies along $\widetilde C_{-\tau}$
it follows that $\widetilde K$ is trace-class, and hence $\det(I+ \xi
\widetilde K)$ is an entire function of $\xi$.

In order to see that the left-hand side is uniformly convergent for
small enough~$|\xi|$, we utilize the probabilistic interpretation for
$\tilde\mu$. By combining Lemmas~\ref{qmomQtilde} and~\ref
{QtildeIntlemma}, we find that
\[
\mathbb{E} \bigl[\tau^{n N_x(\eta(t))} \bigr] = \sum
_{k=0}^{n} \pmatrix{n\cr k}_{\tau}
(-1)^k \tilde\mu_k.
\]
This transformation from $\tilde\mu_k$ to $\mathbb{E} [\tau^{n
N_x(\eta(t))} ]$ is upper-triangular, and hence can be inverted.
One checks that the inverse is given by
\[
(-1)^k \tilde\mu_k = (-1)^k
\tau^{k(k-1)/2} \sum_{j=0}^{k} \pmatrix{k \cr j}_{\tau^{-1}} \tau^{-j(j-1)/2} (-1)^j \mathbb {E}
\bigl[\tau^{j N_x(\eta(t))} \bigr].
\]
By (\ref{finqBinExp}), we find
\begin{eqnarray*}
(-1)^k \tilde\mu_k &=& (-1)^k
\tau^{k(k-1)/2} \mathbb {E} \bigl[\bigl(1-\tau^{N_x(\eta(t))}\bigr)\cdots
\bigl(1-\tau^{N_x(\eta(t))-k}\bigr) \bigr]
\\
&=& \mathbb{E} \bigl[\bigl(\tau^{N_x(\eta(t))}-1\bigr) \bigl(
\tau^{N_x(\eta
(t))}-\tau \bigr)\cdots\bigl(\tau^{N_x(\eta(t))}-
\tau^k\bigr) \bigr].
\end{eqnarray*}
This probabilistic interpretation of $\tilde\mu_k$ implies that
$|\tilde\mu_k|\leq1$, hence for $|\xi|$ small enough the series on
the left-hand side of (\ref{mseriesabove}) is convergent and the
equality is numerical.

By replacing $\xi= -\zeta/(1-\tau)$ and using the probabilistic
interpretation for $\tilde\mu_k$ to justify the exchange of summation
and expectation (assuming $|\zeta|$ small enough) this left-hand side
series equals
\begin{eqnarray*}
\sum_{k\geq0} \tilde\mu_k
\frac{ (-\zeta/(1-\tau)
)^k}{k_{\tau}!} &=& \mathbb{E} \biggl[\sum_{k\geq0}
\frac{ (\tau
^{-N_x(\eta(t))};\tau)_k}{(\tau;\tau)_k} \bigl(\zeta\tau^{N_x(\eta
(t))}\bigr)^k \biggr]
\\
&=& \mathbb{E} \biggl[ \frac{(\zeta;\tau)_{\infty}}{(\zeta\tau
^{N_{x}(\eta(t))};\tau)_{\infty}} \biggr].
\end{eqnarray*}
Since we already wrote down the Fredholm determinant for this
expression in (\ref{mseriesabove}), this establishes the claimed
result of the theorem, for $|\zeta|$ small enough.

Finally, note that
\[
\mathbb{E} \biggl[ \frac{(\zeta;\tau)_{\infty}}{(\zeta\tau
^{N_{x}(\eta
(t))};\tau)_{\infty}} \biggr] = \sum
_{k\geq0} \mathbb{P}\bigl(N_x\bigl(\eta(t)\bigr)=k
\bigr) (\zeta;\tau)_k.
\]
For any $\zeta\in\mathbb{C}$ and any compact neighborhood $\Omega$ of
$\zeta$, it is clear at as $k\to\infty$, the product defining
$(\zeta;\tau)_k$ converges uniformly over $\Omega$ to a finite
limit. This implies that the series is likewise uniformly convergent in
that compact neighborhood and, therefore, the series is analytic in a
neighborhood of $\zeta$. As $\zeta$ was arbitrary, this implies that
the left-hand side of (\ref{ASEPcauchyeqn}) is an entire function of
$\zeta$. We showed earlier that the right-hand side is entire,
therefore, since the two functions of $\zeta$ are equal for $|\zeta|$
small enough, by the uniqueness of analytic continuations it follows
that the equality holds for all $\zeta\in\mathbb{C}$, completing the proof.
\end{pf*}

\begin{appendix}
\section{Semidiscrete directed polymers}\label{seclimtran}
There are three main parameters in \mbox{$q$-}TASEP: time $t$, particle label
$n$ and the repulsion strength $q$ (the $a_i$ are also present, but
play a somewhat auxiliary role). On account of this, there are many
interesting scaling limits to be explored. We will presently focus on
one which involves scaling $q\to1$ and $t\to\infty$, but keeping $n$
fixed. We show that the limit of \mbox{$q$-}TASEP corresponds to a certain
semidiscrete version of the multiplicative stochastic heat equation
(and hence also the O'Connell--Yor semidiscrete directed polymer
partition function \cite{OY}). We then introduce the limit of the
\mbox{$q$-}TASEP free evolution equation with $k-1$ boundary conditions and
the Schr\"{o}dinger equation with Bosonic Hamiltonian [Proposition~\ref{propsystemsODEqTASEP}(B) and (C)] and show how these limits are
achieved from the analogous statement for \mbox{$q$-}TASEP. Finally, we remark
on the fact that taking a limit of the Mellin--Barnes-type Fredholm
determinant formula for the $e_q$-Laplace transform of \mbox{$q$-}TASEP yields
a rigorous derivation of an analogous formula for the Laplace transform
of the solution to the semidiscrete multiplicative stochastic heat equation.

From this semidiscrete limit, it is possible to take another limit to
the fully continuous (space--time) multiplicative stochastic heat
equation \cite{ACQ}. The free evolution equation with $k-1$ boundary
conditions and the Schr\"{o}dinger equation with Bosonic Hamiltonian
limit to the two different formulations of the attractive quantum delta
Bose gas.

%
\begin{definition}\label{def:semidiscSHE}
The semidiscrete multiplicative stochastic heat equation (SHE) with
initial data $z_0$ and drift vector $\tilde a=(\tilde a_1,\ldots,
\tilde a_N)$ is the solution to the system of stochastic ODEs
\[
dz(\tau,n) = \nabla z(\tau,n) \,d\tau+ z(\tau,n) \,dB_n,\qquad z(0,n) =
z_0(n),\qquad z(\tau,0)\equiv0,
\]
where $(B_1(s),\ldots, B_N(s))$ are independent standard Brownian
motions such that $B_i$ has drift $\tilde a_i$, and we use the notation
$\nabla z(\tau,n) = z(\tau,n-1)-z(\tau,n)$.
\end{definition}

There is a Feynman--Kac path integral representation for $z(\tau,n)$.
Let $\phi$ be a Markov process with state space $\mathbb{Z}$ which increases
by one at rate one (this is a~standard Poisson jump process whose
generator is the adjoint of $\nabla$). Let $\mathcal{E}$ denote the
expectation with respect to this path measure on $\phi$. Define the
energy of $\phi$ as the path integral through the disorder (the white
noises given by $dB_i$) along $\phi$:
\[
E_{\tau}(\phi) = \int_{0}^{\tau}
\,dB_{\phi(s)} \,ds.
\]
Also write: $E_{\tau}(\phi) $: for $E_{\tau}(\phi) - \frac{\tau
}{2}$. Then
%
%
\begin{equation}
\label{eqnpathint} z(\tau,n) = \mathcal{E}^{\phi(\tau)=n} \bigl[e^{:E_{\tau}(\phi
):}
z_0\bigl(\phi(0)\bigr) \bigr].
\end{equation}
This path integral is essentially the O'Connell--Yor semidiscrete
directed polymer partition function \cite{OY}.

\subsection{Semidiscrete limit of \mbox{$q$-}TASEP dynamics}\label{seclimitdyn}
We now show how \mbox{$q$-}\break TASEP rescales to the semidiscrete SHE. We state
the result for step initial data and then provide a scaling argument
which makes clear the correspondence for general initial data. For the
below proposition,\vspace*{1pt} let $C([0,T],\mathbb{R}^N)$ represent the space of
functions from $[0,T]$ to $\mathbb{R}^N$ endowed with the topology of uniform
convergence on compact subsets.

%
\begin{proposition}\label{propqlimitthm}
Consider \mbox{$q$-}TASEP started from step initial data and scaled according to
%
%
\begin{eqnarray}\label{eqnsemiscalings}
q&=&e^{-\varepsilon},\qquad a_i = e^{-\varepsilon\tilde a_i},
\qquad t=\varepsilon^{-2}\tau,
\nonumber\\[-8pt]\\[-8pt]
x_n(t) &=&\varepsilon^{-2}\tau- (n-1)\varepsilon^{-1}\log\varepsilon
^{-1} -\varepsilon^{-1} F^{n}_{\varepsilon}(\tau).\nonumber
\end{eqnarray}
Let $z_{\varepsilon}(\tau,n) = \exp (-\frac{3\tau}{2} +
F^{n}_{\varepsilon
}(\tau) )$. Then for any $N\geq1$, $T>0$, as $\varepsilon\to0$, the
law of the stochastic process $ \{z_{\varepsilon}(\tau,n)\dvtx \tau\in[0,T],
1\leq n\leq N \}$ converges\vspace*{1pt} in the topology of measures on
$C([0,T],\mathbb{R}^N)$ to a limit given by the law of $ \{z(\tau,n)\dvtx \tau\in[0,T], 1\leq n\leq N \}$ where $z(\tau,n)$ solves
the semidiscrete SHE with drift vector $\tilde a=(\tilde a_1,\ldots,
\tilde a_N)$ and initial data $z_0(n)=\delta_{n=1}$.
\end{proposition}

This result is a corollary of \cite{BorCor} Theorem 4.1.26 which deals
with a larger two-dimensional extension of \mbox{$q$-}TASEP and its limit.
That result is not entirely elementary as it relies upon the
convergence of $q$-Whittaker processes to Whittaker processes \cite
{BorCor} as well as the relationship of Whittaker processes to the
semidiscrete directed polymer \cite{OCon}. We will presently provide a
purely probabilistic sketch of why this result is true, without making
any attempt to fill in the details of rigorous justifications.

It is easy to check the initial data. Observe that via the scalings,
$z_{\varepsilon}(0,n) = \varepsilon^{n-1} e^{\varepsilon n}$. Hence, if $n>1$
the limit is 0,
whereas for $n=1$ the limit is~1. This shows\vadjust{\goodbreak} that as $\varepsilon\to0$,
$z_{\varepsilon
}(0,n)\to\delta_{n=1}$. To\vspace*{1pt} achieve a general initial data $z_0$, one
should scale $x_n(0)$ so that $\varepsilon^{n-1}e^{-\varepsilon x_n(0)}\to z_0(n)$.

To see how the dynamics behave under scaling, it is easiest to work in
terms of~$F^n_{\varepsilon}(\tau)$. Observe that
\begin{eqnarray*}
d F^n_{\varepsilon}(\tau) &=& F^n_{\varepsilon}(
\tau) - F^n_{\varepsilon
}(\tau- d\tau)
\\
&=& \bigl(\varepsilon^{-1}\tau-(n-1)\log\varepsilon^{-1} -\varepsilon
x_n\bigl(\varepsilon^{-2}\tau \bigr) \bigr)
\\
&&{} - \bigl(
\varepsilon^{-1}(\tau-d\tau) -(n-1)\log\varepsilon ^{-1} -\varepsilon
x_n\bigl(\varepsilon ^{-2}\tau- \varepsilon^{-2} \,d
\tau\bigr) \bigr)
\\
&=& \varepsilon^{-1} \,d\tau- \varepsilon \bigl( x_n\bigl(
\varepsilon^{-2}\tau\bigr) - x_n\bigl(\varepsilon^{-2}
\tau- \varepsilon^{-2} \,d\tau\bigr) \bigr).
\end{eqnarray*}
The jump rate for \mbox{$q$-}TASEP, in the rescaled variables, is given by
\[
a_n \bigl(1-q^{x_{n-1}(t) - x_{n}(t) -1}\bigr) = 1 - \varepsilon \bigl(\tilde a_n + e^{F^{n-1}_{\varepsilon}(\tau) - F^{n}_{\varepsilon}(\tau)} \bigr) + O\bigl(\varepsilon^2\bigr).
\]
This means that in an increment of time $\varepsilon^{-2} \,d\tau$, we should
see that
\[
\varepsilon \bigl( x_n\bigl(\varepsilon^{-2}\tau\bigr) -
x_n\bigl(\varepsilon^{-2}\tau- \varepsilon^{-2} \,d\tau
\bigr) \bigr) = \varepsilon^{-1} - \bigl(\tilde a_n +
e^{F^{n-1}_{\varepsilon
}(\tau) -
F^{n}_{\varepsilon}(\tau)} \bigr)\,d\tau+dW_n +o(1),
\]
where the $W_n$ are independent Brownian motions which arise from the
approximation of a Poisson process by a Brownian motion. Setting $B_n
=\tilde a_n - W_n$ (a~Brownian motion with drift $\tilde a_n$ now), we
find that
\[
d F^n_{\varepsilon}(\tau) = e^{F^{n-1}_{\varepsilon}(\tau) -
F^{n}_{\varepsilon}(\tau)} + dB_n +
o(1).
\]

By It\^o's lemma,
\[
d \exp \bigl(F^n_{\varepsilon}(\tau) \bigr) = \bigl(
\tfrac{1}{2} \exp \bigl(F^{n}_{\varepsilon}(\tau) \bigr) +
\exp \bigl(F^{n-1}_{\varepsilon
}(\tau) \bigr) \bigr)\,d\tau+ \exp
\bigl(F^{n}_{\varepsilon}(\tau) \bigr) \,dB_n + o(1)
\]
and hence rewriting this in terms of $z_{\varepsilon}(\tau,n)$ we have
\[
d z_{\varepsilon}(\tau,n) = \nabla z_{\varepsilon}(\tau,n) \,d\tau+
z_{\varepsilon}(\tau,n) \,dB_n + o(1).
\]
As $\varepsilon\to0$, this equation limits to that for $z(\tau,n)$ as desired.

\subsection{Semidiscrete limit of \mbox{$q$-}TASEP duality}\label{secrepappr}
By utilizing the path integral formulation of $z(\tau,n)$ given in
(\ref{eqnpathint}) let us compute expressions for joint moments of
$z(\tau,n)$ for fixed $\tau$ but different values of $n$. For
simplicity, we assume below that all $\tilde{a}_i\equiv0$, though the
general case is no more difficult. This procedure is sometimes called
the \textit{replica approach} (not to be confused with the \textit{replica
trick}---see Section~\ref{secreptrick}) as it involves replication
of the path measure.

Observe that
\[
\mathbb{E} \Biggl[\prod_{i=1}^{k} z(
\tau,n_i) \Biggr] = \mathbb {E} \Biggl[\prod
_{i=1}^{k} \mathcal{E}^{\phi_i(\tau)=n_i}
\bigl[e^{:E_{\tau}(\phi
_i):} z_0\bigl(\phi_i(0)\bigr) \bigr]
\Biggr],
\]
where the $\phi_i$'s are independent copies of the Poisson jump
process $\phi$. Interchanging the disorder and path expectations, we
are left to evaluate the (now inner) expectation
\[
\mathbb{E} \Biggl[\prod_{i=1}^{k}
e^{:E_{\tau}(\phi_i):} \Biggr] = \exp \Biggl(\int_0^{\tau}
\sum_{i<j}^k \delta_{\phi_i(s)=\phi_j(s)} \,ds
\Biggr).
\]
This leads to the final formula
%
%
\begin{eqnarray}\label{eqnidreplica}
&& \mathbb{E} \Biggl[\prod_{i=1}^{k}
z(\tau,n_i) \Biggr]
\nonumber\\[-8pt]\\[-8pt]
&&\qquad = \mathcal {E}^{\phi
_1(\tau)=n_1}\cdots
\mathcal{E}^{\phi_k(\tau)=n_k} \Biggl[\exp \Biggl(\int_0^{\tau}
\sum_{i<j}^k \delta_{\phi_i(s)=\phi_j(s)} \,ds
\Biggr) \prod_{i=1}^{k} z_0
\bigl(\phi_i(0)\bigr) \Biggr].\hspace*{-26pt}\nonumber
\end{eqnarray}
This identity should be thought of as a duality between the
semidiscrete SHE and a system of Poisson jump processes energetically
rewarded via the sum of their local times. The proof of the above
identity follows from the simple fact that for $X$ distributed as a
centered normal random variable with variance $\sigma^2$,
\[
\mathbb{E} \bigl[ e^{k(X-\sigma^2/2)} \bigr] = e^{\sigma^2 k(k-1)/2}.
\]
This implies that it is the Gaussian nature of the noise and not the
underlying generator $\nabla$ which is behind this identity.
Therefore, if $\nabla$ is replaced in Definition~\ref
{def:semidiscSHE} by an arbitrary generator $L$, the same identity
holds if $\phi$ is defined via the adjoint generator of $L$. For more
on these generalities, see Section~6 of \cite{BorCor}. Note that for
the continuum SHE, there exist other types of noise for which dualities
have been shown (see, e.g., \cite{HobsonTribe}).

Just as for \mbox{$q$-}TASEP, (\ref{eqnidreplica}) implies that the joint
moments of $z$ satisfy systems of ODEs (recall Proposition~\ref
{propsystemsODEqTASEP}). The (A) system follows from (\ref
{eqnidreplica}) directly. We now record the limiting versions of
Proposition~\ref{propsystemsODEqTASEP}(B) and (C).

%
\begin{proposition}\label{propsemiODEs} Let $z(\tau;n)$ be as above
with initial data $z_0(n)$ supported on $\mathbb{Z}_{>0}$.
\begin{longlist}[(A)]
\item[(B)] \emph{Free evolution equation with $k-1$ boundary
conditions}: If $\tilde u\dvtx \mathbb{R}_{+}\times(\mathbb{Z}_{\geq
0})^k \to\mathbb{R}$ solves:
\begin{enumerate}[(4)]
\item[(1)] For all $\vec{n}\in(\mathbb{Z}_{\geq0})^k$ and $\tau\in
\mathbb{R}_{+}$,
\[
\frac{d}{d\tau} \tilde u(\tau;\vec{n}) = \sum_{i=1}^{k}
\nabla_i \tilde u (\tau;\vec{n});
\]
\item[(2)] For all $\vec{n}\in(\mathbb{Z}_{\geq0})^k$ such that for some
$i\in\{
1,\ldots, k-1\}$, $n_i=n_{i+1}$,
\[
(\nabla_i - \nabla_{i+1} - 1) \tilde u(\tau;\vec{n}) = 0;
\]
\item[(3)] For\vspace*{2pt} all $\vec{n}\in(\mathbb{Z}_{\geq0})^k$ such that $n_k=0$,
$\tilde u(\tau;\vec{n}) \equiv0$ for all $\tau\in\mathbb{R}_{+}$;
\item[(4)] For all $\vec{n}\in W^{k}_{>0}$, $\tilde u(0;\vec{n}) = \prod_{i=1}^{k} z_0(n_i)$.\vadjust{\goodbreak}
\end{enumerate}
Then for all $\vec{n}\in W^{k}_{>0}$, $\mathbb{E} [\prod_{i=1}^{k}
z(\tau,n_i)  ] = \tilde u(\tau;\vec{n})$.

\item[(C)] \emph{Schr\"{o}dinger equation with Bosonic Hamiltonian}:
If $\tilde v\dvtx \mathbb{R}_{+}\times(\mathbb{Z}_{\geq0})^k$ solves:
\begin{enumerate}[(4)]
\item[(1)] For all $\vec{n}\in(\mathbb{Z}_{\geq0})^k$ and $\tau\in
\mathbb{R}_{+}$,
\[
\frac{d}{d\tau} \tilde v(\tau;\vec{n}) = \widetilde H \tilde v(\tau;\vec{n}),
\qquad\widetilde H = \Biggl[\sum_{i=1}^{k}
\nabla_i + \sum_{i<j}^k
\delta_{n_i=n_j} \Biggr];
\]
\item[(2)] For all permutations of indices $\sigma\in S_k$, $\tilde v(\tau;\sigma\vec{n}) = \tilde v(\tau;\vec{n})$;
\item[(3)] For\vspace*{1.5pt} all $\vec{n}\in(\mathbb{Z}_{\geq0})^k$ such that $n_k=0$,
$\tilde v(\tau;\vec{n}) \equiv0$ for all $\tau\in\mathbb{R}_{+}$;
\item[(4)] For all $\vec{n}\in W^{k}_{>0}$, $\tilde v(0;\vec{n}) = \prod_{i=1}^{k} z_0(n_i)$.
\end{enumerate}
Then for all $\vec{n}\in W^{k}_{>0}$, $\mathbb{E} [\prod_{i=1}^{k}
z(\tau,n_i)  ] = \tilde v(\tau;\vec{n})$.
\end{longlist}
\end{proposition}

These systems of ODEs can be proved from (\ref{eqnidreplica})
directly. Instead, we sketch how they arise as limits of the analogous
ODEs for \mbox{$q$-}TASEP.

Let us first consider (B). Recall that $u(t;\vec{n}) = \mathbb
{E}
[\prod_{i=1}^{k} q^{x_{n_i}(t) +n_i} ]$. Thus, defining
\[
\tilde u_{\varepsilon}(\tau,\vec{n}) = \prod_{i=1}^{k}
e^{\varepsilon
^{-1}\tau} \varepsilon ^{(n_i-1)} u\bigl(\varepsilon^{-2}\tau,
\vec{n}\bigr),
\]
we expect (from Section~\ref{seclimitdyn}) that
\[
\lim_{\varepsilon\to0} e^{-(3k\tau)/2}\tilde u_{\varepsilon
}(\tau,
\vec{n}) = \mathbb{E} \Biggl[\prod_{i=1}^k
z(\tau,n_i) \Biggr].
\]
Call this limit $\tilde{u}(\tau,\vec{n})$. We now check that $\tilde{u}$ indeed satisfies conditions (B.1)--(B.4) above. The fact that it
satisfies (B.3) and (B.4) is clear. Note that
%
%
\begin{equation}
\label{eqntildeunabla} \prod_{i=1}^{k}
e^{\varepsilon^{-1}\tau} \varepsilon^{(n_i-1)} \nabla_i u\bigl(\varepsilon
^{-2}\tau,\vec{n}\bigr) = \varepsilon\tilde u_{\varepsilon}\bigl(\tau,
\vec{n}_{i}^{-}\bigr) -\tilde u_{\varepsilon}(\tau,
\vec{n}).
\end{equation}
Using this, it follows by rescaling (B.1) of Proposition~\ref
{propsystemsODEqTASEP} that
\[
\frac{d}{d\tau} \tilde u_{\varepsilon}(\tau,\vec{n}) = k\varepsilon
^{-1} \tilde u_{\varepsilon}(\tau,\vec{n}) + \biggl(
\varepsilon^{-1}-\frac{1}{2}\biggr) \sum
_{i=1}^{k} \bigl(\varepsilon\tilde u_{\varepsilon}
\bigl(\tau,\vec{n}_{i}^{-}\bigr) - \tilde u_{\varepsilon}(
\tau,\vec{n}) \bigr) + O(\varepsilon).
\]
The factor $\varepsilon^{-1}-\frac{1}{2}$ comes from the expansion of
$\varepsilon
^{-2}(1-q)$. The above can be rewritten as
\[
\frac{d}{d\tau} \tilde u_{\varepsilon}(\tau,\vec{n}) = \sum
_{i=1}^{k} \biggl(\tilde u_{\varepsilon}\bigl(
\tau,\vec{n}_{i}^{-}\bigr) + \frac{1}{2} \tilde u_{\varepsilon}(\tau,\vec{n}) \biggr) + O(\varepsilon),
\]
which in turn implies that
\[
\frac{d}{d\tau} e^{-(3k\tau)/2} \tilde u_{\varepsilon}(\tau,\vec {n}) =
\sum_{i=1}^{k} \nabla_i e^{-(3k\tau)/2}\tilde u_{\varepsilon
}(\tau,\vec{n}) +O(\varepsilon).
\]
This shows that in the $\varepsilon\to0$ limit, $\tilde{u}$ satisfies
(B.1) above.

Using (\ref{eqntildeunabla}) and the expansion $q=1-\varepsilon+
O(\varepsilon^2)$,
it follows from (B.2) of Proposition~\ref{propsystemsODEqTASEP} that
\[
\tilde u_{\varepsilon}\bigl(\tau, \vec{n}_{i}^{-}\bigr)
= \tilde u_{\varepsilon
}\bigl(\tau,\vec {n}_{i+1}^{-}
\bigr) + \tilde u_{\varepsilon}(\tau,\vec{n}) + O(\varepsilon).
\]
Multiplying by $e^{-(3k\tau)/2}$ has no effect on this
equality, and so in the limit $\varepsilon\to0$, we find that $\tilde{u}$
satisfies (B.2).

We now consider (C). Define $\tilde v_{\varepsilon}$ and $\tilde v$ analogously
to $\tilde u_{\varepsilon}$ and $\tilde u$ above. The fact that $\tilde v$
satisfies (C.2), (C.3) and (C.4) is clear. Using (\ref
{eqntildeunabla}) and second-order expansions of $(1-q)$ and
$(1-q^{-1})$, we find that
\[
\frac{d}{d\tau} \tilde v_{\varepsilon}(\tau,\vec{n}) = \sum
_{i=1}^{k} \biggl( \tilde v_{\varepsilon}\bigl(
\tau,\vec{n}_{i}^{-}\bigr) + \frac{1}{2} \tilde v_{\varepsilon
}(\tau,\vec{n}) \biggr) + \sum_{i<j}
\delta_{n_i=n_j} \tilde v_{\varepsilon
}(\tau,\vec{n}) + O(\varepsilon).
\]
Multiplying by $e^{-(3k\tau)/2}$ and taking $\varepsilon\to0$ leads
to (C.1) as desired.

For $z_0(\vec{n})=\prod_{i=1}^{k} \delta_{n_i=1}$ initial data, it
is possible to explicitly solve (B)~and (C) in Proposition~\ref
{propsemiODEs} via nested contour integral formulas which arise as
scaling limits (\ref{FdefqTASEP}). In fact, if we change the boundary
condition in (B.2) to $(\nabla_i - \nabla_{i+1} - c)$ for any $c\in
\mathbb{R}$ [or analogously put this $c$ factor in (C.1) in front of
the sum
over $i<j$] essentially the same integral formulas work and we find
that (B) is solved by
%
%
\begin{equation}
\label{zmoments} \tilde{u}(\tau,\vec{n}) = \frac{e^{-k\tau}}{(2\pi\iota)^k} \int \cdots\int
\prod_{1\leq A<B\leq k} \frac{w_A-w_B}{w_A-w_B-c} \prod
_{j=1}^{k} \frac{e^{tw_j}}{w_j^{n_j}} \,dw_j,
\end{equation}
where the integration contour for $w_A$ contains 0 and $\{w_B+c\}
_{B>A}$. These systems of ODEs are semidiscrete versions of the delta
Bose gas, and $c$ plays the role of the coupling constant. This
remarkable symmetry between attractive ($c>0$) and repulsive ($c<0$)
systems is discussed more in Section~6 of \cite{BorCor}.

\subsection{Semidiscrete limit of \mbox{$q$-}TASEP Fredholm determinant}

Proposition~\ref{propqlimitthm} implies that as $q\to1$, under
proper scaling \mbox{$q$-}TASEP converges to the solution of the semidiscrete
SHE. From this weak convergence result, it follows that the
$e_q$-Laplace transform of particle location for \mbox{$q$-}TASEP converges to
the Laplace transform of the limiting SHE. This Laplace transform
completely characterizes the one-point distribution of the solution
$z(\tau,n)$. The \mbox{$q$-}TASEP Mellin--Barnes-type Fredholm determinant
formula has a nice scaling limit, and thus yields (we will state it for
a zero drift vector) the following.

%
\begin{theorem}\label{OConFredDet}
For $\tau\in\mathbb{R}_{+}$, and $n\geq1$, the solution of the SHE with
delta initial data and drift vector $\tilde a=(0,\ldots, 0)$ is
characterized by (for $\operatorname{Re}u \geq0$):
\[
\mathbb{E} \bigl[e^{-u e^{(3 \tau)/2} z(\tau,n)} \bigr] = \det(I+ K_{u}),
\]
where $\det(I+ K_{u})$ is the Fredholm determinant of
$
K_{u}\dvtx  L^2(C_{0})\to L^2(C_{0})
$
for $C_{0}$ a positively oriented contour containing zero and such that
for all $v,v'\in C_{0}$, we have $|v-v'|<1/2$. The operator $K_u$ is
defined in terms of its integral kernel
\[
K_{u}\bigl(v,v'\bigr) = \frac{1}{2\pi\iota}\int
_{-\iota\infty+ 1/2}^{\iota
\infty+1/2}\,ds \Gamma(-s)\Gamma(1+s)
\frac{\Gamma(v)^n}{\Gamma
(s+v)^n} \frac{ u^s e^{vt s+t s^2/2}}{v+s-v'}.
\]
\end{theorem}
\begin{pf}
This is proved in \cite{BorCor}, Theorem 5.2.10. An alternative choice
of contours is developed in \cite{BorCorFer}, Theorem 1.16. The formula
follows from rigorous asymptotic analysis of Theorem~\ref{PlancherelfredThm}.
\end{pf}

\subsection{The replica trick}\label{secreptrick}

It is enticing to think that one might be able to compute the Laplace
transform formula in Theorem~\ref{OConFredDet} directly from the
explicit formula for $\mathbb{E}[z(\tau,n)^k]$ [such as the one given by
combining (\ref{zmoments}) with Proposition~\ref{propsemiODEs}(B)].
If $X$ is a suitably nice nonnegative random variable (e.g., if
$X$ were bounded), then for $u$ with $\operatorname{Re}(u)>0$,
\[
\mathbb{E} \bigl[ e^{-u X} \bigr] = \sum_{k=0}^{\infty}
\frac{(-u)^k}{k!} \mathbb{E}\bigl[X^k\bigr].
\]
This identity only makes sense if one can rigorously justify
interchanging the summation. Yet worse, if the moments of $X$ grow too
rapidly, the right-hand side might not even be convergent for any value
of $u$ even though the left-hand side would be necessarily finite. This
is exactly the case when $X= e^{(3\tau)/2}z(\tau,n)$. From
(\ref{zmoments}), one can estimate that for this choice of $X$,
$\mathbb{E}
[X^k] \approx e^{c_k k^2}$ where $c_k>c>0$ for all $k$. This means
that, from a mathematical perspective, one cannot use this approach to
compute the Laplace transform.

One variation of the so-called replica trick discussed in physics
literature is an attempt to sum this divergent series in such a way as
to guess the Laplace transform. (In fact, the most typical version of
the replica trick asks for less than the Laplace transform, rather just
for $\mathbb{E}[\log z(\tau,n)]$, and tries to access it from analytically
continuing formulas for integer moments to $k=0$.)

This replica trick procedure has been implemented for the continuum SHE
(a~scaling limit of the semidiscrete SHE) in which the ODEs in
Proposition~\ref{propsemiODEs}(B)~and~(C) become two equivalent
forms of the attractive quantum delta Bose gas. The moments of the
solutions of the continuum SHE grow even faster, like $e^{c_k k^3}$ for
$c_k>c>0$. By diagonalizing the Bosonic Hamiltonian [the limit of (C)]
via the Bethe ansatz, \cite{Dot,CDR} both made initial attempts at
computing the Laplace transform via the replica trick. These initial
attempts yielded a wrong answer. However, very soon afterward, the
formula of \cite{ACQ,SaSp1} was posted (with a rigorous proof given in
\cite{ACQ}) and \cite{Dot,CDR} showed that their approach was able to
recover the correct Laplace transform formula.


\section{Combinatorics}\label{sec7}
\subsection{Useful $q$-deformations}\label{secqdef}
We record some $q$-deformations of classical functions and transforms.
Section~10 of \cite{AAR} is a good reference for many of these
definitions and statements. We assume throughout that $|q|<1$. The
classical functions are recovered in the $q\to1$ limit.

The $q$-\textit{Pochhammer symbol} is written as $(a;q)_{n}$ and defined
via the product (infinite convergent product for $n=\infty$)
\begin{eqnarray*}
(a;q)_{n}&=&(1-a) (1-aq) \bigl(1-aq^2\bigr)\cdots
\bigl(1-aq^{n-1}\bigr),
\\
(a;q)_{\infty}&=&(1-a) (1-aq)
\bigl(1-aq^2\bigr)\cdots.
\end{eqnarray*}

There are two different $q$-\textit{exponential functions} which were
introduced by Hahn \cite{Hahn} in 1949. The first (which we will use)
is denoted $e_q(x)$ and defined as
\[
e_q(x) = \frac{1}{ ((1-q)x;q )_{\infty}},
\]
while the second is defined as
\[
E_q(x) = \bigl(-(1-q)x;q \bigr)_{\infty}.
\]

Both $e_q(x)$ and $E_q(x)$ converge to $e^{x}$ as $q\to1$, cf. (\ref
{qLaplace}) below. In fact, $e_q(x)$ converges uniformly to $e^x$ on
$x\in[-\infty,a]$ for any $a\in\mathbb{R}$.

The $q$-\textit{factorial} is written as either $[n]_{q}!$ 
or just $n_q!$ and is defined as
\[
n_q! = \frac{(q;q)_n}{(1-q)^n} = \frac{(1-q)(1-q^2)\cdots
(1-q^n)}{(1-q)(1-q)\cdots(1-q)}.
\]
The $q$-\textit{binomial coefficients} are defined in terms of
$q$-factorials as
\[
\pmatrix{n\cr k}_q = \frac{n_q!}{k_q!(n-k)_q!} = \frac
{(q;q)_{n}}{(q;q)_{k}(q;q)_{n-k}}.
\]

We also have \cite{KC}
%
%
\begin{equation}
\label{qBinexpansion} \pmatrix{n\cr k}_q = \mathop{\sum
_{S\subset\{1,\ldots, n\}}}_{|s|=k} q^{\|S\|-(k(k+1)/2)},
\end{equation}
where
\[
\|S\|= \sum_{i\in S} i.
\]

The $q$-\textit{binomial theorem} (\cite{AAR}, Theorem 10.2.1) says that
for all $|x|<1$ and \mbox{$|q|<1$},
\[
\sum_{k=0}^{\infty} \frac{(a;q)_k}{(q;q)_k}
x^k = \frac
{(ax;q)_{\infty}}{(x;q)_{\infty}}.
\]
Two corollaries of this theorem (\cite{AAR}, Corollary 10.2.2a/b) are
that under the same hypothesis on $x$ and $q$,
%
%
\begin{equation}
\label{qLaplace} \sum_{k=0}^{\infty}
\frac{x^k}{k_q!} = e_q(x), \qquad\qquad\sum
_{k=0}^{\infty} \frac{q^{k(k-1)/2} (-x)^k}{k_q!} =E_q(x).
\end{equation}
For any $x$ and $q$, we also have (\cite{AAR}, Corollary 10.2.2.c)
%
%
\begin{equation}
\label{finqBinExp} \sum_{k=0}^{n} \pmatrix{n\cr k}_{q} (-1)^k q^{k(k-1)/2}x^k =
(x;q)_n.
\end{equation}





Define the following transform of a function $f\in\ell^1(\mathbb
{Z}_{\geq0})$:
%
%
\begin{equation}
\label{qlaplacedef} \hat{f}^{q}(\zeta):= \sum
_{n\geq0} \frac{f(n)}{(\zeta
q^n;q)_{\infty}},
\end{equation}
where $\zeta\in\mathbb{C}$.

We call this the $e_q$-Laplace transform of $q^X$ since if $X$ is a
random variable taking values in $\mathbb{Z}_{\geq0}$ and
$f(n)=\mathbb{P}(X=n)$,
\[
\hat{f}^{q}(\zeta) = \mathbb{E} \biggl[e_q \biggl(
\frac{\zeta
q^X}{1-q} \biggr) \biggr].
\]

An inversion formula is given as Proposition 3.1.1 of \cite{BorCor}
and can also be found in \cite{Banger}.
%
\begin{proposition}\label{qlaplaceinverse}
One may recover the function $f\in\ell^1(\mathbb{Z}_{\geq0})$ from its
transform $\hat{f}^{q}(\zeta)$ with $\zeta\in\mathbb
{C}\setminus\{q^{-k}\}
_{k\geq0}$ via the inversion formula
%
%
\begin{equation}
\label{qlaplaceinverseEQN} f(m) = -q^{m} \frac{1}{2\pi\iota} \int
_{C_m} \bigl(q^{m+1}\zeta;q\bigr)_{\infty}
\hat{f}^{q}(\zeta) \,d\zeta,
\end{equation}
where $C_m$ is any positively oriented contour which encircles $\zeta
=q^{-M}$ for $0\leq M \leq m$.
\end{proposition}

\subsection{Symmetrization identities}
We state and prove the following two useful symmetrization identities.

%
\begin{lemma}\label{combiden}
For all $k\geq1$
%
%
\begin{eqnarray}\label{combiden1}
&& \sum_{\sigma\in S_k} \prod
_{1\leq A<B\leq k} \frac{z_{\sigma
(A)}-z_{\sigma(B)}}{z_{\sigma(A)}-\tau z_{\sigma(B)}}
\nonumber\\[-8pt]\\[-8pt]
&&\qquad  = (\tau;\tau )_k \tau^{-k(k-1)/2} z_1\cdots z_k \det \biggl[
\frac{1}{z_i-\tau z_j} \biggr]_{i,j=1}^k.\nonumber
\end{eqnarray}
Setting $\xi_i = \frac{1+z_i}{1+z_i/\tau}$ we also have
%
%
\begin{eqnarray}\label{combiden2}
&& \sum_{\sigma\in S_k} \prod
_{1\leq A<B\leq k} \frac{z_{\sigma
(A)}-z_{\sigma(B)}}{z_{\sigma(A)}-\tau z_{\sigma(B)}} \prod_{i=1}^k
\frac{1}{\xi_{\sigma(1)}\cdots\xi_{\sigma(i)}-1}
\nonumber\\[-8pt]\\[-8pt]
&&\qquad = (-1)^k \tau ^{-k(k-1)/2} \det \biggl[
\frac{1}{z_i-\tau z_j} \biggr]_{i,j=1}^k \prod
_{i=1}^{k} (\tau+ z_i).\nonumber
\end{eqnarray}
\end{lemma}

%
\begin{remark}\label{xirem}
Before proving these identities note that for $\xi_i = \frac
{1+z_i}{1+z_i/\tau}$ and $\tau=p/q$,
%
%
\begin{eqnarray}\label{ratio1}
\frac{z_i-z_j}{z_i-\tau z_j} &=& q \frac{\xi_i-\xi_j}{p+q\xi_i\xi
_j - \xi_j},\nonumber
\\
-\frac{z_i(p-q)^2}{(z_i+1)(p+qz_i)} &=& p\xi_i^{-1} +q\xi
_i -1,
\\
\tau+ z_i &=& \frac{\tau-1}{1-\xi_i/\tau}.\nonumber
\end{eqnarray}
\end{remark}

\begin{pf*}{Proof of Lemma~\ref{combiden}}
The first identity (\ref{combiden1}) is \cite{M}, Chapter~III, equation~(1.4). The
second identity is equivalent to the identity (1.7) in \cite{TW1}. In
order to see this equivalence expand the Cauchy determinant as
\[
\det \biggl[\frac{1}{z_i-\tau z_j} \biggr]_{i,j=1}^k =
\frac{\tau
^{k(k-1)/2}}{z_1\cdots z_k (1-\tau)^k} \prod_{1\leq i\neq
j\leq k} \frac{ z_i-z_j}{z_i-\tau z_j}.
\]
Multiply both sides of the claimed identity by the factor $\prod_{1\leq i\neq j\leq k} \frac{z_i-\tau z_j}{z_i-z_j}$, reducing the
identity to
\[
\sum_{\sigma\in S_k} \prod_{1\leq A<B\leq k}
\frac{z_{\sigma
(B)}-\tau z_{\sigma(A)}}{z_{\sigma(B)}-z_{\sigma(A)}} \prod_{i=1}^k
\frac{1}{\xi_{\sigma(1)}\cdots\xi_{\sigma(i)}-1} = \prod_{i=1}^k
\frac{- (\tau+z_i)}{z_i (1-\tau)}.\vadjust{\goodbreak}
\]
Noting that $\frac{-(\tau+z_i)}{z_i (1-\tau)} = (\xi_i-1)^{-1}$ and
using the relation (\ref{ratio1}), it remains to prove that
\begin{eqnarray*}
&& \sum_{\sigma\in S_k} q^{-k(k-1)/2} \prod
_{1\leq A<B\leq k} \frac{p+q+\xi_{\sigma(B)}\xi_{\sigma(A)}-\xi_{\sigma(A)}}{\xi
_{\sigma(B)}-\xi_{\sigma(A)}} \prod_{i=1}^k
\frac{1}{\xi_{\sigma
(1)}\cdots\xi_{\sigma(i)}-1}
\\
&&\qquad  = \prod_{i=1}^k (
\xi_i -1)^{-1}.
\end{eqnarray*}
Using the antisymmetry of the Vandermonde determinant, we rewrite the
above as
\begin{eqnarray*}
&&\sum_{\sigma\in S_k} \operatorname{sgn}(\sigma) \prod
_{1\leq
A<B\leq k} (p+q\xi _{\sigma(B)}\xi_{\sigma(A)}-
\xi_{\sigma(A)})\prod_{i=1}^k
\frac
{1}{\xi_{\sigma(1)}\cdots\xi_{\sigma(i)}-1}
\\
&&\qquad  = q^{k(k-1)/2}\frac{\prod_{A<B}(\xi_{B}-\xi_{A})}{\prod_{i=1}^k (\xi
_i -1)}.
\end{eqnarray*}
The above identity is (1.7) in \cite{TW1}, and the proof is complete.
\end{pf*}

\subsection{Defining a Fredholm determinant}
Fix a Hilbert space $L^2(X,\mu)$ where $X$ is a measure space and $\mu
$ is a measure on $X$. When $X=\Gamma$, a simple smooth contour in
$\mathbb{C}
$, we write $L^2(\Gamma)$ where $\mu$ is understood to be the path
measure along $\Gamma$ divided by $2\pi\iota$. When $X$ is the
product of a discrete set $D$ and a contour $\Gamma$, $\mu$ is
understood to be the product of the counting measure on $D$ and the
path measure along $\Gamma$ divided by $2\pi\iota$.

Let $K$ be an \textit{integral operator} acting on $f(\cdot)\in L^2(X,\mu
)$ by $(Kf)(x) = \int_{X} K(x,y)f(y) \,d\mu(y)$. $K(x,y)$ is called the
\textit{kernel} of $K$. A \textit{formal Fredholm determinant expansion} of
$I+K$ is a formal series written as
\[
\det(I+K) = 1+\sum_{n=1}^{\infty}
\frac{1}{n!} \int_{X} \cdots \int
_{X} \det \bigl[K(x_i,x_j)
\bigr]_{i,j=1}^{n} \prod_{i=1}^{n}
\,d\mu(x_i).
\]

If the above series is absolutely convergent, then we call this a \textit{numerical Fredholm determinant expansion} as it actually takes a
numerical value. If $K$ is a \textit{trace-class} operator (see \cite
{Lax} or \cite{Bornemann}), then the expansion is always absolutely
convergent, though it is possible to have operators which are not
trace-class, for which convergence still holds.

\section{Uniqueness of systems of ODEs}\label{appenduniq}
We prove the uniqueness result of Proposition~\ref{asepuniq} by a
probabilistic approach. It is possible to extend this proof to a more
general class of generators, but we do not pursue this here.

\begin{pf*}{Proof of Proposition~\ref{asepuniq}}
Let us first demonstrate the existence of one solution to the system of
ODEs given in Definition~\ref{WASEPtrueevol}. Let $\tilde h^1$ denote
the proposed solution, equation (\ref{propsoln}), in the statement of
the proposition. The definition of the generator implies that $\tilde h^1$ satisfies condition 1 of Definition~\ref{WASEPtrueevol}.

To prove that $\tilde h^1$ satisfies conditions 2 and 3 requires an
estimate. In time $t$, the number of jumps in ASEP is bounded by a
Poisson random variable with parameter\vadjust{\goodbreak} given by constant time $t$. This
means that for some constant $c'>0$
%
%
\begin{equation}
\label{ineq2a} \mathbb{P} \bigl(\bigl\|\vec{x}(-t) -\vec{x}\bigr\|_1 = n
\bigr) \leq e^{-c't} \frac
{(c't)^n}{n!}.
\end{equation}
Observe now that
\begin{eqnarray*}
&& \bigl\llvert \tilde h^1(t;\vec{x}) - \mathbb{P}^{-t;\vec{x}}
\bigl(\vec {x}(0)=\vec {x}\bigr) \tilde h_0(\vec{x})\bigr\rrvert
\\
&&\qquad \leq \sum_{n\geq1} \sum_{\vec
{x}'\dvtx  \|\vec{x}-\vec{x}'\|=n}
\mathbb{P}^{-t;\vec{x}}\bigl(\vec {x}(0)=\vec {x}'\bigr) \bigl|\tilde h_0\bigl(\vec{x}'\bigr)\bigr|
\\
&&\qquad \leq \sum_{n\geq1} \sum
_{\vec{x}'\dvtx  \|\vec{x}-\vec{x}'\|=n} e^{-c' t} \frac{(c't)^n}{n!} C
e^{-c\|\vec{x}'\|_1}
\\
&&\qquad \leq \sum_{n\geq1} \bigl(c''
\bigr)^n e^{-c' t} \frac{(c't)^n}{n!} C e^{-c(\max(0,n-\|\vec{x}\|_1))}
\\
&&\qquad
\leq e^{c\|\vec{x}\|_1} \bigl(e^{c''' t} -1\bigr).
\end{eqnarray*}
The first inequality follows from the definition of $\tilde h^1$ as an
expectation, along with the triangle inequality.
For the second inequality, we can use the bounds (\ref{ineq1a})~and~(\ref{ineq2a}) to replace $\mathbb{P}^{-t;\vec{x}}(\vec{x}(0)=\vec{x}')
|\tilde h_0(\vec{x}')|$ by $e^{-c' t} \frac{(c't)^n}{n!} C
e^{-c\|\vec{x}'\|_1}$. For\vspace*{1pt} the third inequality, we\vspace*{2pt} observe that
$\|\vec{x}'\|_1\geq\max(0,n-\|\vec{x}\|_1)$. Plugging this bound
into $e^{-c\|\vec{x}'\|_1}$, we find that the summand is now
independent of $\vec{x}'$ and the summation over $\vec{x}'$ can be
replaced by a rough combinatorial bound of $(c'')^n$ for the number of
such $\vec{x}'$ ($c''$ is some sufficiently large constant). The
fourth equality comes from factoring out $e^{c\|\vec{x}\|_1}$ from the
summation and then bounding the remaining summation in $n\geq1$ by the
Taylor series for the exponential.

The conclusion of the above line of inequalities is that for some $c'''>0$,
\[
\bigl\llvert \tilde h^1(t;\vec{x}) - \mathbb{P}^{-t;\vec{x}}
\bigl(\vec {x}(0)=\vec {x}\bigr) \tilde h_0(\vec{x})\bigr\rrvert \leq
e^{c\|\vec{x}\|_1} \bigl(e^{c''' t} -1\bigr).
\]
Observe\vspace*{1pt} that using the triangle inequality and the exponential bound on
$\tilde h_0(\vec{x})$, the above inequality implies that $\tilde h^1$
satisfies condition 2. Similarly, as $t\to0$, $\mathbb{P}^{-t;\vec
{x}}(\vec
{x}(0)=\vec{x}) \to1$ and $ e^{c\|\vec{x}\|_1} (e^{c''' t} -1)\to
0$ we obtain the pointwise convergence (condition 3):
\[
\tilde h^1(t;\vec{x}) \to\tilde h_0(\vec{x}), \qquad t
\to0.
\]

The argument to prove uniqueness is very similar to the argument used
to prove condition 3. Assume now that in addition to $\tilde h^1$,
there existed another solution to the true evolution equation which we
will denote by $\tilde h^2$. The idea is to prove that $g:= \tilde h^1-\tilde h^2$ must be identically 0. The solution $g$ has zero
initial data.

To prove that $g\equiv0$ it suffices to show that for any $T>0$ and
any $\vec{x}\in\widetilde W^{k}$, $g(t;\vec{x})=0$ for all $t\in[0,T]$.
Since $\tilde h^1$ and $\tilde h^2$ solve the true evolution equation, so too must their difference $g$. Hence, we readily see that for any
$\delta\in(0,T]$,
\begin{eqnarray*}
&&\Biggl\llvert g(t;\vec{x}) - \sum_{n=0}^{n(T)}
\sum_{\vec{x}'\dvtx  \|\vec
{x}-\vec{x}'\|=n} \mathbb{P}^{-t;\vec{x}}\bigl(\vec{x}(-
\delta)=\vec{x}'\bigr) g\bigl(\delta;\vec{x}'\bigr)
\Biggr\rrvert
\\
&&\qquad \leq \sum_{n>n(T)} \sum
_{\vec{x}'\dvtx  \|\vec{x}-\vec{x}'\|=n} \mathbb{P} ^{-t;\vec{x}}\bigl(\vec{x}(-\delta)=
\vec{x}'\bigr) \bigl|g\bigl(\delta;\vec{x}'\bigr)\bigr|
\\
&&\qquad \leq \sum_{n>n(T)} \bigl(c''
\bigr)^n e^{-c' t} \frac{(c't)^n}{n!} C e^{-c(\max(0,n-\|\vec{x}\|_1))},
\end{eqnarray*}
where $n(T)$ is a positive integer which depends on $T$ and will be
specified soon. The above inequalities follow for similar reasons as in
the proof of condition 3 for $\tilde h^1$. Now observe that by choosing
$n(T)$ sufficiently large, the summation in the last line above can be
made arbitrarily small. This is due to the fact that $1/n!$ decays
super-exponentially. This means that for any $\varepsilon>0$ and any $T>0$,
there exists $n(T)$ such that
\[
\Biggl\llvert g(t;\vec{x}) - \sum_{n=0}^{n(T)}
\sum_{\vec{x}'\dvtx  \|\vec
{x}-\vec{x}'\|=n} \mathbb{P}^{-t;\vec{x}}\bigl(\vec{x}(-
\delta)=\vec{x}'\bigr) g\bigl(\delta;\vec{x}'\bigr)
\Biggr\rrvert \leq\varepsilon.
\]
Since the set of $\vec{x}'$ such that $\|\vec{x}-\vec{x}'\|=n$ with
$n\in\{0,1,\ldots, n(T)\}$ is a finite set, condition 3 implies that
as $\delta\to0$, each $g(\delta;\vec{x}')\to0$ as well. Choosing
$\delta$ sufficiently small, this implies that
\[
\bigl\llvert g(t;\vec{x})\bigr\rrvert \leq2\varepsilon
\]
and since $\varepsilon$ was arbitrary this implies in fact that $g(t;\vec{x})
=0$ for all $t\in[0,T]$. This completes the proof of uniqueness.
\end{pf*}

\section{GUE Tracy--Widom asymptotics for ASEP}\label{GUEasym}
We provide a critical point analysis for the long-time asymptotics of
our Mellin--Barnes-type Fredholm determinant formula for the $e_{\tau
}$-Laplace transform of $\tau^{N_x(t)}$. We assume that $\tau<1$ is
fixed and straightforwardly arrive at the GUE Tracy--Widom limit
theorem recorded in (\ref{ASEPlimitthm}) and proved first (via an
analysis of the Cauchy-type formula) by Tracy and Widom \cite{TW3}. In
order to make this analysis a rigorous one, would need to control the
tails of the integrand defining the kernel. Another scaling limit of
interest is the weakly asymmetric limit in which $\tau$ goes to 1
simultaneously with $t$ going to infinity. Under the correct scaling
(as in \cite{ACQ}), our formula should lead to the Laplace transform
of the Hopf--Cole solution to the KPZ equation with narrow wedge
initial data. We do not pursue these directions presently, but rather
remark that it appears that the Mellin--Barnes-type formula is very
well suited for such a rigorous asymptotic analysis.

For simplicity, let us consider ASEP with step initial data and fix
$x=0$. We seek to study the large $t$ behavior of $N_0(\eta(t))$ via
its $e_{\tau}$-Laplace transform. Let us recall the formula we have
proved in Theorem~\ref{ASEPMellinBarnesThm}:
%
%
\begin{equation}
\label{wehaveproved} \mathbb{E} \bigl[e_{\tau}\bigl(\zeta\tau^{N_x(\eta(t))}
\bigr) \bigr] = \det (I-K_{\zeta}),
\end{equation}
where $\det(I-K_\zeta)$ is the Fredholm determinant of the operator
$K_\zeta\dvtx\break  L^2(C_{0,-\tau;-1,\tau\theta}) \to L^2(C_{0,-\tau;-1,\tau
\theta})$
defined in terms of its integral kernel
\begin{eqnarray*}
&& K_\zeta\bigl(w,w'\bigr)
\\
&&\qquad  = \frac{1}{2\pi\iota} \int
_{D_{R,d}} \Gamma (-s)\Gamma(1+s) \bigl[-(1-\tau)\zeta
\bigr]^{s} \frac{\exp{t(\tau/(\tau+w))}}{\exp{t(\tau/(\tau+\tau^s w))}} \frac{ds}{w'- \tau^s w }.
\end{eqnarray*}

As noted in Section~\ref{secqdef}, $e_{\tau}(z)$ converges uniformly
for $z\in[-\infty,0]$. This means that if $z\to-\infty$ then
$e_{\tau}(z)\to0$ and if $z\to0$ then $e_{\tau}(z)\to1$. On
account of this, if we set
\[
\zeta= \tau^{-(t/4) + t^{1/3} r}
\]
then it follows (cf. \cite{BorCor}, Lemma 4.1.39) that
\begin{eqnarray*}
&&\lim_{t\to\infty}\mathbb{E} \bigl[e_{\tau}\bigl(-\zeta
\tau^{N_0(\eta
(t/\gamma))}\bigr) \bigr]\\
&&\qquad = \lim_{t\to\infty} \mathbb{P}
\biggl(\frac{N_0(\eta(t/\gamma)) - (t/4)}{t^{1/3}} \geq-r \biggr),
\end{eqnarray*}
where we have set $\gamma= q-p$.

Theorem~\ref{ASEPMellinBarnesThm} [see equation (\ref{wehaveproved})
above] reduces this to a problem in asymptotic analysis. We proceed now
without careful estimates and only discuss contours briefly. There are
a few estimates which would be necessary to turn this into a rigorous
proof. Making the change of variables in (\ref{wehaveproved}) $z=\tau
^s w$ and using the fact that $\Gamma(-s)\Gamma(1+s) = \pi/ \sin
(-\pi s)$, we arrive at the following:
\[
\lim_{t\to\infty} \mathbb{P} \biggl(\frac{N_0(\eta(t/\gamma)) -
(t/4)}{t^{1/3}} \geq-r
\biggr) = \lim_{t\to\infty} \det \bigl(I-K'_{\zeta}
\bigr),
\]
where the kernel is now given by
\begin{eqnarray*}
K'_\zeta\bigl(w,w'\bigr) &=&
\frac{1}{2\pi\iota} \frac{1}{\log\tau}
\int \frac{\pi}{\sin(\pi(\log_{\tau} w - \log_{\tau} z))} (1-\tau
)^{\log_\tau z - \log_\tau w}
\\
&&\hspace*{53pt}{}\times\exp \bigl(t\bigl[G(z)-G(w)\bigr]
\\
&&\hspace*{84pt}{} + t^{1/3} \log\tau r [
\log_\tau z - \log_\tau w] \bigr) \frac{1}{z-w'}
\frac{dz}{z}
\end{eqnarray*}
with
\[
G(z) = -\frac{\log z}{4} - \frac{\tau}{\tau+z}.
\]
The critical point of $G(z)$ is readily calculated by solving
\[
G'(z) = -\frac{1}{4z} + \frac{\tau}{(\tau+z)^2} = 0.
\]
This yields $z_c = \tau$ as the critical point. Actually, it is a
double root of the above equation and accordingly one sees that
$G''(z_c)=0$. The fact that the third derivative is nonzero (and the
second derivative is) indicates $t^{1/3}$ scaling. Up to high order
terms in $(z-\tau)$ and $(w-\tau)$, we have
\[
G(z)-G(w) \approx-\frac{ (z-\tau)^3}{48 \tau^3} + \frac{(w-\tau
)^3}{48 \tau^3}.
\]
The $w$ contour can freely be deformed to go through the critical point
$\tau$ and to depart it at angles $\pm\pi/3$ (oriented with
increasing imaginary part). Likewise the $z$ contour can go through
$\tau-t^{1/3}$ and depart at angles $\pm2\pi/3$ (oriented with
decreasing imaginary part---as is a consequence of the change of
variables). The $w$ contour needs to cross the negative real axis
between $-1$ and $-\tau$. The $z$ contour must stay inside the $w$
contour. There is one nuance with the $z$ curve which is that it keeps
wrapping around in a circle (since the imaginary part of $s$ went from
$-\infty$ to $\infty$). However, with a suitable a priori bound one
should be able to show that this $s$ contour can be made finite at a
cost going to 0 as $t$ goes to infinity.

The above considerations suggest scaling into a window of size
$t^{1/3}$ around the critical point $\tau$. Consider the change of
variables $z-\tau= t^{-1/3} \tau\tilde z$ and likewise for $w$ and
$w'$. This leads to
\begin{eqnarray*}
\frac{1}{\log\tau} \frac{ \pi}{\sin(\pi\log_\tau w -\log_\tau
z)} &\approx& t^{1/3}
\frac{1}{\tilde w - \tilde z},
\\
t\bigl[G(z)-G(w)\bigr] &\approx& -\frac{z^3}{48} + \frac{w^3}{48},
\\
t^{1/3} \log\tau r (\log_\tau z - \log\tau w) & \approx& r(
\tilde z - \tilde w),
\\
(1-\tau)^{\log_\tau z -\log_\tau w} &\approx& 1,
\\
\frac{1}{z-w'} &\approx& t^{1/3} \frac{\tau^{-1}}{\tilde z - \tilde w'},
\\
\frac{dz}{z} &\approx& t^{-1/3} \,d\tilde z.
\end{eqnarray*}
Additionally, there is an extra factor of $\tau^{-1} t^{1/3}$ which is
factored into the kernel, due to the Jacobian of the $w$ and $w'$
change of variables. Combining all of these factors, we get that the
kernel has rescaled to
\[
\widetilde K_r\bigl(\tilde w,\tilde w'\bigr) =
\frac{1}{2\pi\iota} \int e^{-(\tilde z^3/48) + (\tilde w^3/48) + r(\tilde z-\tilde w)} \frac
{1}{\tilde w-\tilde z}\frac{d\tilde z}{\tilde z-\tilde w'},
\]
where the $w$ contour is given by two infinite rays departing $1$ at
angles $\pm\pi/3$ (oriented with increasing imaginary part) and the
$z$ contour is given by two infinite rays departing $0$ at angles $\pm
2\pi/3$ (oriented with decreasing imaginary part). Recalling $\tilde z
= 2^{4/3} z$ and likewise for $w$ and $w'$ yields
\[
K_r\bigl(w,w'\bigr) = \frac{1}{2\pi\iota} \int
e^{-(z^3/3) + (w^3/3) + 2^{4/3} r(z-w)} \frac{1}{w-z}\frac{dz}{z-w'}.
\]
The Fredholm determinant with this kernel is readily shown to be
equivalent to the Fredholm determinant of the Airy kernel (see, e.g.,
Lemma 8.6 of \cite{BorCorFer}), and hence its Fredholm
determinant is equal to $F_{\mathrm{GUE}}(2^{4/3} r)$.

This implies that
\[
\lim_{t\to\infty} \mathbb{P} \biggl(\frac{N_0(t/\gamma) - (t/4)}{t^{1/3}} \geq-r
\biggr) = F_{\mathrm{GUE}}\bigl(2^{4/3} r\bigr)
\]
as we desired to show.
\end{appendix}

\section*{Acknowledgements}
We wish to thank P. Ferrari, T. Imamura, J. Quastel, C.~Tracy and H.
Widom for multiple discussions on the subject. Some work on this
subject occurred during the Oberwolfach meeting on Stochastic Analysis,
the Warwick EPSRC Symposium on Probability, the American Institute of
Mathematics workshop on the Kardar--Parisi--Zhang Equation and
Universality Class, and the Institute for Mathematical Sciences, NUS,
workshop on Polymers and Related Topics.



\printaddresses

\end{document}